\documentclass[11 pt]{amsart}
\usepackage{latexsym,amscd,amssymb, graphicx, amsthm, bm, enumerate}  
\usepackage{young}
\usepackage{arydshln}
\usepackage{tikz}
\usepackage{bm, enumerate, algorithm, algorithmic}

\usepackage[margin=1in]{geometry}

\numberwithin{equation}{section}

\newtheorem{theorem}{Theorem}[section]
\newtheorem{proposition}[theorem]{Proposition}
\newtheorem{corollary}[theorem]{Corollary}
\newtheorem{lemma}[theorem]{Lemma}
\newtheorem{conjecture}[theorem]{Conjecture}
\newtheorem{observation}[theorem]{Observation}

\newtheorem{problem}[theorem]{Problem}
\newtheorem{example}[theorem]{Example}
\newtheorem{remark}[theorem]{Remark}

\newtheorem{defn}[theorem]{Definition}
\theoremstyle{definition}

\newcommand{\maj}{{\mathrm {maj}}}
\newcommand{\inv}{{\mathrm {inv}}}
\newcommand{\code}{{\mathrm {code}}}
\newcommand{\sign}{{\mathrm {sign}}}

\newcommand{\iDes}{{\mathrm {iDes}}}

\newcommand{\Stir}{{\mathrm {Stir}}}
\newcommand{\Hilb}{{\mathrm {Hilb}}}
\newcommand{\grFrob}{{\mathrm {grFrob}}}

\newcommand{\codinv}{{\mathrm {codinv}}}
\newcommand{\coinv}{{\mathrm {coinv}}}

\newcommand{\rd}{{\mathrm {read}}}

\newcommand{\CT}{{\mathcal {CT}}}
\newcommand{\Frob}{{\mathrm {Frob}}}

\newcommand{\ann}{{\mathrm {ann}}}

\newcommand{\symm}{{\mathfrak{S}}}

\newcommand{\initial}{{\mathrm{in}}}

\newcommand{\CC}{{\mathbb {C}}}
\newcommand{\QQ}{{\mathbb {Q}}}

\newcommand{\OP}{{\mathcal{OP}}}

\newcommand{\WWW}{{\mathbb{W}}}
\newcommand{\HHH}{{\mathbb{H}}}
\newcommand{\UUU}{{\mathbb{U}}}
\newcommand{\SSS}{{\mathcal{SS}}}
\newcommand{\CCC}{{\mathcal {C}}}
\newcommand{\AAA}{{\mathcal{A}}}
\newcommand{\MMM}{{\mathcal{M}}}
\newcommand{\BBB}{{\mathcal{B}}}
\newcommand{\OSP}{{\mathcal{OSP}}}

\newcommand{\iii}{{\mathbf {i}}}

\newcommand{\xx}{{\mathbf {x}}}

\newcommand{\II}{{\mathbf {I}}}

\begin{document}

\title[Set Superpartitions and Superspace Duality Modules]
{Set Superpartitions and Superspace Duality Modules}

\author{Brendon Rhoades}
\address
{Department of Mathematics \newline \indent
University of California, San Diego \newline \indent
La Jolla, CA, 92093, USA}
\email{bprhoades@ucsd.edu}

\author{Andrew Timothy Wilson}
\address
{Department of Mathematics \newline \indent
Kennesaw State University \newline \indent
Marietta, GA, 30060, USA}
\email{awils342@kennesaw.edu}

\begin{abstract}
The superspace ring $\Omega_n$ is a rank $n$ polynomial ring tensor a rank $n$ exterior algebra.
Using an extension of the Vandermonde determinant to $\Omega_n$, the authors
previously defined a family of doubly graded quotients $\WWW_{n,k}$
of $\Omega_n$ which carry an action of the
symmetric group $\symm_n$ and satisfy a bigraded version of Poincar\'e Duality.
In this paper, we examine the duality modules $\WWW_{n,k}$ in greater detail.
We describe a monomial basis of $\WWW_{n,k}$ and give
combinatorial formulas for its bigraded Hilbert and 
Frobenius series.
These formulas involve new combinatorial objects  called {\em ordered set superpartitions}.
These are ordered set partitions $(B_1 \mid \cdots \mid B_k)$ of $\{1,\dots,n\}$ in which
the non-minimal elements of any block $B_i$ may be barred or unbarred.
\end{abstract}

\maketitle

\section{Introduction}
\label{Introduction}

Let $n$ be a positive integer. {\em Superspace} of rank $n$ (over the ground field $\QQ$)
is the tensor product
\begin{equation}
\Omega_n := \QQ[x_1, \dots, x_n] \otimes \wedge \{ \theta_1, \dots, \theta_n \}
\end{equation}
of a rank $n$ polynomial ring with a rank $n$ exterior algebra. 
The ring $\Omega_n$ carries a `diagonal'
action of the symmetric group $\symm_n$ on $n$ letters, viz.
\begin{equation}
w \cdot x_i := x_{w(i)} \quad \quad w \cdot \theta_i := \theta_{w(i)} \quad \quad w \in \symm_n, \, \,  1 \leq i \leq n
\end{equation}
which turns $\Omega_n$ into a bigraded $\symm_n$-module by considering 
$x$-degree and $\theta$-degree separately.

The ring $\Omega_n$ plays a large role in physics, where the `bosonic' $x_i$ variables 
model the states of bosons and the `fermionic' $\theta_i$ variables model the states of fermions \cite{PS}.
A number of recent papers in algebraic combinatorics consider $\symm_n$-modules
constructed with a mix of commuting and anticommuting variables
\cite{BRT, KR, RW, Zabrocki}. The Fields Institute Combinatorics Group
made the tantalizing conjecture (see \cite{Zabrocki}) that, if 
$\langle (\Omega_n)^{\symm_n}_+ \rangle \subseteq \Omega_n$ denotes the 
ideal generated by $\symm_n$-invariants with vanishing constant term, we have an  $\symm_n$-module
isomorphism
\begin{equation}
\label{fields-conjecture}
\Omega_n/  \langle (\Omega_n)^{\symm_n}_+ \rangle \cong \QQ[\OP_n] \otimes \sign,
\end{equation} 
where $\OP_n$ denotes the family of all ordered set partitions of $[n] := \{1, \dots, n \}$ 
(with its natural permutation action of $\symm_n$) and $\sign$ is the 1-dimensional sign representation of 
$\symm_n$. Despite substantial effort, the conjecture \eqref{fields-conjecture} has remained out of reach;
even proving $\Omega_n/  \langle (\Omega_n)^{\symm_n}_+ \rangle$ has the expected vector space 
dimension remains open.

The {\em Vandermonde determinant} $\delta_n \in \QQ[x_1, \dots, x_n]$ is the polynomial
\begin{equation}
\delta_n := \varepsilon_n \cdot (x_1^{n-1} x_2^{n-2} \cdots x_{n-1}^1 x_n^0)
\end{equation}
where $\varepsilon_n := \sum_{w \in \symm_n} \sign(w) \cdot w \in \QQ[\symm_n]$ is the antisymmetrizing 
element of the symmetric group algebra.
Given integers $k \leq n$, the authors defined \cite{RW} the following extension 
$\delta_{n,k}$ of the Vandermonde to superspace:
\begin{equation}
\delta_{n,k} := \varepsilon_n \cdot (x_1^{k-1} \cdots x_{n-k}^{k-1} x_{n-k+1}^{k-1} 
x_{n-k+2}^{k-2} \cdots x_{n-1}^1 x_n^0 \cdot \theta_1 \cdots \theta_{n-k}).
\end{equation}
When $k = n$, we recover the classical Vandermonde: $\delta_{n,n} = \delta_n$.
The $\delta_{n,k}$ may be used to build bigraded $\symm_n$-stable quotient rings $\WWW_{n,k}$
of $\Omega_n$ as follows.

For $1 \leq i \leq n$, the partial derivative operator $\partial/\partial x_i$ acts naturally on 
the polynomial ring $\QQ[x_1, \dots, x_n]$ and, by treating the $\theta_i$ as constants, 
on the ring $\Omega_n$.  We also have a 
$\QQ[x_1,\dots,x_n]$-linear operator $\partial/\partial \theta_i$ on $\Omega_n$ defined by
\begin{equation}
\partial / \partial \theta_i: \theta_{j_1} \cdots \theta_{j_r} \mapsto \begin{cases}
(-1)^{s-1} \theta_{j_1} \cdots \theta_{j_{s-1}} \theta_{j_{s+1}} \cdots \theta_{j_r} & \text{if $j_s = i$ for some $s$,} \\
0 & \text{otherwise,}
\end{cases}
\end{equation}
where $1 \leq j_1, \dots, j_r \leq n$ are any distinct indices.
\footnote{In differential geometry $\partial/\partial \theta_i$ is called a 
{\em contraction} operator.}
These operators satisfy the defining relations of $\Omega_n$, namely
\begin{equation*}
(\partial/\partial x_i)(\partial/\partial x_j) = (\partial/\partial x_j)(\partial/\partial x_i) \quad 
(\partial/\partial x_i)(\partial/\partial \theta_j) = (\partial/\partial \theta_j)(\partial/\partial x_i) 
\end{equation*}
\begin{equation*}
(\partial/\partial \theta_i)(\partial/\partial \theta_j) = -(\partial/\partial \theta_j)(\partial/\partial \theta_i) 
\end{equation*}
for all $1 \leq i, j \leq n$.  Given $f \in \Omega_n$, we therefore have a well-defined operator 
$\partial f$ on $\Omega_n$ given by replacing each $x_i$ in $f$ with $\partial/\partial x_i$
and each $\theta_i$ in $f$ with $\partial/\partial \theta_i$.  This gives rise to an action,
denoted $\odot$, of $\Omega_n$ on itself:
\begin{equation}
\odot: \Omega_n \times \Omega_n \longrightarrow \Omega_n \quad \quad
f \odot g := (\partial f)(g).
\end{equation}

\begin{defn}
\label{w-module-definition} (Rhoades-Wilson \cite{RW})
Given positive integers $k \leq n$, let $\ann \, \delta_{n,k} \subseteq \Omega_n$ be the annihilator 
of the superspace Vandermonde $\delta_{n,k}$ under the $\odot$-action:
\begin{equation}
\ann \,  \delta_{n,k} := \{ f \in \Omega_n \,:\, f \odot \delta_{n,k} = 0 \}.
\end{equation}
We let $\WWW_{n,k}$ be the quotient of $\Omega_n$ by this annihilator:
\begin{equation}
\WWW_{n,k} := \Omega_n / \ann \,  \delta_{n,k}.
\end{equation}
\end{defn}

When $k = n$, the ring $\WWW_{n,n}$ may be identified with the singly-graded type A coinvariant ring
\begin{equation}
R_n := \QQ[x_1, \dots, x_n]/\langle e_1, \dots, e_n \rangle
\end{equation}
where $e_d = e_d(x_1, \dots, x_n)$ is the degree $d$ elementary symmetric polynomial.
Borel \cite{Borel} proved that $R_n$ presents the cohomology 
$H^{\bullet}(\mathcal{F \ell}_n; \QQ)$ of the variety $\mathcal{F \ell}_n$ of complete flags in $\CC^n$.
Since $\mathcal{F \ell}_n$ is a smooth compact complex manifold, this means that the ring $R_n$ 
satisfies Poincar\'e Duality and the Hard Lefschetz Theorem.
%
In this paper we provide a simple generating set (Definition~\ref{ideal-definition}, 
Theorem~\ref{M-is-basis}) of the ideal $\ann \, \delta_{n,k}$.

The quotient ring $\WWW_{n,k}$ is a bigraded $\symm_n$-module;
we let $(\WWW_{n,k})_{i,j}$ denote its bihomogeneous piece in  $x$-degree $i$
and $\theta$-degree $j$.
In \cite{RW}, the following facts were proven about this module.  
Let $\grFrob(\WWW_{n,k}; q, z)$
be the bigraded Frobenius image of $\WWW_{n,k}$, with $q$ tracking  $x$-degree
and $z$ tracking  $\theta$-degree.

\begin{theorem} \label{previous-w-knowledge} (Rhoades-Wilson \cite{RW})
Let $k \leq n$ be positive integers.  We have the following facts concerning the quotient $\WWW_{n,k}$.
\begin{enumerate}
\item {\em (Bidegree bound)} 
The bigraded piece $(\WWW_{n,k})_{i,j}$ is zero unless $0 \leq i \leq N$ and $0 \leq j \leq M$, where
\begin{center}
$N = (n-k) \cdot (k-1) + {k \choose 2}$ and $M = n-k$.
\end{center}
\item  {\em (Superspace Poincar\'e Duality)}
The vector space $(\WWW_{n,k})_{N,M} = \QQ$ is 1-dimensional. For any $0 \leq i \leq N$
and $0 \leq j \leq M$, the multiplication pairing 
\begin{equation*}
(\WWW_{n,k})_{i,j} \times (\WWW_{n,k})_{N-i,M-j} \longrightarrow (\WWW_{n,k})_{N,M} = \QQ
\end{equation*}
is perfect.
\item {\em (Anticommuting degree zero)} The anticommuting degree zero piece of 
$\WWW_{n,k}$ is isomorphic to the quotient ring
\begin{equation}
R_{n,k} := \QQ[x_1, \dots, x_n] / \langle e_n, e_{n-1}, \dots, e_{n-k+1}, x_1^k, \dots, x_n^k \rangle
\end{equation}
where $e_d = e_d(x_1, \dots, x_n)$ is the degree $d$ elementary symmetric polynomial.
\item {\em (Rotational Duality)} The symmetric function $\grFrob(\WWW_{n,k};q,z)$ admits the symmetry
\begin{equation}
(q^N  z^M) \cdot \grFrob(\WWW_{n,k}; q^{-1}, z^{-1}) = \omega (\grFrob(\WWW_{n,k};q,z)).
\end{equation}
Here $\omega$ is the involution on symmetric functions trading $e_n$ and $h_n$.
\end{enumerate}
\end{theorem}

The rings $R_{n,k}$ appearing in 
Theorem~\ref{previous-w-knowledge} (3) were introduced by Haglund, Rhoades, and Shimozono \cite{HRS}
in their study of the Haglund-Remmel-Wilson {\em Delta Conjecture} \cite{HRW}
(whose `rise formulation' was recently proven by D'Adderio and Mellit \cite{DM}).
Pawlowski and 
Rhoades \cite{PR} proved that $R_{n,k}$ presents the cohomology ring $H^{\bullet}(X_{n,k}; \QQ)$
of the variety $X_{n,k}$ of $n$-tuples $(\ell_1, \dots, \ell_n)$ of lines in $\CC^k$ which satisfy 
$\ell_1 + \cdots + \ell_n = \CC^k$.
The variety $X_{n,k}$ is smooth, but not compact. Correspondingly, the Hilbert series of the ring $R_{n,k}$
is not palindromic.
Theorem~\ref{previous-w-knowledge} (3) implies that the `superization' $\WWW_{n,k}$ of $R_{n,k}$
satisfies a bigraded analog of Poincar\'e Duality.
It is for this reason that our title alludes to $\WWW_{n,k}$ as a `duality module'.

Theorem~\ref{previous-w-knowledge} notwithstanding, the paper \cite{RW} left many open questions
about the nature of the bigraded $\symm_n$-modules $\WWW_{n,k}$. 
Indeed, the dimension
of $\WWW_{n,k}$ was unknown.
The purpose of this paper is to elucidate the structure of the duality modules $\WWW_{n,k}$.
In order to do this, we will need the following superspace extensions of set partitions.

\begin{defn}
A {\em set superpartition} of $[n]$ is a set partition $\{ B_1, \dots, B_k \}$ of $[n]$ in which the letters 
$1, \dots, n$ may be decorated with bars, and in which the minimal element $\min(B_i)$ of any block $B_i$
must be unbarred.   An {\em ordered set superpartition} is a  set superpartition
$(B_1 \mid \cdots \mid B_k)$ equipped with a total order on its blocks. 
\end{defn}

As an example,
\begin{equation*}
\{ \{1, \bar{2}, 4 \}, \{3\}, \{5, \bar{6} \} \}
\end{equation*}
is a set superpartition of $[6]$ with three blocks.
This set superpartition gives rise to
$3!$ ordered set superpartitions, one of which is
\begin{equation*}
( 5, \bar{6} \mid 1, \bar{2}, 4 \mid 3 )
\end{equation*}
where by convention we write elements in increasing order within blocks.
Roughly speaking, barred letters will correspond algebraically to  $\theta$-variables.

We define the following families of ordered superpartitions
\begin{align*}
\OSP_{n,k} &:= \{ \text{all  ordered set superpartitions of $[n]$ into $k$ blocks} \}, \\
\OSP_{n,k}^{(r)} &:= \{ \text{all  ordered set superpartitions of $[n]$ into $k$ blocks with $r$ barred letters} \}.
\end{align*}
These sets are counted by
\begin{equation}
\OSP_{n,k} = 2^{n-k} \cdot k! \cdot \Stir(n,k) \quad \text{and} \quad
\OSP_{n,k}^{(r)} = {n-k \choose r} \cdot k! \cdot \Stir(n,k),
\end{equation}
where $\Stir(n,k)$ is the Stirling number of the second kind counting set partitions of $[n]$ into $k$ blocks.

The algebra of $\WWW_{n,k}$ is governed by the combinatorics of  ordered superpartitions.
More precisely, we prove the following.
\begin{itemize}
\item The ideal $\ann \, \delta_{n,k} \subseteq \Omega_n$ defining 
$\WWW_{n,k}$ has an explicit presentation (Definition~\ref{ideal-definition}) involving
elementary symmetric polynomials in partial variable sets 
(Theorem~\ref{M-is-basis}).
\item  The vector space $\WWW_{n,k}$  has a basis indexed by $\OSP_{n,k}$ (Theorem~\ref{M-is-basis}).
\item There are explicit statistics $\coinv$ and $\codinv$ on $\OSP_{n,k}$ (see
Section~\ref{Colored}) such that the bigraded Hilbert series of $\WWW_{n,k}$ is given by
\begin{equation}
\Hilb(\WWW_{n,k}; q, z) = \sum_{r = 0}^{n-k} z^r \cdot \sum_{\sigma \in \OSP_{n,k}^{(r)}} q^{\coinv(\sigma)} =
 \sum_{r = 0}^{n-k} z^r \cdot \sum_{\sigma \in \OSP_{n,k}^{(r)}} q^{\codinv(\sigma)}.
\end{equation}
This bigraded Hilbert series may be computed using a simple recursion
(Corollary~\ref{matrix-recursion}).
\item The $\theta$-degree pieces of $\WWW_{n,k}$ are built out of hook-shaped irreducibles.
More precisely, if we regard $\WWW_{n,k}$ as a singly-graded module under $\theta$-degree,
\begin{equation} \grFrob(\WWW_{n,k}; z) = 
\sum_{(\lambda^{(1)}, \dots, \lambda^{(k)})} z^{n - \lambda^{(1)}_1 - \cdots - \lambda^{(n)}_1} \cdot
s_{\lambda^{(1)}} \cdots s_{\lambda^{(k)}},
\end{equation}
where the sum is over all $k$-tuples $(\lambda^{(1)}, \dots, \lambda^{(k)})$ of nonempty hook-shaped partitions
which satisfy 
$|\lambda^{(1)}| + \cdots + |\lambda^{(k)}| = n$ (Corollary~\ref{hook-corollary}).
\item The monomial expansion of $\grFrob(\WWW_{n,k};q,z)$ is a generating function for
the statistics $\coinv$ and $\codinv$, extended to a multiset analog of  ordered set  superpartitions
(Theorem~\ref{graded-module-structure}).
\end{itemize}
Although the $\codinv$ interpretation of $\grFrob(\WWW_{n,k};q,z)$ will implicitly describe this symmetric
function as a positive sum of LLT polynomials, we do not have a combinatorial interpretation for its Schur 
expansion and leave this as an open problem.

Our results on $\WWW_{n,k}$-modules are  `superizations' of 
facts about the rings $R_{n,k}$ proven in \cite{HRS}. Loosely speaking,
ordered set partitions are replaced by  ordered set superpartitions 
in  appropriate ways.
The {\em proofs} of these results will be significantly different from those of \cite{HRS}, for reasons which 
we now explain.  

Switching coefficients to the complex field, many interesting
graded quotients of $\CC[x_1, \dots, x_n]$ may be constructed
using the method of {\em orbit harmonics} \cite{GP, Griffin, HRS, Kostant, KR, OR}.
If $X \subseteq \CC^n$ is any finite locus of points, let $\II(X) \subseteq \CC[x_1, \dots, x_n]$ 
be the ideal of polynomials vanishing on $X$
\begin{equation}
\II(X) := \{ f \in \CC[x_1, \dots, x_n] \,:\, f(\xx) = 0 \text{ for all $\xx \in X$} \}
\end{equation}
and let $\mathrm{gr} \, \II(X) \subseteq \CC[x_1, \dots, x_n]$ be its associated graded ideal.  
We have isomorphisms
of $\CC$-vector spaces
\begin{equation}
\label{orbit-harmonics-isomorphisms}
\CC[X] \cong \CC[x_1, \dots, x_n]/ \II(X) \cong \CC[x_1, \dots, x_n] / \mathrm{gr} \, \II(X) 
\end{equation}
where  $\CC[x_1, \dots, x_n] / \mathrm{gr} \, \II(X)$ is graded.
If $X$ is stable under the action of some finite subgroup $G \subseteq \mathrm{GL}_n(\CC)$,
these are also isomorphisms of $G$-modules.
In \cite{HRS}, the quotient ring $R_{n,k}$ is shown to arise in this fashion from a locus
$X$ in bijective correspondence with $\OP_{n,k}$.

There is no known version of  orbit harmonics  which applies to the superspace ring 
$\Omega_n$.
Given $f \in \Omega_n$, it is unclear how to define an `evaluation' 
$f(\xx)$ in such a way that the chain of isomorphisms
\eqref{orbit-harmonics-isomorphisms} holds. Indeed, it is unclear what kind of object $\xx$ should be.
One cannot substitute numbers for the $\theta$-variables while respecting anticommutativity.
If $f \in \Omega_n$ has homogeneous $\theta$-degree $k$, we may define
\begin{equation*}
f(\mathbf{v},A) \in \CC  \quad \quad \text{where } \mathbf{v} \in \CC^n, \, \, 
A \in \mathrm{Mat}_{k \times n}(\CC)
\end{equation*}
by evaluating the $x_i$ at the entries of $\mathbf{v}$ and 
letting $\theta$-monomials $\theta_{i_1} \cdots \theta_{i_k}$ act by minors on $A$ in the natural way,
but it is unclear how to extend this to $f \in \Omega_n$ of inhomogeneous $\theta$-degree.
The lack of a superspace theory of orbit harmonics has played a significant part in the still-conjectural
status of isomorphisms like \eqref{fields-conjecture} about superspace quotients.

Given the lack of orbit harmonics in superspace, we analyze the quotient ring 
$\WWW_{n,k}$ more directly by considering its isomorphic harmonic subspace 
$\HHH_{n,k} \subseteq \Omega_n$. This is  the submodule of $\Omega_n$ generated by $\delta_{n,k}$
under the $\odot$-action:
\begin{equation}
\HHH_{n,k} := \{ f \odot \delta_{n,k} \,:\, f \in \Omega_n \}.
\end{equation}
Our analysis of $\HHH_{n,k}$ involves
\begin{itemize}
\item a new total order $\prec$ on monomials in $\Omega_n$ (see Section~\ref{Hilbert}) used
to describe $\HHH_{n,k}$ as a graded vector space, and
\item a new total order $\triangleleft$ on the components of a certain direct sum
decomposition $\Omega_n = \bigoplus_{p,q \geq 0} \Omega_n(p,q)$ (both depending on an auxiliary
parameter $j$) used to describe the graded $\symm_n$-structure of $\HHH_{n,k}$ 
(see Section~\ref{Frobenius}).
\end{itemize}
Roughly speaking, the orders $\prec$ and $\triangleleft$ arise from the superspace intuition that 
a product $x_i^j \theta_i$ of an $x$-variable and the corresponding $\theta$-variable should be given
 a `negative' exponent weight $-j$.
The order $\prec$ restricts to the lexicographical term order on monomials in $\QQ[x_1, \dots, x_n]$, 
but is not a term order on $\Omega_n$ in the sense of Gr\"obner theory.
Indeed, the Gr\"obner theory of important ideals
(such as the superspace coinvariant ideal) in $\Omega_n$ tends to be messier than that of analogous
ideals in $\QQ[x_1, \dots, x_n]$.
On the other hand, we will see in Section~\ref{Hilbert} that the $\prec$-leading terms of elements
in $\HHH_{n,k}$ correspond in a natural way to ordered set superpartitions.

It is our hope that the tools in this paper will prove useful in understanding other quotient rings
involving $\Omega_n$ such as the superspace coinvariant ring.  Indeed, the Fields Group has a conjecture 
(see \cite{Zabrocki}) for the bigraded Frobenius image of 
$\Omega_n/\langle (\Omega_n)^{\symm_n}_+ \rangle$
which is equivalent (by work of \cite{HRS, HRS2, RW}) to
\begin{equation}
\label{true-fields-conjecture}
\{ z^{n-k} \} \,  \grFrob( \Omega_n/\langle (\Omega_n)^{\symm_n}_+ \rangle; q, z) =  \{ z^{n-k} \} \,
\grFrob( \WWW_{n,k}; q, z) \quad  \text{for all $n,k \geq 0$}
\end{equation}
where $\{ z^{n-k} \}$ is the operator which extracts the coefficient of $z^{n-k}$.
In our analysis of $\WWW_{n,k}$, we will give an explicit generating set of its defining 
ideal $\ann \, \delta_{n,k}$.  
This gives rise (Proposition~\ref{ideal-comparison}) to a side-by-side comparison of the 
$\theta$-degree $n-k$ pieces of $\Omega_n/\langle (\Omega_n)^{\symm_n}_+ \rangle$
and $\WWW_{n,k}$ as quotient modules with explicit relations.
Hopefully this similarity will assist in proving \eqref{true-fields-conjecture}.

The rest of the paper is organized as follows.
In {\bf Section~\ref{Background}} we give background material on superspace, $\symm_n$-modules,
and symmetric functions.
{\bf Section~\ref{Colored}} develops combinatorics of ordered set superpartitions necessary for the 
algebraic study of the $\WWW$-modules.
{\bf Section~\ref{Hilbert}} uses harmonic spaces to give a monomial basis of the $\WWW$-modules and
describe their bigraded Hilbert series.
{\bf Section~\ref{Frobenius}} uses skewing operators and harmonics to give a combinatorial formula
for the bigraded Frobenius image of the $\WWW$-modules.
In {\bf Section~\ref{Conclusion}} we conclude with some open problems.

\section{Background}
\label{Background}

\subsection{Alternants in superspace} 
Recall that if $V$ is an $\symm_n$-module, a vector $v \in V$ is an {\em alternant} if
\begin{equation}
w \cdot v = \sign(w) \cdot v \quad \quad \text{for all $w \in \symm_n$.}
\end{equation}
The superspace Vandermondes $\delta_{n,k} \in \Omega_n$ are alternants used 
to construct the quotient rings $\WWW_{n,k}$. 
In order to place $\WWW_{n,k}$ in the proper inductive context, 
we will need a more general family 
of alternants and rings.

\begin{defn}
\label{general-w-definition}
Let $n, k, s \geq 0$ be integers.
Define $\delta_{n,k,s} \in \Omega_n$ to be the element
\begin{equation}
\delta_{n,k,s} := \varepsilon_n \cdot 
( x_1^{k-1} \cdots x_{n-s}^{k-1} x_{n-s+1}^{s-1} \cdots x_{n-1}^1 x_n^0 \times 
\theta_1 \cdots \theta_{n-k})
\end{equation}
where $\varepsilon_n = \sum_{w \in \symm_n} \sign(w) \cdot w \in \QQ[\symm_n]$.
Let $\ann \, \delta_{n,k,s} \subseteq \Omega_n$ be the annihilator of $\delta_{n,k,s}$ and define
$\WWW_{n,k,s}$ to be the quotient ring
\begin{equation}
\WWW_{n,k,s} := \Omega_n / \ann \, \delta_{n,k,s}.
\end{equation}
\end{defn}

In the special case $s = k$, we have $\WWW_{n,k,k} = \WWW_{n,k}$.
The $\WWW_{n,k,s}$ will be useful in putting our arguments in the proper inductive context.
We will primarily be interested in these rings in the case $k \geq s$.

\subsection{Harmonics in superspace} It will be convenient to have a model for $\WWW_{n,k,s}$
as a subspace rather than a quotient of $\Omega_n$.
To this end, we define the {\em harmonic module} $\HHH_{n,k,s} \subseteq \Omega_n$ as follows.

\begin{defn}
\label{general-h-definition}
Let  $n, k, s \geq 0$ and consider $\Omega_n$ as a module over itself
by the $\odot$-action $f \odot g = \partial f(g)$.
We define $\HHH_{n,k,s} \subseteq \Omega_n$ to be the $\Omega_n$-submodule generated
by $\delta_{n,k,s}$.
\end{defn}

More explicitly, the harmonic module $\HHH_{n,k,s}$ is the smallest linear subspace of $\Omega_n$
containing $\delta_{n,k,s}$ which is closed under the action of the commuting 
partial derivatives $\partial/\partial x_1, \dots, \partial/\partial x_n$
as well as the anticommuting partial derivatives
$\partial/\partial \theta_1, \dots, \partial/\partial \theta_n$.
The subspace $\HHH_{n,k,s}$ is a bigraded $\symm_n$-module.

We have a natural inclusion map $\HHH_{n,k,s} \hookrightarrow \Omega_n$.
The composition 
\begin{equation}
\HHH_{n,k,s} \hookrightarrow \Omega_n \twoheadrightarrow \WWW_{n,k,s}
\end{equation}
of this inclusion with the canonical projection of $\Omega_n$ onto $\WWW_{n,k,s}$ is an isomorphism
of bigraded $\symm_n$-modules. We will make use of $\HHH_{n,k,s}$ when we need to consider 
superspace elements in $\Omega_n$ rather than cosets in $\WWW_{n,k,s}$.

\subsection{Symmetric functions and $\symm_n$-modules}
Throughout this paper, we will use the following standard $q$-analogs of numbers, factorials, and 
binomial coefficients:
\begin{equation}
[n]_q := \frac{q^n - 1}{q - 1} = 1 + q + \cdots + q^{n-1} \quad 
[n]!_q := [n]_q [n-1]_q \cdots [1]_q \quad
{n \brack k}_q := \frac{[n]!_q}{[k]!_q \cdot [n-k]!_q}.
\end{equation}

A {\em partition} $\lambda \vdash n$ is a weakly decreasing sequence 
$\lambda = (\lambda_1 \geq \cdots \geq \lambda_k)$ of positive integers which sum to $n$.
We write $\lambda \vdash n$ to mean that $\lambda$ is a partition of $n$.

Let $\Lambda = \bigoplus_{n \geq 0} \Lambda_n$ be the ring of symmetric functions 
in an infinite variable set $\xx = (x_1, x_2, \dots )$ over the ground field $\QQ(q,z)$. 
Bases of the $n^{th}$ graded piece $\Lambda_n$ of this ring are indexed by partitions
$\lambda \vdash n$.  We let 
\begin{equation}
\{ m_{\lambda} \,:\, \lambda \vdash n \}, \quad
\{ e_{\lambda} \,:\, \lambda \vdash n \}, \quad 
\{ h_{\lambda} \,:\, \lambda \vdash n\},\quad \text{and} \quad
\{ s_{\lambda} \,:\, \lambda \vdash n\}
\end{equation}
be the {\em monomial, elementary, homogeneous,} and {\em Schur} bases of $\Lambda_n$.
Given two partitions $\lambda, \mu$ with $\lambda_i \geq \mu_i$ for all $i$, we let
$s_{\lambda/\mu}$ be the corresponding {\em skew Schur function}.

A formal power series $F$ in the variable set $\xx = (x_1, x_2, \dots )$ of bounded degree is 
 {\em quasisymmetric} if the coefficient of $x_1^{a_1} \cdots x_k^{a_k}$ equals the coefficient
of $x_{i_1}^{a_1} \cdots x_{i_k}^{a_k}$ for any strictly increasing sequence $i_1 < \cdots < i_k$ 
of indices.  Given a subset $S \subseteq [n-1]$, the {\em fundamental quasisymmetric function}
of degree $n$ is 
\begin{equation}
F_{S,n} := \sum_{\substack{i_1 \leq \cdots \leq i_n \\ j \in S \, \Rightarrow \, i_j < i_{j+1}}} 
x_{i_1} \cdots x_{i_n}.
\end{equation}
We will encounter the formal power series $F_{S,n}$ exclusively in the case where 
$S$ is the {\em inverse descent set} of a permutation $w \in \symm_n$.
This is the set
\begin{equation}
\iDes(w) := \{ 1 \leq i \leq n-1 \, : \, \pi^{-1}(i) > \pi^{-1}(i+1) \}.
\end{equation}

We  let $\langle -, - \rangle$ be the {\em Hall inner product} on $\Lambda_n$ obtained by declaring the 
Schur functions $s_{\lambda}$ to be orthonormal. For any $F \in \Lambda$, we have a 
`skewing' operator
$F^{\perp}: \Lambda \rightarrow \Lambda$ characterized by
\begin{equation}
\langle F^{\perp} G, H \rangle = \langle G, F H \rangle
\end{equation}
for all $G, H \in \Lambda$. We will make use of the following fact.

\begin{lemma}
\label{skew-by-e-lemma}
Let $F, G \in \Lambda$ be two homogeneous symmetric functions of positive degree. The following 
are equivalent.
\begin{enumerate}
\item We have $F = G$.
\item We have $h_j^{\perp} F = h_j^{\perp} G$ for all $j \geq 1$.
\item We have $e_j^{\perp} F = e_j^{\perp} G$ for all $j \geq 1$.
\end{enumerate}
\end{lemma}

Lemma~\ref{skew-by-e-lemma} follows from the fact that either of the sets 
$\{e_1, e_2, \dots \}$ or $\{h_1, h_2, \dots \}$ are algebraically independent generating sets 
of the ring $\Lambda$ of symmetric functions.

Irreducible representations of $\symm_n$ are in  bijective correspondence with partitions 
$\lambda$ of $n$. If $\lambda \vdash n$ is a partition, let $S^{\lambda}$ be the corresponding 
irreducible $\symm_n$-module.
If $V$ is any finite-dimensional $\symm_n$-module, there are unique multiplicities $c_{\lambda} \geq 0$
such that $V \cong \bigoplus_{\lambda \vdash n} c_{\lambda} S^{\lambda}$.  The {\em Frobenius image}
of $V$ is the symmetric function
\begin{equation}
\Frob(V) := \sum_{\lambda \vdash n} c_{\lambda} s_{\lambda} \in \Lambda_n
\end{equation}
obtained by replacing each irreducible $S^{\lambda}$ with the corresponding Schur function $s_{\lambda}$.

Given two positive integers $n, m$, we have the corresponding {\em parabolic subgroup}
$\symm_n \times \symm_m \subseteq \symm_{n+m}$ obtained by permuting the first $n$ letters 
and the last $m$ letters in $[n+m]$ separately. If $V$ is an $\symm_n$-module and $W$ is an
$\symm_m$-module, their {\em induction product} $V \circ W$ is the $\symm_{n+m}$-module
given by
\begin{equation}
V \circ W := \mathrm{Ind}^{\symm_{n+m}}_{\symm_n \times \symm_m}(V \otimes W).
\end{equation} 
Induction product and Frobenius image are related in that
\begin{equation}
\Frob(V \circ W) = \Frob(V) \cdot \Frob(W).
\end{equation}

Frobenius images interact with the skewing operators $h_j^{\perp}$ and 
 $e_j^{\perp}$ interact in the following way.
Let $V$ be an $\symm_n$-module, let $1 \leq j \leq n$, and consider the parabolic subgroup 
$\symm_j \times \symm_{n-j}$ of $\symm_n$. 
We have the group algebra elements $\eta_j, \varepsilon_j \in \QQ[\symm_j]$
\begin{equation}
\eta_j := \sum_{w \in \symm_j} w \quad \quad \quad \quad 
\varepsilon_j := \sum_{w \in \symm_j} \sign(w) \cdot w 
\end{equation}
which symmetrize and antisymmetrize in the {\bf first} $j$ letters. 
Since $\eta_j$ and $\varepsilon_j$ commute
with permutations in the second parabolic factor $\symm_{n-j}$, 
the vector spaces
\begin{equation}
\eta_j V = \{ \eta_j  \cdot v \,:\, v \in V \} \quad \quad \quad \quad 
 \varepsilon_j V = \{ \varepsilon_j \cdot v \,:\, v \in V \}
\end{equation}
are naturally $\symm_{n-j}$-modules.
The Frobenius images of these modules are as follows.

\begin{lemma}
\label{module-skew-by-e}
We have $\Frob(\eta_j V) = h_j^{\perp} \Frob(V)$ and  $\Frob(\varepsilon_j V) = e_j^{\perp} \Frob(V)$.
\end{lemma}

Lemma~\ref{module-skew-by-e} may be generalized by considering the image of $V$ under 
$\sum_{w \in \symm_j} \chi^{\mu}(w) \cdot w \in \QQ[\symm_j]$ where $\mu \vdash k$ is any partition
and $\chi^{\mu}: \symm_j \rightarrow \CC$ is the irreducible character; the effect on Frobenius images
is the operator $s_{\mu}^{\perp}$. This fact may be proved by Frobenius reciprocity.

In this paper we will consider (bi)graded vector spaces and modules. If $V = \bigoplus_{i \geq 0} V_i$
is a graded vector space with each piece $V_i$ finite-dimensional, recall that its {\em Hilbert series}
is given by 
\begin{equation}
\Hilb(V; q) = \sum_{i \geq 0} \dim (V_i) \cdot q^i.
\end{equation}
Similarly, if $V = \bigoplus_{i,j \geq 0} V_{i,j}$ is a bigraded vector space, we have the 
{\em bigraded  Hilbert series}
\begin{equation}
\Hilb(V; q,z) = \sum_{i,j \geq 0} \dim(V_{i,j}) \cdot q^i z^j.
\end{equation}
If $V = \bigoplus_{i \geq 0} V_i$ is a graded $\symm_n$-module, its {\em graded Frobenius image}
is 
\begin{equation}
\grFrob(V; q) = \sum_{i \geq 0} \Frob(V_i) \cdot q^i.
\end{equation}
Extending this, if $V = \bigoplus_{i,j \geq 0} V_{i,j}$ is a bigraded $\symm_n$-module, its
{\em bigraded Frobenius image} is 
\begin{equation}
\grFrob(V; q,z) = \sum_{i,j \geq 0} \Frob(V_{i,j}) \cdot q^i z^j.
\end{equation}

\subsection{Ordered set superpartitions} We will show that the duality modules $\WWW_{n,k}$ 
are governed by the combinatorics of ordered set superpartitions in $\OSP_{n,k}$.
The more general modules $\WWW_{n,k,s}$ of Definition~\ref{general-w-definition} which
we will use to inductively describe the $\WWW_{n,k}$ are controlled by the 
following more general combinatorial objects. 

\begin{defn}
\label{colored-osp}
For $n, k, s \geq 0$, we let $\OSP_{n,k,s}$ be the family of $k$-tuples 
$(B_1 \mid \cdots \mid B_k)$ of sets of positive integers such that
\begin{itemize}
\item we have the disjoint union decomposition $[n] = B_1 \sqcup \cdots \sqcup B_k$,
\item the first $s$ sets $B_1, \dots, B_s$ are nonempty,
\item the elements of $B_1, \dots, B_k$ may be barred or unbarred, provided
\item the minimal elements $\min B_1, \dots, \min B_s$ of the first $s$ sets are unbarred.
\end{itemize}
We denote by $\OSP_{n,k,s}^{(r)} \subseteq \OSP_{n,k,s}$ the subfamily of $\sigma \in \OSP_{n,k,s}$
with $r$ barred elements.
\end{defn}

We refer to elements $\sigma = (B_1 \mid \cdots \mid B_k) \in \OSP_{n,k,s}$ 
as {\em  ordered set superpartitions}, despite the 
fact that  any of the last $k-s$ sets $B_{s+1}, \dots, B_k$ in  $\sigma$ could be empty and that the minimal
elements of $B_{s+1}, \dots, B_k$ (if they exist) may be barred.

\section{Ordered Set Superpartitions}
\label{Colored}

\subsection{The statistics $\coinv$ and $\codinv$}
In this section, we define two statistics on $\OSP_{n,k,s}$. The first of these is an extension
of the classical inversion statistic (or rather, its complement) on permutations in $\symm_n$.
Given $\pi = \pi_1 \dots \pi_n \in \symm_n$, its {\em coinversion code} is the 
sequence $(c_1, \dots, c_n)$ where $c_i$ is the number of entries in the set $\{i+1, i+2, \dots, n \}$
which appear to the right of $i$ in $\pi$. 
The {\em coinversion number} of $\pi$ is the sum 
$\coinv(\pi) = c_1 + \cdots + c_n$.
We generalize these concepts as follows.

\begin{defn}
\label{coinversion-code}
Let $\sigma = (B_1 \mid \cdots \mid B_k) \in \OSP_{n,k,s}$ 
be an ordered set superpartition. The {\em coinversion code}
is the length $n$ sequence  $\code(\sigma) = (c_1, \dots, c_n)$  over the alphabet
$\{0, 1, 2, \dots, \bar{0}, \bar{1}, \bar{2}, \dots \}$ whose $a^{th}$ entry $c_a$
is defined as follows.  Suppose that $a$ lies in the $i^{th}$ block $B_i$ of $\sigma$.
\begin{itemize}
\item The entry $c_a$ is barred if and only if $a$ is barred in $\sigma$.
\item If $a = \min B_i$ and $i \leq s$, then
\begin{equation*}
c_a = | \{ i+1 \leq j \leq s \} \,:\, \min B_j > a \}.
\end{equation*}
\item If $a$ is barred, then
\begin{equation*}
c_a = \begin{cases}
| \{ 1 \leq j \leq i-1 \,:\, \min B_j < a \} | & \text{if $i \leq s$}, \\
| \{ 1 \leq j \leq s \,:\, \min B_j < a \} | + (i-s-1) & \text{if $i > s$}.
\end{cases}
\end{equation*}
\item Otherwise, we set
\begin{equation*}
c_a = | \{ i+1 \leq j \leq s \,:\, \min B_j > a \} | + (i-1).
\end{equation*}
\end{itemize}
The {\em coinversion number} of $\sigma$ is the sum
\begin{equation*}
\coinv(\sigma) := c_1 + \cdots + c_n
\end{equation*}
of the entries $(c_1, \dots, c_n)$ in $\code(\sigma)$.
\end{defn}

For example, consider the  ordered set superpartition 
\begin{equation*} \sigma = 
( 2, \, \bar{5} \mid 3, \, 6, \, \bar{8}, \, 9 \mid \varnothing \mid 
\bar{1}, \, \bar{4}, \, 7 \mid \varnothing) \in \OSP_{9,5,2}.
\end{equation*}
The coinversion code of $\sigma$ is given by 
\begin{equation*}
\code(\sigma) = (c_1, \dots, c_9) = 
( \bar{1}, 1, 0, \bar{3}, \bar{0}, 1, 3, \bar{1}, 1)
\end{equation*}
so that
\begin{equation*}
\coinv(\sigma) = 1 + 1 + 0 + 3 + 0 + 1 + 3 + 1 + 1 = 11.
\end{equation*}
Bars on the entries $c_1, \dots, c_n$ are ignored when calculating  
$\coinv(\sigma) = c_1 + \cdots + c_n$.

The coinversion code $(c_1, \dots, c_n)$ can be visualized by considering the 
{\em column diagram} notation for colored ordered set partitions. 
Given $\sigma = (B_1 \mid \cdots \mid B_k) \in \OSP_{n,k,s}$, we draw the entries of $B_i$
in the $i^{th}$ column. The entries of $B_i$ fill the $i^{th}$ column according to the following rules.
\begin{itemize}
\item For $1 \leq i \leq k$, the barred entries of $B_i$ start at height 1 and fill up in increasing order.
\item For $1 \leq i \leq s$, the 
unbarred entries of $B_i$ start at height 0 and fill down in increasing order.
\item For $s+1 \leq i \leq k$, the 
unbarred entries of $B_i$ start at height $-1$ and fill down in increasing order; we also place
a $\bullet$ at height 0 in these columns.
\end{itemize}
In our example, the column diagram of 
\begin{equation*}
( 2, \, \bar{5} \mid 3, \, 6, \, \bar{8}, \, 9 \mid \varnothing \mid 
\bar{1}, \, \bar{4}, \, 7 \mid \varnothing) \in \OSP_{9,5,2}.
\end{equation*}
is given by
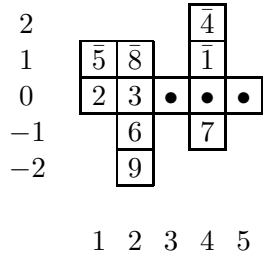
\begin{equation*}
\begin{Young}
,2& ,&  ,& ,& ,& \bar{4} &, \cr
,1 & ,& \bar{5} & \bar{8} & , &  \bar{1} &, \cr
 ,0&,& 2 & 3 & \bullet & \bullet & \bullet \cr
 ,-1 &,& ,& 6 & ,& 7 &, \cr
  ,-2 &,&, & 9 &, & ,&, \cr
    ,&,& ,& ,& ,& ,&, \cr
     ,&,& ,1& ,2& ,3& ,4& ,5\cr
\end{Young}
\end{equation*}
where the column indices are shown below and the heights of entries are shown on the left.
The three $\bullet$'s at the height 0 level correspond to the fact that the blocks 
$B_3, B_4, B_5$ are allowed to be empty for $\sigma \in \OSP_{9,5,2}$.
Given $\sigma \in \OSP_{n,k,s}$, we have the following column diagram interpretation of the $a^{th}$ letter
$c_a$ of
$\code(\sigma) = (c_1, \dots, c_n)$.
\begin{itemize}
\item If $a$ appears at height 0, then $c_a$ counts the number of height zero entries to the right
of $a$ which are  $> a$.
\item If $a$ appears at negative height in column $i$, then $c_a$ is $i-1$, plus the number of height 0 entries
to the right of $a$ which are $> a$.
\item If $a$ appears at positive height, then $c_a$ counts the number of height 0 entries to the left of $a$
which are $< a$, plus the number of $\bullet$'s to the left of $a$.
\end{itemize}

We will see that  codes $(c_1, \dots, c_n)$
of elements $\sigma \in \OSP_{n,k,s}$ correspond to a monomial basis of the quotient ring
$\WWW_{n,k,s}$.  In order to obtain the bigraded Frobenius image of $\WWW_{n,k,s}$ in terms
of $\coinv$, we use diagrams to define a formal power series as follows.

The {\em reading word} of $\sigma \in \OSP_{n,k,s}$
is the permutation in 
$\symm_n$ obtained from by reading the column
diagram of $\sigma$ from top to bottom, and, within each row,
{\bf from right to left} (ignoring any bars on letters).
If $\sigma$ is our example ordered set partition above, we have
\begin{equation*}
\rd(\sigma) = 418532769 \in \symm_9.
\end{equation*}

\begin{defn}
\label{c-function-definition}
Let $n, k, s \geq 0$ be integers and let $0 \leq r \leq n-s$. We define a quasisymmetric function
$C^{(r)}_{n,k,s}(\xx;q)$ by the formula
\begin{equation}
C^{(r)}_{n,k,s}(\xx;q) := \sum_{\sigma \in \OSP_{n,k,s}^{(r)}} q^{\coinv(\sigma)} \cdot F_{\iDes(\rd(\sigma)),n}(\xx).
\end{equation}
Also define a quasisymmetric function $C_{n,k,s}(\xx;q,z)$ by 
\begin{equation}
C_{n,k,s}(\xx;q,z) := \sum_{r = 0}^{n-s}  C^{(r)}_{n,k,s}(\xx;q) \cdot z^r.
\end{equation}
\end{defn}

We may avoid the use of the $F$'s in the definition of the $C$-functions by allowing repeated entries
in our column diagrams. Let $\CT_{n,k,s}$ be the family of all fillings of finite subsets of the infinite
strip $[k] \times \infty$ with the alphabet $\{ \bullet, 1, 2,  \dots, \bar{1},  \bar{2}, \dots \}$
such that
\begin{itemize}
\item The height 0 row 
is filled with a sequence of $s$ unbarred numbers, followed by a sequence of $k-s$ $\bullet$'s.
\item The filled cells form a contiguous sequence within each column.
\item Entries at positive height are barred while entries at negative height are unbarred
\item Unbarred numbers weakly increase  going down and barred numbers strictly increase  going up. 
\item An unbarred number at height 0 is strictly smaller than any barred number above it.
\end{itemize}
An example column tableau $\tau \in \CT_{13, 5, 2}$ is shown below.
\begin{equation*}
\begin{Young}
,2& ,& \bar{4} & \bar{6} & ,& \bar{4} &, \cr
,1 & ,& \bar{3} & \bar{5} &,  &  \bar{1} &, \cr
 ,0&,& 2 & 3 & \bullet & \bullet & \bullet \cr
 ,-1 &,& ,& 3 &, & 7 & 1 \cr
  ,-2 &,& ,& 4 & ,& ,& 1 \cr
    ,&,& ,& ,& ,& ,&, \cr
     ,&,& ,1& ,2& ,3& ,4& ,5\cr
\end{Young}
\end{equation*}
The coinversion number $\coinv(\tau)$ extends naturally to column tableaux $\tau \in \CT_{n,k,s}$;
given an entry $a$ in such a tableau we may compute it contribution $c_a$ 
to $\coinv$ as before, and sum over all entries.
Let $\xx^{\tau} := x_1^{\alpha_1} x_2^{\alpha_2} \cdots $ where $\alpha_i$ is the 
number of $i's$ in $\tau$.
In our example we have $\xx^{\tau} = x_1^3 x_2 x_3^3 x_4^3 x_5 x_6 x_7$.
The exponent sequence $\alpha = (\alpha_1, \alpha_2, \dots )$ of $\xx^{\tau}$ is called 
the {\em content} of $\tau$.
Let $\CT_{n,k,s}^{(r)} \subseteq \CT_{n,k,s}$ be the subfamily of column tableaux with $r$ 
barred letters.

\begin{observation}
\label{c-observation}
The formal power series $C_{n,k,s}^{(r)}(\xx;q)$ and $C_{n,k,s}(\xx;q,z)$ are given by 
\begin{equation*}
C_{n,k,s}^{(r)}(\xx;q) = \sum_{\tau \in \CT_{n,k,s}^{(r)}} q^{\coinv(\tau)} \xx^{\tau} \quad \text{and} \quad
C_{n,k,s}(\xx;q) = \sum_{\tau \in \CT_{n,k,s}} q^{\coinv(\tau)} \xx^{\tau}. 
\end{equation*}
\end{observation}

The formulas for the $C$-functions in Observation~\ref{c-observation} are more aesthetic,
but less efficient, than those in Definition~\ref{c-function-definition}.
It will turn out that $C_{n,k,s}(\xx;q,z)$ is the bigraded Frobenius image $\grFrob(\WWW_{n,k,s};q,z)$.
For reasons related to skewing recursions, the refinement 
$C^{(r)}_{n,k,s}(\xx;q)$ will be convenient to consider.
At this point, it is not clear that the $C$-functions are even symmetric.
Their symmetry (and Schur positivity) will follow from an alternative description in terms of 
another statistic on $\OSP_{n,k,s}$.

Let $\sigma \in \OSP_{n,k,s}$. We  augment the column diagram of $\sigma$ by placing a $+ \infty$
below the entries in every column.
Furthermore, we regard every $\bullet$ in the column diagram as being filled with a 0.
A pair of entries $a < b$ form a {\em diagonal coinversion pair}
if 
\begin{itemize}
\item $a$ appears to the left of and at the same height as $b$, or
\item $a$ appears to the right of and at height one less than $b$.
\end{itemize}
Schematically, these conditions have the form
\begin{equation*}
\begin{Young}
a &, &,  \cdots &, &  b
\end{Young}  \quad \quad  \text{and} \quad \quad 
\begin{Young}
b &, &, &, & \cr
&, &, \cdots  &, & a
\end{Young}
\end{equation*}
for $a < b$.

\begin{defn}
\label{dinv-defn}
For $\sigma \in \OP_{n,k,s}$, the {\em diagonal coinversion number} $\codinv(\sigma)$ is the total
number of diagonal coinversion pairs in the column diagram of
$\sigma$.
\end{defn}

It will turn out that $\codinv$ is equidistributed with $\coinv$ on $\OSP_{n,k,s}$, and even on the subsets
$\OSP_{n,k,s}^{(r)}$ obtained by restricting to $r$ barred  letters.
In analogy with the $C$-functions, we define a quasisymmetric function attached to $\codinv$.

\begin{defn}
\label{d-function-definition}
Let $n, k, s \geq 0$ be integers 
and let $0 \leq r \leq n-s$. Define a quasisymmetric function 
$D^{(r)}_{n,k,s}(\xx;q)$ by the formula
\begin{equation}
D^{(r)}_{n,k,s}(\xx;q) := \sum_{\sigma \in \OSP_{n,k,s}^{(r)}} q^{\codinv(\sigma)} \cdot F_{\iDes(\rd(\sigma)),n}(\xx).
\end{equation}
Also define 
\begin{equation}
D_{n,k,s}(\xx;q,z) := \sum_{r = 0}^{n-s} D^{(r)}_{n,k,s}(\xx;q) \cdot z^r.
\end{equation}
\end{defn}

Like the $C$-functions, the $D$-functions may be expressed in terms of infinite sums 
over column tableaux.

\begin{observation}
\label{d-observation}
The formal power series $D_{n,k,s}^{(r)}(\xx;q)$ and $D_{n,k,s}(\xx;q,z)$ are given by 
\begin{equation*}
D_{n,k,s}^{(r)}(\xx;q) = \sum_{\tau \in \CT_{n,k,s}^{(r)}} q^{\codinv(\tau)} \xx^{\tau} \quad \text{and} \quad
D_{n,k,s}(\xx;q) = \sum_{\tau \in \CT_{n,k,s}} q^{\codinv(\tau)} \xx^{\tau}. 
\end{equation*}
\end{observation}

We use the theory of LLT polynomials
to show that the $D$-functions are symmetric and Schur positive.

\begin{proposition}
\label{d-is-symmetric}
The quasisymmetric functions $D_{n,k,s}^{(r)}(\xx;q)$ and 
$D_{n,k,s}(\xx;q,z)$ are symmetric and Schur positive.
\end{proposition}

\begin{proof}
We prove this fact by writing $D_{n,k,s}(\xx; q, z)$ as a positive
linear combination of \emph{LLT polynomials} \cite{LLT}. We use the version of LLT polynomials employed by Haglund, Haiman, and Loehr \cite{HHL}. 

A \emph{skew diagram} is a set of cells in the first quadrant given by $\mu / \nu$ for some partitions $\mu \supseteq \nu$. Given a tuple $\bm{\lambda} = (\lambda^{(1)}, \lambda^{(2)}, \ldots, \lambda^{(l)})$ of skew diagrams, the LLT polynomial indexed by $\bm{\lambda}$ is 
$$\mathrm{LLT}_{\bm{\lambda}}(\xx; q) = \sum q^{\mathrm{inv}(T)} \xx^T$$
where the sum is over all semistandard fillings of the skew diagrams $\bm{\lambda}$, $\xx^T$ is the product where $x_i$ appears as many times as $i$ appears in the filling $T$, and $\inv(T)$ is the following statistic. Given a cell $u$ appearing at coordinates $(x, y)$ in one of the skew shapes $\lambda^{(i)}$, the \emph{content} of $u$ is $c(u) = x - y$. Then $\inv(T)$ is the number of pairs of cells $u \in \lambda^{(i)}$, $v \in \lambda^{(j)}$ with $i < j$ such that
\begin{itemize}
\item $c(u) = c(v)$ and $T(u) > T(v)$, or
\item $c(u) + 1 = c(v)$ and $T(u) < T(v)$.
\end{itemize}
LLT polynomials are Schur-positive symmetric functions \cite{GH, HHL, LLT}. We claim that we can decompose $D_{n,k,s}^{(r)}(\xx;q)$ into LLT polynomials:
\begin{align}
\label{llt-decomp}
D_{n,k,s}^{(r)}(\xx; q) = \sum_{\bm{\lambda}} q^{\mathrm{stat}(\bm{\lambda})} \mathrm{LLT}_{\bm{\lambda}}(\xx; q)
\end{align}
where $\mathrm{stat}$ is a fixed statistic depending on $\bm{\lambda}$ and the sum is over all tuples of skew diagrams $\bm{\lambda} = (\lambda^{(1)}, \ldots, \lambda^{(k)})$ satisfying
\begin{itemize}
\item if $i \leq k-s$, then $\lambda^{(i)} = \mu^{(i)} / (1)$ for a hook shape $\mu^{(i)}$, 
\item if $i > k-s$, then $\lambda^{(i)}$ is a single non-empty hook shape,
\item $\sum_{i=1}^{k} |\lambda^{(i)}| = n$, and
\item $\bm{\lambda}$ has $r$ total cells of negative content. 
\end{itemize}
We define nonnegative integer sequences $\alpha, \beta \in \mathbb{N}^{k}$ so that 
\begin{itemize}
\item $\lambda^{(i)} = (\alpha_i +1, 1^{\beta_i}) / (1)$ for $i \leq k-s$ and
\item $\lambda^{(i)} = (\alpha_i, 1^{\beta_i})$ for $i > k-s$.
\end{itemize}
Given a specific semistandard filling $T$ that contributes to $\mathrm{LLT}_{\bm{\lambda}}(\xx; q)$ for such a $\bm{\lambda}$, we will create a column tableau $\tau \in \CT_{n,k,s}^{(r)}$ such that $\tau$ will have exactly $\alpha_i$ unbarred and $\beta_i$ barred entries in column $k-i+1$. The sequences $\alpha$ and $\beta$ will completely determine how many diagonal coinversions involving an $\infty$ or a $\bullet$ appear in $\tau$. In Figure~\ref{fig:llt}, there are 5 such diagonal coinversions involving an $\infty$ and 7 involving a $\bullet$. In general, this number is $$\mathrm{stat}(\bm{\lambda}) := |\alpha| + \sum_{i =1}^{k-s} | \{j > i: \beta_j \neq 0 \}|,$$ a number which can be computed from $\bm{\lambda}$ only.

Now, given a specific semistandard filling $T$ of $\bm{\lambda}$, we map the unique entry in $\lambda^{(i)}$ with content $j$ to column $k-i+1$ and row $-j$ in $\tau$. Every pair which contributes to $\inv(T)$ now contributes a diagonal coinversion between integers (not $\bullet$ or $\infty$) in $\tau$. In Figure~\ref{fig:llt}, there are 10 such diagonal coinversions and the filling $T$ has $\inv(T) = 10$. Since this map is bijective, we have the desired result.
\end{proof}

\begin{figure}
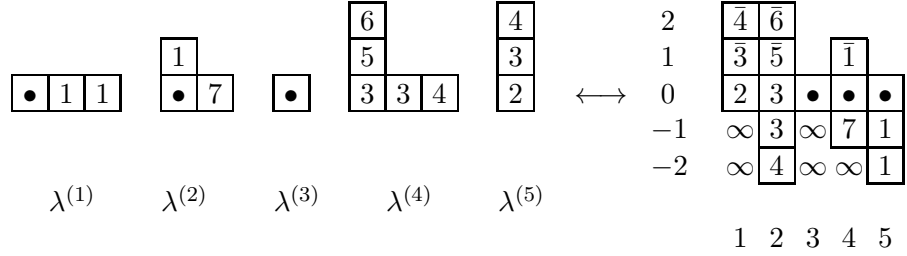

\begin{center}
\begin{Young}
, & $\bullet$ &1 & 1 \cr
, \cr
, \cr
, &, &, $\lambda^{(1)}$ \cr
, \cr
\end{Young}
\quad
\begin{Young}
1 \cr
$\bullet$ &7 \cr
, \cr
, \cr
, $\lambda^{(2)}$ \cr
, \cr
\end{Young}
\quad
\begin{Young}
$\bullet$ \cr
, \cr
, \cr
, $\lambda^{(3)}$ \cr
, \cr
\end{Young}
\quad
\begin{Young}
6\cr
5 \cr
3 & 3 & 4 \cr
, \cr
, \cr
, &,$\lambda^{(4)}$ \cr
, \cr
\end{Young}
\quad
\begin{Young}
4\cr
3 \cr
2  \cr
, \cr
, \cr
,$\lambda^{(5)}$ \cr
, \cr
\end{Young}
\quad
\begin{Young}
, \cr
,\cr 
,$\longleftrightarrow$  \cr
, \cr
, \cr
, \cr
, \cr
\end{Young}
\quad
\begin{Young}
,2& ,& $\bar{4}$ & $\bar{6}$ & ,& , &, \cr
,1 & ,& $\bar{3}$ & $\bar{5}$ &,  &  $\bar{1}$ &, \cr
 ,0&,& 2 & 3 & $\bullet$ & $\bullet$ & $\bullet$ \cr
 ,$-1$ &,& ,$\infty$& 3 &,$\infty$ & 7 & 1 \cr
  ,$-2$ &,& ,$\infty$& 4 & ,$\infty$& ,$\infty$& 1 \cr
    ,&,& ,& ,& ,& ,&, \cr
     ,&,& ,1& ,2& ,3& ,4& ,5\cr
\end{Young}
\end{center}
\caption{We depict an example of the correspondence in the proof of Proposition \ref{d-is-symmetric}. 
Each of the 5 skew diagrams on the left is justified so that its bottom left entry has content 0. }
\label{fig:llt}
\end{figure}

\subsection{$C = D$}
Our first main result states that the $C$-functions and $D$-functions coincide. 
This given,
Proposition~\ref{d-is-symmetric} implies that the
$C$-functions are symmetric and Schur positive.
We know of no direct proof of either of these facts.

\begin{theorem}
\label{c-equals-d}
For any integers $n,k,s \geq 0$ with 
$k \geq s$ and any $0 \leq r \leq n-s$, we have the equality of formal power
series
\begin{equation}
C_{n,k,s}^{(r)}(\xx;q) = D_{n,k,s}^{(r)}(\xx;q).
\end{equation}
Consequently, we have
\begin{equation}
C_{n,k,s}(\xx;q,z) = D_{n,k,s}(\xx;q,z).
\end{equation}
\end{theorem}

Our proof of Theorem~\ref{c-equals-d} is combinatorial and uses the 
column tableau forms of the formal power series
$C_{n,k,s}^{(r)}(\xx;q)$ and $D_{n,k,s}^{(r)}(\xx;q)$ in 
Observations~\ref{c-observation} and \ref{d-observation}.
The idea is to consider building up a general column tableau 
$\tau \in \CT_{n,k,s}^{(r)}$ from the column tableau consisting of $k-s$ empty columns
by successively adding larger entries
and showing that the statistics $\coinv$ and $\codinv$ satisfy the same recursions.

\begin{proof}
Let $\tau$ be a column tableau with $k$ columns  and $k-s$ $\bullet$'s whose entries are $< N$.
We consider building a larger column tableau involving $N$'s from $\tau$ by the following three-step
process.
\begin{enumerate}
\item Placing $\bar{N}$'s on top of some subset of the $k$ columns of $\tau$.
\item Placing $N$'s at height 0 between and on either side of the $s$ columns of $\tau$
without a  $\bullet$.
\item Placing $N$'s at negative heights below some multiset of columns of the resulting figure.
\end{enumerate}
We track the behavior of $\coinv$ and $\codinv$ as we perform this procedure, starting with placing the
$\bar{N}$'s on top of columns.

Placing a $\bar{N}$ on top of column
$i$ of $\tau$ increases the statistic $\coinv$ by $i-1$. 
We reflect this fact by giving column $i$ the {\em barred $\coinv$ label} of $i-1$.
The barred $\coinv$ labels of the column tableau are shown in bold and barred below.
\begin{equation*}
\begin{Young}
,3 &, &, \bf{\bar{0}} &, \bf{\bar{1}} &, & ,\bf{\bar{3}} &, \cr
,2& ,&  \bar{4} & \bar{6} &, & \bar{4} &, \cr
,1 & ,& \bar{3} & \bar{5} &,  \bf{\bar{2}} &  \bar{1}  &, \bf{\bar{4}} \cr
 ,0&,& 2 & 3 & \bullet & \bullet & \bullet \cr
 ,-1 &,&, & 3 &, & 7 & 1 \cr
  ,-2 &,&, & 4 &, &, & 1 \cr
    ,&,& ,& ,& ,& ,&, \cr
     ,&,& ,1& ,2& ,3& ,4& ,5\cr
\end{Young}
\end{equation*}
The barred $\coinv$ labels give rise to a bijection
\begin{multline}
\iota^{\coinv}_{\bar{N}}:
\left\{  \begin{array}{c}\text{$k$-column tableaux $\tau$ with} \\\text{$k-s$ $\bullet$'s and all entries $< N$} \end{array} \right\}
\times 
\left\{ \begin{array}{c}\text{subsets $S$ of} \\ \text{$\{0,1,\dots,k-1\}$}\end{array} \right\} \longrightarrow  \\
\left\{  \begin{array}{c}\text{$k$-column tableaux $\tau'$ with $k-s$ $\bullet$'s} \\
\text{such that all entries of $\tau'$ are $\leq N$ and}\\ \text{the only $N$'s in $\tau'$ are barred}\end{array} \right\}
\end{multline} 
by letting $\iota^{\coinv}_{\bar{N}}(\tau,S)$ be the tableau $\tau'$ obtained by placing a 
$\bar{N}$ on top of every column with barred $\coinv$ label in $S$. 
If $\iota^{\coinv}_{\bar{N}}: (\tau,S) \mapsto \tau'$, then
$\coinv(\tau') = \coinv(\tau) + \sum_{i \in S} i$.

Next, we consider the effect of $\bar{N}$ insertion on $\codinv$.
We bijectively label the $k$ columns of $\tau$ with the $k$ {\em barred $\codinv$ labels} 
$0, 1, \dots, k-1$ (in that order) in descending order of maximal height and, within columns of the same 
maximal height, proceeding from right to left.
The barred $\codinv$ labels in our example are shown below.
\begin{equation*}
\begin{Young}
,3 &, & ,{\bf  \bar{2}} & ,{\bf \bar{1}} & ,&, {\bf \bar{0}} &, \cr
,2& ,&  \bar{4} & \bar{6} &, & \bar{4} &, \cr
,1 & ,& \bar{3} & \bar{5} &, {\bf \bar{4}}  &  \bar{1} &, {\bf \bar{3}} \cr
 ,0&,& 2 & 3 & \bullet & \bullet & \bullet \cr
 ,-1 &,&, & 3 &, & 7 & 1 \cr
  ,-2 &,&, & 4 & ,&, & 1 \cr
    ,&,& ,& ,& ,& ,&, \cr
     ,&,& ,1& ,2& ,3& ,4& ,5\cr
\end{Young}
\end{equation*}
The barred $\codinv$ labels give a bijection
\begin{multline}
\iota^{\codinv}_{\bar{N}}:
\left\{  \begin{array}{c}\text{$k$-column tableaux $\tau$ with} \\\text{$k-s$ $\bullet$'s and all entries $< N$} \end{array} \right\}
\times 
\left\{ \begin{array}{c}\text{subsets $S$ of} \\ \text{$\{0,1,\dots,k-1\}$}\end{array} \right\} \longrightarrow  \\
\left\{  \begin{array}{c}\text{$k$-column tableaux $\tau'$ with $k-s$ $\bullet$'s} \\
\text{such that all entries of $\tau'$ are $\leq N$ and}\\ \text{the only $N$'s in $\tau'$ are barred}\end{array} \right\}
\end{multline} 
where $\iota^{\codinv}_{\bar{N}}(\tau,S)$ is obtained by $\tau$ by placing a $\bar{N}$
on top of every column with barred $\codinv$ label indexed by $S$.
If $\iota^{\codinv}_{\bar{N}}: (\tau,S) \mapsto \tau'$, then
$\codinv(\tau') = \codinv(\tau) + \sum_{i \in S} i$.

We move on to Step 2 of our insertion procedure: creating new columns by placing $N$'s at height 0. 
A single `height 0 insertion map'
$\iota_0$ of this kind on column tableaux has the same effect on $\coinv$ and $\codinv$.

If $\tau$ is a column tableau with $s$ non-bullet letters at height zero,  $N$'s can be placed 
(with repetition) in any of the $s$ places between and on either side of these non-bullet letters.
This gives rise to a bijection
\begin{multline}
\iota_0: 
\left\{  \begin{array}{c}\text{$k$-column tableaux $\tau'$ with $k-s$ $\bullet$'s} \\
\text{such that all entries of $\tau'$ are $\leq N$ and}\\ \text{the only $N$'s in $\tau'$ are barred}\end{array} \right\} 
\times
\left\{
\begin{array}{c}
\text{finite multisets $S$} \\ \text{drawn from $\{0,1,\dots,s\}$}
\end{array}
\right\} \longrightarrow \\
\bigsqcup_{K \geq k} 
\left\{  \begin{array}{c}\text{$K$-column tableaux $\tau''$ with $k-s$ $\bullet$'s} \\
\text{such that all entries of $\tau''$ are $\leq N$ and}\\ 
\text{no $N$'s in $\tau''$ have negative height}\end{array} \right\} 
\end{multline}
where $\iota_0(\tau',S)$ is obtained from $\tau'$ by inserting $m_i$ copies of $N$ after column $i$ of 
$\tau$, where $m_i$ is the multiplicity of $i$ in $S$.
Suppose $\iota_0(\tau',S) = \tau''$. The number $K$ of columns of $\tau''$ is related to the number $k$
of columns of $\tau'$ by $K = k + |S|$. Furthermore, if there are $b$ unbarred entries in $\tau'$ of negative 
height, we have
\begin{equation}
\coinv(\tau'') = \coinv(\tau') + b + \sum_{i \in S} i \quad \text{and} \quad
\codinv(\tau'') = \codinv(\tau') + b + \sum_{i \in S} i.
\end{equation}
In other words, the height zero insertion map $\iota_0$ has the same effect on 
$\coinv$ and $\codinv$.

Finally, let $\tau''$ be column tableau with $K$ columns with entries $\leq N$ in which there are no
$N$'s with negative height. To perform Step 3 of our insertion process, we insert  
$N$'s at the bottom of some multiset of columns of $\tau''$.
We track the effect on $\coinv$ and $\codinv$ as before.

We label the columns of $\tau''$ from left-to-right with the {\em unbarred $\coinv$ labels} 
$0,1, \dots, K$. An example of this labeling is shown below in bold.
Since we add unbarred letters on the bottom of a tableau, we show the unbarred labels there, as well.
\begin{equation*}
\begin{Young}
,2& ,&  \bar{4} & \bar{6} & ,& \bar{4} &, \cr
,1 & ,& \bar{3} & \bar{5} &,  &  \bar{1} &, \cr
 ,0&,& 2 & 3 & \bullet & \bullet & \bullet \cr
 ,-1 &,& ,{\bf 0} & 3 & ,{\bf 2} & 7 & 1 \cr
  ,-2 &,& ,& 4 &, & ,{\bf 3} & 1 \cr
  ,-3 &, &, &,{\bf 1} &, &, &, {\bf 4} \cr
    ,&,& ,& ,& ,& ,&, \cr
     ,&,& ,1& ,2& ,3& ,4& ,5\cr
\end{Young}
\end{equation*}
The unbarred $\coinv$ labels give a bijection
\begin{multline}
\iota^{\coinv}_N: 
\left\{  \begin{array}{c}\text{$K$-column tableaux $\tau''$ with $k-s$ $\bullet$'s} \\
\text{such that all entries of $\tau''$ are $\leq N$ and}\\ 
\text{no $N$'s in $\tau''$ have negative height}\end{array} \right\}  \times 
\left\{
\begin{array}{c}
\text{finite multisets $S$} \\ \text{drawn from $\{0,1,\dots,K-1\}$}
\end{array} \right\} \longrightarrow \\
\left\{ \begin{array}{c} \text{$K$-column tableaux $\tau'''$ with $k-s$ $\bullet$'s} \\
\text{such that all entries in $\tau'''$ are $\leq N$} \end{array} \right\}.
\end{multline}
where $\iota^{\coinv}_N(\tau'',S)$ is obtained by placing $m_i$ copies of $N$ below the column
with label $i$, where $m_i$ is the multiplicity of $i$ in $S$. If 
$\iota^{\coinv}_N: (\tau'',S) \mapsto \tau'''$ then $\coinv(\tau''') = \coinv(\tau'') + \sum_{i \in S} i$.

The effect of inserting $N$'s at negative height on $\codinv$ may be described as follows.
We label the columns of $\tau$ bijectively with the {\em unbarred $\codinv$ labels} $0,1,\dots,k-1$
(in that order) starting at columns of lesser maximal depth and, within columns of the same 
maximal depth, proceeding from right to left.
The unbarred $\codinv$ labels in our example as shown below.
\begin{equation*}
\begin{Young}
,2& ,&  \bar{4} & \bar{6} &, & \bar{4} &, \cr
,1 & ,& \bar{3} & \bar{5} &,  & \bar{1} &, \cr
 ,0&,& 2 & 3 & \bullet & \bullet & \bullet \cr
 ,-1 &,&,{\bf 1} & 3 &, {\bf 0} & 7 & 1 \cr
  ,-2 &,&, & 4 & ,&,{\bf 2} & 1 \cr
  ,-3 &,&, &, {\bf 4} &, &, &, {\bf 3} \cr
    ,&,& ,& ,& ,& ,&, \cr
     ,&,& ,1& ,2& ,3& ,4& ,5\cr
\end{Young} 
\end{equation*}
The unbarred $\codinv$ labels give a bijection
\begin{multline}
\iota^{\codinv}_N: 
\left\{  \begin{array}{c}\text{$K$-column tableaux $\tau''$ with $k-s$ $\bullet$'s} \\
\text{such that all entries of $\tau''$ are $\leq N$ and}\\ 
\text{no $N$'s in $\tau''$ have negative height}\end{array} \right\}  \times 
\left\{
\begin{array}{c}
\text{finite multisets $S$} \\ \text{drawn from $\{0,1,\dots,K-1\}$}
\end{array} \right\} \longrightarrow \\
\left\{ \begin{array}{c} \text{$K$-column tableaux $\tau'''$ with $k-s$ $\bullet$'s} \\
\text{such that all entries in $\tau'''$ are $\leq N$} \end{array} \right\}
\end{multline}
as follows. Given $(\tau'',S)$, we process the entries of $S$ {\bf in weakly increasing order}
and, given an entry $i$, we place an $N$ at the bottom of the column with 
unbarred $\codinv$ label $i$. 
Recalling that we consider every unfilled cell below a column to have the label $+ \infty$
for the purpose of calculating $\codinv$, we see that if
$\iota^{\codinv}_N: (\tau'',S) \mapsto \tau'''$ then 
$\codinv(\tau''') = \codinv(\tau'') + \sum_{i \in S} i$.
\end{proof}

Although the $D$-functions are more directly seen to be symmetric and 
Schur positive, in order to describe a recursive formula for the action of 
$h_j^{\perp}$ it will be more convenient to use the $C$-functions instead 
(Lemma~\ref{c-skewing-lemma}).
This formula will involve the more refined 
$C_{n,k,s}^{(r)}(\xx;q)$  rather than their coarsened versions
$C_{n,k,s}(\xx;q,z)$.

\subsection{Substaircase shuffles}
Given $\sigma \in \OSP_{n,k,s}$,  the coinversion code
$\code(\sigma) = (c_1, \dots, c_n)$ is a sequence of length $n$ 
over the alphabet $\{0,1,2,\dots,\bar{0},\bar{1},\bar{2},\dots\}$.
The purpose of this subsection is to show that the map $\sigma \mapsto \code(\sigma)$
which sends $\sigma \in \OSP_{n,k,s}$ to its code is injective, and to characterize its image.

The image of the classical coinversion code map on permutations in $\symm_n$ is given by  
words $(c_1, \dots, c_n)$ which are componentwise $\leq$ the `staircase' word
$(n-1, n-2, \dots, 1, 0)$.
For ordered set partitions of $[n]$ into $k$ blocks with no barred letters, this result 
was generalized in \cite{RWLine} where  the appropriate notion of `staircase'
is given by a shuffle of two sequences.
We extend these definitions to barred letters as follows.

Recall that a {\em shuffle} of two sequence $(a_1, \dots, a_r)$ and $(b_1, \dots, b_s)$ is an 
interleaving $(c_1, \dots, c_{r+s})$ of these sequences which preserves the relative order of the $a$'s and the 
$b$'s.  The following collection $\SSS^{(r)}_{n,k,s}$ of words will turn out to be the image of the map
$\code$ on $\OSP_{n,k,s}^{(r)}$.

\begin{defn}
\label{substaircase-definition}
Let $n, k, s \geq 0$ be integers.
For $0 \leq r \leq n-s$, a {\em staircase} with $r$ barred letters is a length
$n$ word obtained by shuffling
\begin{equation*}
( \overline{k-s-1} )^{r_0} (s-1)
( \overline{k-s} )^{r_1} (s-2)
( \overline{k-s+1} )^{r_2} (s-3) \cdots 
0 ( \overline{k-1})^{r_s} \quad \text{and}
\quad (k-1)^{n-r-s},
\end{equation*}
where $r_0 + r_1 + \cdots + r_s = r$.  
If $k=s$ we insist that $r_0 = 0$ in the above expression.
Let $\SSS^{(r)}_{n,k,s}$ be the family of 
words $(c_1, \dots, c_n)$ which are $\leq$ some staircase $(b_1, \dots, b_n)$.
Here $(c_1, \dots, c_n) \leq (b_1, \dots, b_n)$ means that $c_i \leq b_i$ in value and that $c_i$ is barred
if and only if $b_i$ is barred, for all $i$.  We also let
\begin{equation}
\SSS_{n,k,s} := \SSS_{n,k,s}^{(0)} \sqcup 
\SSS_{n,k,s}^{(1)} \sqcup \cdots \sqcup
\SSS_{n,k,s}^{(n-s)}
\end{equation}
and refer to words in $\SSS_{n,k,s}$ as {\em substaircase}.
\end{defn}

Definition~\ref{substaircase-definition} implies that 
\begin{equation}
\SSS_{n,k,s} = \varnothing \quad \quad \text{if $n < s$.}
\end{equation}
We give an example to clarify these concepts.

\begin{example}
\label{substaircase-example}
Consider the case $n = 5, k = 3, s = 2,$ and $r = 2$. The staircases with $r$
barred letters are the shuffles 
of any of the six words
\begin{equation*}
( \bar{0}, \bar{0}, 1, 0), \, \, \,
(  \bar{0}, 1,  \bar{1}, 0), \, \, \,
(\bar{0}, 1, 0, \bar{2}), \, \, \,
(1, \bar{1}, \bar{1}, 0), \, \, \,
(1, \bar{1}, 0, \bar{2}), \, \, \,
(1, 0, \bar{2}, \bar{2})
\end{equation*}
with the single-letter word $(2)$.
For example, if we shuffle $(2)$ into the second sequence from the left, we get the five staircases
\begin{equation*}
(2, \bar{0}, 1, \bar{1}, 0), \, \, \, 
(\bar{0}, 2, 1, \bar{1}, 0), \, \, \, 
(\bar{0}, 1, 2, \bar{1}, 0), \, \, \,
(\bar{0}, 1, \bar{1}, 2, 0), \, \, \, 
(\bar{0}, 1, \bar{1}, 0, 2).
\end{equation*}
The leftmost sequence displayed above contributes
\begin{equation*}
(2, \bar{0}, 1, \bar{1}, 0), \, \, \,
(1, \bar{0}, 1, \bar{1}, 0), \, \, \,
(0, \bar{0}, 1, \bar{1}, 0), \, \, \,
(2, \bar{0}, 0, \bar{1}, 0), \, \, \,
(1, \bar{0}, 0, \bar{1}, 0), \, \, \,
(0, \bar{0}, 0, \bar{1}, 0), \, \, \,
\end{equation*}
\begin{equation*}
(2, \bar{0}, 1, \bar{0}, 0), \, \, \,
(1, \bar{0}, 1, \bar{0}, 0), \, \, \,
(0, \bar{0}, 1, \bar{0}, 0), \, \, \,
(2, \bar{0}, 0, \bar{0}, 0), \, \, \,
(1, \bar{0}, 0, \bar{0}, 0), \, \, \,
(0, \bar{0}, 0, \bar{0}, 0) \, \, \,
\end{equation*}
to $\SSS_{5,3,2}^{(2)}$.
These are the twelve sequences which have the same bar pattern as, and are componentwise
$\leq$ to, the staircase $(2, \bar{0}, 1, \bar{1}, 0)$.
\end{example}

The previous example shows that applying 
Definition~\ref{substaircase-definition} to obtain the set of words in $\SSS_{n,k,s}$ 
can be involved.
The following lemma gives a simple recursive definition of substaircase sequences 
of length $n$ in terms of substaircase sequences of length $n-1$.
This recursion gives an efficient way to calculate $\SSS_{n,k,s}$ and
will be useful in our 
algebraic analysis of $\WWW_{n,k,s}$ in Section~\ref{Hilbert}.
It will turn out that words in $\SSS_{n,k,s}$ index a monomial basis of 
$\WWW_{n,k,s}$.

\begin{lemma}
\label{disjoint-union-decomposition}  Let $n, k, s \geq 0$ be integers with $k \geq s$ and let
$0 \leq r \leq n-s$. 
The set $\SSS_{n,k,s}^{(r)}$ has the disjoint union decomposition
\begin{multline*}
\SSS_{n,k,s}^{(r)} = \bigsqcup_{a = 0}^{k-s-1} 
\left\{ ( \bar{a}, c_2, \dots, c_n) \,:\, (c_2, \dots, c_n) \in \SSS_{n-1,k,s}^{(r-1)} \right\} \, \sqcup \\
\bigsqcup_{a = 0}^{s-1} 
\left\{ (a, c_2, \dots, c_n) \,:\, (c_2, \dots, c_n) \in \SSS_{n-1,k,s-1}^{(r)} \right\} \, \sqcup \\
\bigsqcup_{a = s}^{k-1}
\left\{ (a, c_2, \dots, c_n) \,:\, (c_2, \dots, c_n) \in \SSS_{n-1,k,s}^{(r)} \right\}
\end{multline*}
and the set $\SSS_{n,k,s}$ has the disjoint union decomposition
\begin{multline*}
\SSS_{n,k,s}  = \bigsqcup_{a = 0}^{k-s-1} 
\left\{ ( \bar{a}, c_2, \dots, c_n) \,:\, (c_2, \dots, c_n) \in \SSS_{n-1,k,s} \right\} \, \sqcup \\
\bigsqcup_{a = 0}^{s-1}
 \left\{ (a, c_2, \dots, c_n) \,:\, (c_2, \dots, c_n) \in \SSS_{n-1,k,s-1}  \right\}  \, \sqcup \\
\bigsqcup_{a = s}^{k-1} 
\left\{ (a, c_2, \dots, c_n) \,:\, (c_2, \dots, c_n) \in \SSS_{n-1,k,s}  \right\}.
\end{multline*}
\end{lemma}

\begin{proof}
The second disjoint union decomposition for $\SSS_{n,k,s}$ follows from the first 
disjoint union decomposition for $\SSS_{n,k,s}^{(r)}$ by taking the (disjoint) union over all $r$,
so we focus on the first decomposition.

Given $(c_1,  c_2, \dots, c_n) \in \SSS_{n,k,s}^{(r)}$, there exists some word
$(b_1, b_2, \dots, b_n)$ obtained by shuffling
\begin{equation*}
( \overline{k-s-1} )^{r_0} (s-1)
( \overline{k-s} )^{r_1} (s-2)
( \overline{k-s+1} )^{r_2} (s-3) \cdots 
0 ( \overline{k-1})^{r_s} 
\end{equation*}
where $r_0 + r_1 + \cdots + r_s = r$ with the constant sequence $(k-1)^{n-r-s}$
such that $(c_1, c_2, \dots, c_n)$ has the same bar pattern as
$(b_1, b_2, \dots, b_n)$ and $c_i \leq b_i$ for all $i$. There are three possibilities for the first 
entry $c_1$.
\begin{itemize}
\item If $c_1 = \bar{a}$ is barred, then we must have $r_0 > 0$ and $b_1 = \overline{k-s-1}$.
It follows that $0 \leq a \leq k-s-1$.  Furthermore, the word $(b_2, \dots, b_n)$ is a shuffle of 
\begin{equation*}
( \overline{k-s-1} )^{r_0-1} (s-1)
( \overline{k-s} )^{r_1} (s-2)
( \overline{k-s+1} )^{r_2} (s-3) \cdots 
0 ( \overline{k-1} )^{r_s}  \quad \text{and} \quad (k-1)^{n-r-s}.
\end{equation*}
This implies that $(c_2, \dots, c_n) \in \SSS_{n-1,k,s}^{(r-1)}$. Conversely, given 
$(c_2, \dots, c_n) \in \SSS_{n-1,k,s}^{(r-1)}$ and $0 \leq a \leq k-s-1$, we see that 
$(\bar{a}, c_2, \dots, c_n) \in \SSS_{n,k,s}^{(r)}$ by prepending the letter $\overline{k-s-1}$
to any $(n-1,k,s)$-staircase $(b_2, \dots, b_n) \geq (c_2, \dots, c_n)$.
\item  If $c_1 = a$ is unbarred with $0 \leq a \leq s - 1$, the first letter $b_1$ of $(b_1, b_2, \dots, b_n)$ 
must also be unbarred. This means that $b_1 = s-1$ or $k-1$.
If $k > s$ and $b_1 = k-1$, we may interchange the $b_1$ with the unique occurrence of 
$s-1$ in $(b_1, b_2, \dots, b_n)$ to get a new staircase $(b'_1, b'_2, \dots, b'_n)$ which
shares the color pattern of $(c_1, c_2, \dots, c_n)$ and satisfies $c_i \leq b'_i$ for all $i$.
We may therefore assume that $b_1 = s-1$.  This means that the sequence $(b_2, \dots, b_n)$ 
is a shuffle of the words
\begin{equation*}
( \overline{k-s} )^{r_1} (s-2)
( \overline{k-s+1} )^{r_2} (s-3) \cdots 
0 ( \overline{k-1} )^{r_s}  \quad \text{and} \quad (k-1)^{n-r-s}
\end{equation*}
and $r_1 + r_2 + \cdots + r_s = r$. Therefore, we have 
$(c_2, \dots, c_n) \in \SSS_{n-1,k,s-1}^{(r)}$. 
Conversely, given $(c_2, \dots, c_n) \in \SSS_{n-1,k,s-1}^{(r)}$, one sees that
$(a, c_2, \dots, c_n) \in \SSS_{n,k,s}^{(r)}$ by prepending an $s-1$ to any 
$(n-1,k,s-1)$-staircase $(b_2, \dots, b_n) \geq (c_2, \dots, c_n)$.
\item  Finally, if $c_1 = a$ is unbarred and $c_1 \geq s$, the first letter $b_1$ of $(b_1, b_2, \dots, b_n)$
must be unbarred and $\geq s$.  This implies that $b_1 = k-1$, so that $s \leq a < k-1$. Furthermore,
the sequence $(b_2, \dots, b_n)$ is a shuffle of the words
\begin{equation*}
( \overline{k-s-1} )^{r_0} (s-1)
( \overline{k-s} )^{r_1} (s-2)
( \overline{k-s+1} )^{r_2} (s-3) \cdots 
0 ( \overline{k-1} )^{r_s}  \quad \text{and} \quad (k-1)^{n-r-s-1}.
\end{equation*}
where $r_0 + r_1 + \cdots + r_s = r$. This means that $(c_2, \dots, c_n) \in \SSS_{n-1,k,s}^{(r)}$.
Conversely, for any $(c_2, \dots, c_n) \in \SSS_{n-1,k,s}^{(r)}$ we see that 
$(a, c_2, \dots, c_n) \in \SSS_{n,k,s}^{(r)}$ by prepending a $k-1$ to any 
$(n-1,k,s)$-staircase $(b_2, \dots, b_n) \geq (c_2, \dots, c_n)$.
\end{itemize}
The three bullet points above show that $\SSS_{n,k,s}^{(r)}$ is a union of the claimed sets of words.
The disjointness of this union follows since the first letters of the words in these sets are distinct.
\end{proof}

We are ready to state the main result of this subsection: 
the codes of ordered set superpartitions are precisely the substaircase words.
The key idea of the proof is to invert the map
$\code: \sigma \mapsto (c_1, \dots, c_n)$ sending an ordered set partition to 
its coinversion code. The inverse map
$\iota: (c_1, \dots, c_n) \mapsto \sigma$ is a variant on the insertion maps
 in the proof of Theorem~\ref{c-equals-d}.

\begin{theorem}
\label{code-is-bijection}
Let $n, k, s \geq 0$ be integers with $k \geq s$ and 
 and let $r \leq n-k$.
The coinversion code map gives a well-defined bijection
\begin{equation*}
\code: \OSP_{n,k,s}^{(r)} \longrightarrow \SSS_{n,k,s}^{(r)}.
\end{equation*}
\end{theorem}

\begin{proof}
Our first task is to show that the function $\code$ is well-defined, i.e. that $\code(\sigma) \in \SSS_{n,k,s}^{(r)}$
for any $\sigma \in \OSP_{n,k,s}^{(r)}$. To this end, let 
$\sigma = (B_1 \mid \cdots \mid B_k) \in \OSP_{n,k,s}^{(r)}$ with 
$\code(\sigma) = (c_1, \dots, c_n)$. We associate a staircase $(b_1, \dots, b_n)$ to $\sigma$ as follows.
Write the minimal elements $\min B_1, \dots, \min B_s$ in increasing order $i_1 < \cdots < i_s$
and set $b_{i_j} := s-j$. If $1 \leq i \leq n$ is barred in $\sigma$, write
$b_i = \overline{k-s-1+m}$ where $m = | \{ 1 \leq j \leq s \,:\, i_j < i \} |$. 
Finally, if $1 \leq i \leq n$ and $i$ is unbarred in $\sigma$ and not minimal in any of the first $s$ blocks
$B_1, \dots, B_s$, set $b_i = k-1$.

As an example of these concepts, consider
$\sigma = (5, 7 \mid 1 \mid 3, \bar{4}, \bar{8} \mid \varnothing \mid \bar{2}, 6) \in \OSP_{9,5,3}^{(3)}$.
The associated staircase  
is
$(b_1, \dots, b_8) = (2, \bar{2}, 1, \bar{3}, 0, 4, 4, \bar{4})$.
We have $\code(\sigma) = (c_1, \dots, c_8) = (1, \bar{2}, 0, \bar{1}, 0, 4, 0, \bar{2})$, which
has the same bar pattern as, and is componentwise $\leq$, the sequence $(b_1, \dots, b_8)$.

We leave it for the reader to verify that $(c_1, \dots, c_n)$ has the same bar pattern as $(b_1, \dots, b_n)$
and $(c_1, \dots, c_n) \leq (b_1, \dots, b_n)$ componentwise.
This shows that the function $\code$ in the statement is well-defined.

We show that $\code$ is a bijection by constructing its inverse map
\begin{equation}
\iota: \SSS_{n,k,s}^{(r)} \longrightarrow \OSP_{n,k,s}^{(r)}
\end{equation}
as follows. The map $\iota$ starts with a sequence $(\varnothing \mid \cdots \mid \varnothing)$
of $k$ copies of the empty set
and builds up an ordered set superpartition by insertion.

Given a sequence $(B_1 \mid \cdots \mid B_s)$ of $s$ possibly empty subsets of 
the alphabet $\{1, 2, \dots, \bar{1}, \bar{2}, \dots \}$, we assign the blocks $B_i$ the 
{\em unbarred labels} $0, 1, 2, \dots, k-1$ as follows. Moving from right to left, we assign the labels
$0, 1, \dots, j-1$ to the $j$ empty blocks among $B_1, \dots, B_s$. Then, moving from left to right,
we assign the unlabeled blocks the labels $j, j+1, \dots, k-1$.  An example of unbarred labels when
 $k = 7$ and $s = 4$ is as follows:
\begin{equation*}
(   3, \bar{4} \,_2 \mid  \varnothing \,_1  \mid  \varnothing \,_0   \mid  1, 5 \,_3 \mid  \varnothing \,_4  \mid  \bar{2} \,_5 \mid \varnothing \,_6 );
\end{equation*}
we draw unbarred labels below their blocks. The  {\em barred} labels $\bar{0}, \bar{1}, \dots$
as assigned
to the blocks $B_i$ where either $i > s$ or $B_i \neq \varnothing$ by moving left to right.
In our example, the barred labels are
\begin{equation*}
(   3, \bar{4} \,^{\bar{0}}  \mid  \varnothing   \mid  \varnothing    \mid  1, 5 \,^{\bar{1}}  \mid  \varnothing \,^{\bar{2}}   \mid  \bar{2} \,^{\bar{3}}  \mid \varnothing \,^{\bar{4}}  ),
\end{equation*}
where the blocks $B_2 = B_3 = \varnothing$ do not receive a barred label because $2, 3 \leq s = 4$.
Barred labels are written above their blocks.

Let $(c_1, \dots, c_n) \in \SSS_{n,k,s}^{(r)}$. To define $\iota(c_1, \dots, c_n)$, we start with the 
sequence $(B_1 \mid \cdots \mid B_k) = (\varnothing \mid \cdots \mid \varnothing)$ of $s$ empty blocks
and iteratively insert $i$ into the block with unbarred label $c_i$ (if $c_i$ is unbarred)
or insert $\bar{i}$ into the block with barred label $c_i$ (if $c_i$ is barred).
For example, if $(n,k,s) = (8,5,3)$ and $(c_1, \dots, c_8) = (1, \bar{2}, 0, \bar{1}, 0, 4, 0, \bar{2})$,
we perform the following insertion procedure.
\begin{center}
\[\def\arraystretch{1.5}
\begin{tabular}{c | c | c}
$i$ & $c_i$ & $(B_1 \mid \cdots \mid B_s)$ \\ \hline 
$1$ & $1$ & $(\varnothing \,_1 \mid 1 \,_2 \,^{\bar{0}} \mid \varnothing \,_0 \mid \varnothing \,_3 \,^{\bar{1}} \mid
\varnothing \,_4 \,^{\bar{2}})$ \\
$2$ & $\bar{2}$ & $(\varnothing \,_1 \mid 1 \,_2 \,^{\bar{0}} \mid \varnothing \,_0 \mid \varnothing \,_3 \,^{\bar{1}} \mid
\bar{2} \,_4 \,^{\bar{2}})$ \\
$3$ & $0$ & $(\varnothing \,_0 \mid 1 \,_1 \,^{\bar{0}} \mid 3 \,_2 \,^{\bar{1}} \mid \varnothing \,_3 \,^{\bar{2}} \mid
\bar{2} \,_4 \,^{\bar{3}})$ \\
$4$ & $\bar{1}$ & 
$(\varnothing \,_0 \mid 1 \,_1 \,^{\bar{0}} \mid 3, \bar{4} \,_2 \,^{\bar{1}} \mid \varnothing \,_3 \,^{\bar{2}} \mid
\bar{2} \,_4 \,^{\bar{3}})$ \\
$5$ & $0$ & 
$(5 \,_0 \,^{\bar{0}} \mid 1 \,_1 \,^{\bar{1}} \mid 3, \bar{4} \,_2 \,^{\bar{2}} \mid \varnothing \,_3 \,^{\bar{3}} \mid
\bar{2} \,_4 \,^{\bar{4}})$ \\
$6$ & $4$ & 
$(5 \,_0 \,^{\bar{0}} \mid 1 \,_1 \,^{\bar{1}} \mid 3, \bar{4} \,_2 \,^{\bar{2}} \mid \varnothing \,_3 \,^{\bar{3}} \mid
\bar{2}, 6 \,_4 \,^{\bar{4}})$ \\
$7$ & $0$ & 
$(5, 7 \,_0 \,^{\bar{0}} \mid 1 \,_1 \,^{\bar{1}} \mid 3, \bar{4} \,_2 \,^{\bar{2}} \mid \varnothing \,_3 \,^{\bar{3}} \mid
\bar{2}, 6 \,_4 \,^{\bar{4}})$ \\
$8$ & $\bar{2}$ & 
$(5, 7 \,_0 \,^{\bar{0}} \mid 1 \,_1 \,^{\bar{1}} \mid 3, \bar{4}, \bar{8} \,_2 \,^{\bar{2}} \mid \varnothing \,_3 \,^{\bar{3}} \mid
\bar{2}, 6 \,_4 \,^{\bar{4}})$ \\
\end{tabular}
\]
\end{center}
The above table shows that 
\begin{equation*}
\iota: (1, \bar{2}, 0, \bar{1}, 0, 4, 0, \bar{2}) \mapsto  
(5, 7 \mid 1 \mid 3, \bar{4}, \bar{8} \mid \varnothing \mid \bar{2}, 6).
\end{equation*}

To verify that $\iota$ is well-defined, we need to check that for any $(c_1, \dots, c_n) \in \SSS_{n,k,s}^{(r)}$,
the sequence $\iota(c_1, \dots, c_n) = (B_1 \mid \cdots \mid B_k)$ of sets is a valid element
of $\OSP_{n,k,s}^{(r)}$.  That is, the first $s$ sets $B_1, \dots, B_s$ must be non-empty.
Indeed, there are $s$ indices $1 \leq i_1 < \cdots < i_s \leq n$
such that $c_{i_1}, \dots, c_{i_s}$ are unbarred and $c_{i_j} \leq s-j$.
From the definition of our labeling it follows that, for any $j$, at least $j$ minimal elements of 
the nonempty sets in the list 
$B_1, \dots, B_s$ are $\leq i_j$.  Taking $j = s$, we see that all $s$ of the sets $B_1, \dots, B_s$ 
are nonempty.
The fact that $\code$ and $\iota$ are mutually inverse follows from the definition of our labeling.
\end{proof}

\subsection{Skewing Formula}
To show that the function $C_{n,k,s}(\xx;q,z)$ is the bigraded 
Frobenius image $\grFrob(\WWW_{n,k,s}; q, z)$, we will show that image of both of these symmetric
functions under the operator $h_j^{\perp}$ for $j \geq 1$ satisfy the same recursive formula. 
Lemma~\ref{skew-by-e-lemma} will then imply that these functions coincide.

We need a better understanding of the rings $\WWW_{n,k,s}$ before we prove an
$h_j^{\perp}$-recursion recursion for them,
but we can prove the relevant recursion for the $C$-functions now (using their 
$\coinv$ formulation).
This recursion is stated more naturally in terms of the more refined 
$C_{n,k,s}^{(r)}(\xx;q)$ functions rather than their coarsened versions
$C_{n,k,s}(\xx;q,z)$.

\begin{lemma}
\label{c-skewing-lemma}
Let $n, k, s \geq 0$ with $k \geq s$ and let $0 \leq r \leq n-s$. For any $j \geq 1$ we have
\begin{multline}
\label{goal-c-skewing}
h_j^{\perp} C_{n,k,s}^{(r)}(\xx;q) = \\
\sum_{\substack{0 \leq a,b \leq j \\ a \leq r, \, b \leq s}}
q^{{j-a-b \choose 2} + (s-b)a} \times 
{k-s-1+a+b \brack a}_q \cdot {s \brack b}_q \cdot {k-s \brack j-a-b}_q \cdot 
C^{(r-j+a+b)}_{n-j,k,s-b}(\xx;q).
\end{multline}
\end{lemma}

\begin{proof}
It is well-known that the homogeneous and monomial symmetric functions are dual bases 
of the ring of symmetric functions. That is, we have
\begin{equation}
\langle h_{\lambda}, m_{\mu} \rangle = \begin{cases} 1 & \lambda = \mu \\ 0 & \lambda \neq \mu \end{cases}
\end{equation}
for any partitions $\lambda$ and $\mu$. Therefore, if $\alpha = (\alpha_1, \alpha_2, \dots )$ is any
composition of $n$ with $h_{\alpha} := h_{\alpha_1} h_{\alpha_2} \cdots $, we have
\begin{equation}
\label{basic-skewing-identity}
\langle h_{\alpha}, C_{n,k,s}^{(r)}(\xx;q) \rangle = \sum_{\tau} q^{\coinv(\tau)},
\end{equation}
where the sum is over all column tableaux $\tau \in \CT_{n,k,s}^{(r)}$ with $\alpha_i$ copies of $i$
for all $i$.
Specializing Equation~\eqref{basic-skewing-identity} at $\alpha_1 = j$ and applying the adjoint
property gives
\begin{equation}
\label{adjoint-skewing-identity}
\langle h_{\alpha'}, h_j^{\perp} C_{n,k,s}^{(r)}(\xx;q) \rangle = \sum_{\tau} q^{\coinv(\tau)}
\end{equation}
where $\alpha' = (\alpha_2, \alpha_3, \dots )$ is the composition of $n-j$ obtained by removing $\alpha_1$
from $\alpha$ and the right-hand sides of Equations~\eqref{basic-skewing-identity}
and \eqref{adjoint-skewing-identity} are the same. The strategy is to show that 
the right-hand sides of Equation~\eqref{adjoint-skewing-identity} and
Equation~\eqref{goal-c-skewing} coincide.
This will be achieved combinatorially with a procedure which
inserts $j$ new \textbf{smallest} entries in a column tableau.
Our insertion process will be somewhat more 
involved than the similar procedure in the proof of Theorem~\ref{c-equals-d}, 
where we inserted new largest entries instead.

Fix nonnegative integers $j$, $a$, and $b$ satisfying $0 \leq a, b \leq j$, $a \leq r$, and $b \leq s$. Consider a column tableau $\sigma \in \CT_{n-j,k,s-b}^{(r-j+a+b)}$. We increment every entry in $\sigma$ to obtain another filling $\tau \in \CT_{n-j,k,s-b}^{(r-j+a+b)}$. We will, in total, introduce $j$ 1's and 
$\bar{1}$'s into the tableau $\tau$. Since we must add $j-a-b$ new $\bar{1}$'s, $b$ $1$'s at height 0, and $j$ $1$'s overall to obtain a tableau in $\CT_{n,k,s}^{(r)}$, we will need to add $a$ unbarred $1$'s at negative heights. We begin by inserting these $a$ negative height $1$'s.

Since the unbarred entries must be weakly decreasing down each column, at this stage we can only insert a negative height 1 by placing it just below one of the $k-(s-b)$ $\bullet$'s. If we insert such an entry into column $i$, it will form exactly $i-1$ new coinversions, one with every height 0 entry to its left. In other words, we have a bijection
\begin{multline}
\xi_{<0}: 
\left\{ \tau \in \CT_{n-j,k,s-b}^{(r-j+a+b)} : \text{$\tau$ has no 1's} \right\} \times \\
\left\{\text{finite multisets $S$ of size $a$ drawn from $\{s-b,s-b+1,\dots,k-1\}$} \right\} \longrightarrow \\
\left\{ \tau' \in \CT_{n-j+a,k,s-b}^{(r-j+a+b)} : \tau' \text{ has $a$ $1$'s at negative heights and no other 1's} \right\}
\end{multline}
where $\tau' = \xi_{< 0}(\tau, S)$ is obtained by iteratively placing a 1 at height $-1$ in column $i+1$ for each occurrence of $i$ in $S$. Since $\coinv(\tau') = \coinv(\tau) + \sum_{i \in S} i$, this map contributes $q^{(s-b) a} {k-s-1+a+b \brack a}_q $ to \eqref{adjoint-skewing-identity}.

Next, we insert $b$ 1's at height 0. This will be the most complicated part of the procedure, as we need to be careful not to violate the necessary inequalities in each column. Our map will be a bijection 
\begin{multline}
\xi_{0}: 
\left\{ \tau' \in \CT_{n-j+a,k,s-b}^{(r-j+a+b)} : \tau' \text{ has $a$ $1$'s at negative heights and no other 1's} \right\} \\ 
\left\{\text{finite multisets $S$ of size $b$ drawn from $\{0,1,\ldots,s-b\}$} \right\} \longrightarrow \\
\left\{ \tau'' \in \CT_{n-j+a+b,k,s}^{(r-j+a+b)} : \tau'' \text{ has $a$ $1$'s at negative heights, $b$ 1's at height 0, and no other 1's} \right\} .
\end{multline}
Given $\tau'$ and $S$, we begin by simply replacing the $b$ leftmost $\bullet$'s in $\tau'$ with 1's. Since there are $k-(s-b)$ bullets and $k \geq s$, this can always be done. Furthermore, since $\tau'$ has no $\bar{1}$'s, this replacement does not violate any necessary column inequalities. Suppose $S = \{i_1 \geq \ldots \geq i_b\}$. Initialize $j=1$. For $j = 1$ to $b$, we repeat the following algorithm $i_j$ times:
\begin{enumerate}
\item Suppose the $j^{\text{th}}$ 1 at height 0 is in column $\ell$ and let $m > 1$ be the height 0 entry in column $\ell-1$. 
\item Switch the height 0 entries in columns $\ell$ and $\ell -1$.
\item Move every barred entry $\leq m$ in column $\ell$ and every unbarred entry $< m$ in column $\ell$ to column $\ell-1$ (moving entries up and down as necessary so that the diagram is justified as usual). 
\end{enumerate}
We claim that each loop of this algorithm increments coinv, yields a valid column tableau, and is invertible. The latter two statements are true by construction. Clearly each loop 
creates one new coinversion between the 1 at height 0 and the entry $m$. We check that this is the only change to the total number of coinversions.
\begin{itemize}
\item If $u$ is some other height 0 entry, $u$ has the same number of height 0 entries to its right which are $> u$. 
\item Suppose $u$ is at negative height. The only case that might have affected $u$'s contribution to coinv is if $u< m$ starts in column $\ell$. Then $u$ is moved to column $\ell-1$ but, due to  $m$, there is a new height 0 entry to the right of $u$ that is $> u$. 
\item Suppose $u$ is at positive height. If $u \leq m$, then $m$ and $u$ did not contribute a coinversion so moving $u$ to the left does not affect the total coinv. If $u > m$, then $u > 1$ and $u$ contributes the same number of coinversions.
\end{itemize}
After repeating this procedure $i_j$ times for $j = 1$ to $b$, we have $\coinv(\tau'') = \coinv(\tau') + \sum_{i \in S} i$, yielding the ${s \brack b}_q$ term in \eqref{c-skewing-lemma}.

Finally, we insert the $j-a-b$ $\bar{1}$'s into $\tau''$. A $\bar{1}$ can only be placed above a $\bullet$ in $\tau''$, of which there are now $k-s$, and only at most once in each column. Placing $\bar{1}$ over the $i^{\text{th}}$ $\bullet$ from the left creates $i-1$ new coinversions, one with each bullet strictly to the left. Therefore we have a bijection
\begin{multline}
\xi_{>0}: 
\left\{ \tau'' \in \CT_{n-j+a+b,k,s}^{(r-j+a+b)} : \tau'' \text{ has $a$ $1$'s at negative heights, $b$ 1's at height 0, and no other 1's} \right\} \times \\
\left\{\text{finite sets $S$ of size $j-a-b$ drawn from $\{0,1,\ldots,k-s-1\}$} \right\} \longrightarrow \\
\left\{ \tau''' \in \CT_{n,k,s}^{(r)} : \tau''' \text{ has $a$ $1$'s at negative heights, $b$ 1's at height 0, $j-a-b$ $\bar{1}$'s} \right\}
\end{multline}
which places a $\bar{1}$ above $\bullet$ number $i+1$ from the left for every $i \in S$. This bijection satisfies $\coinv(\tau''') = \coinv(\tau'') + \sum_{i \in S} i$ and contributes the factor $q^{\binom{j-a-b}{2}} {k-s \brack j-a-b}_q$ to \eqref{c-skewing-lemma}.

In Figure \ref{fig:c-skewing}, we depict an example for $n = 18$, $k=5$, $s=4$, $j = 6$, $a=3$, $b=2$, where we have already incremented each entry so that no 1's or $\bar{1}$'s appear.  When applying $\xi_{< 0}$, $\xi_0$, and $\xi_{>0}$ we take $S = \{2,2,4\}$, $\{1,2\}$, and $\{0\}$, respectively. The initial tableau has 17 coinv:
\begin{itemize}
\item 1 at height 0, between the 3 and the 4,
\item $1+1+3+4+4=13$ from negative height entries, all coming from the column of each of these entries, and
\item 3 from positive height entries (the  3 at height 0  with $\bar{6}$ and $\bar{7}$, and the $\bar{2}$ with the $\bullet$ to its left)
\end{itemize}
After the maps in Figure \ref{fig:c-skewing}, the increases in coinv correspond to the sums of the elements in each $S$. Respectively, the coinv increases to $25 = 17 + (2+2+4)$, then $28=25+(1+2)$, and then $28=28 + 0$.
 \end{proof}

\begin{figure}
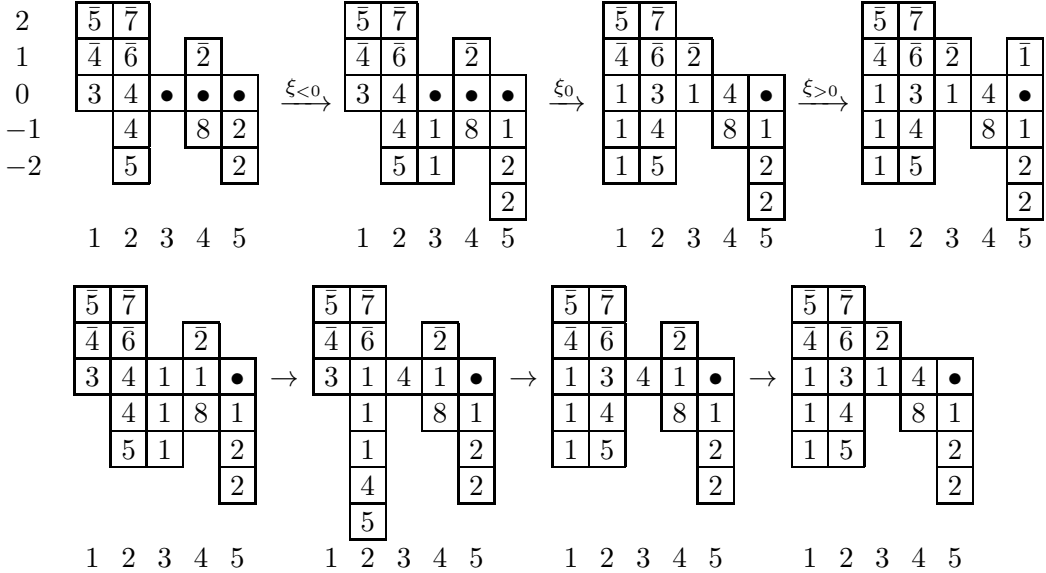

\begin{center}
\begin{Young}
,2& ,& $\bar{5}$ & $\bar{7}$ & ,& , &, \cr
,1 & ,& $\bar{4}$ & $\bar{6}$ &,  &  $\bar{2}$ &, \cr
 ,0&,& 3 & 4 & $\bullet$ & $\bullet$ & $\bullet$ \cr
 ,$-1$ &,& ,& 4 &, & 8 & 2 \cr
  ,$-2$ &,& , & 5 & , & , & 2 \cr
    ,&,& ,& ,& ,& ,&, \cr
     ,&,& ,1& ,2& ,3& ,4& ,5\cr
\end{Young}
\
\begin{Young}
, \cr
,\cr 
,$\xrightarrow{\xi_{< 0}}$  \cr
, \cr
, \cr
, \cr
, \cr
\end{Young}
\ \ 
\begin{Young}
$\bar{5}$ & $\bar{7}$ & ,& , &, \cr
$\bar{4}$ & $\bar{6}$ &,  &  $\bar{2}$ &, \cr
3 & 4 & $\bullet$ & $\bullet$ & $\bullet$ \cr
, & 4 &1 & 8 & 1 \cr
, & 5 & 1 & , & 2 \cr
,& ,& ,& ,&2 \cr
 ,1& ,2& ,3& ,4& ,5\cr
\end{Young}
\
\begin{Young}
, \cr
,\cr 
,$\xrightarrow{\xi_{ 0}}$  \cr
, \cr
, \cr
, \cr
, \cr
\end{Young}
\
\begin{Young}
$\bar{5}$ & $\bar{7}$ & ,& , &, \cr
$\bar{4}$ & $\bar{6}$ & $\bar{2}$  & , &, \cr
1 & 3 & 1 & 4 & $\bullet$ \cr
1 & 4 & , & 8 & 1 \cr
1 & 5 & , & , & 2 \cr
,& ,& ,& ,&2 \cr
 ,1& ,2& ,3& ,4& ,5\cr
\end{Young}
\begin{Young}
, \cr
,\cr 
,$\xrightarrow{\xi_{ >0}}$  \cr
, \cr
, \cr
, \cr
, \cr
\end{Young}
\ \
\begin{Young}
$\bar{5}$ & $\bar{7}$ & ,& , &, \cr
$\bar{4}$ & $\bar{6}$ & $\bar{2}$  & , &$\bar{1}$ \cr
1 & 3 & 1 & 4 & $\bullet$ \cr
1 & 4 & , & 8 & 1 \cr
1 & 5 & , & , & 2 \cr
,& ,& ,& ,&2 \cr
 ,1& ,2& ,3& ,4& ,5\cr
\end{Young}
\\\vspace{10pt}
\begin{Young}
$\bar{5}$ & $\bar{7}$ & ,& , &, \cr
$\bar{4}$ & $\bar{6}$ &,  &  $\bar{2}$ &, \cr
3 & 4 & 1& 1 & $\bullet$ \cr
, & 4 &1 & 8 & 1 \cr
, & 5 & 1 & , & 2 \cr
,& ,& ,& ,&2 \cr
,& ,& ,& ,&, \cr
 ,1& ,2& ,3& ,4& ,5\cr
\end{Young}
\begin{Young}
, \cr 
,$\rightarrow$  \cr
, \cr
, \cr
, \cr
, \cr
, \cr
\end{Young}
\begin{Young}
$\bar{5}$ & $\bar{7}$ & ,& , &, \cr
$\bar{4}$ & $\bar{6}$ &,  &  $\bar{2}$ &, \cr
3 & 1 & 4& 1 & $\bullet$ \cr
, & 1 & ,& 8 & 1 \cr
, & 1 & , & , & 2 \cr
,& 4& ,& ,&2 \cr
, & 5& ,& ,& ,\cr
 ,1& ,2& ,3& ,4& ,5\cr
\end{Young}
\begin{Young}
, \cr 
,$\rightarrow$  \cr
, \cr
, \cr
, \cr
, \cr
, \cr
\end{Young}
\begin{Young}
$\bar{5}$ & $\bar{7}$ & ,& , &, \cr
$\bar{4}$ & $\bar{6}$ &,  &  $\bar{2}$ &, \cr
1 & 3 & 4& 1 & $\bullet$ \cr
1& 4 & ,& 8 & 1 \cr
1& 5 & , & , & 2 \cr
,& ,& ,& ,&2 \cr
,& ,& ,& ,&, \cr
 ,1& ,2& ,3& ,4& ,5\cr
\end{Young}
\begin{Young}
,\cr 
,$\rightarrow$  \cr
, \cr
, \cr
, \cr
, \cr
,\cr
\end{Young}
\begin{Young}
$\bar{5}$ & $\bar{7}$ & ,& , &, \cr
$\bar{4}$ & $\bar{6}$ &$\bar{2}$ & , &, \cr
1 & 3 & 1& 4 & $\bullet$ \cr
1& 4 & ,& 8 & 1 \cr
1& 5 & , & , & 2 \cr
,& ,& ,& ,&2 \cr
,& ,& ,& ,&, \cr
 ,1& ,2& ,3& ,4& ,5\cr
\end{Young}
\end{center}
\caption{We show an example of the bijections described in the proof of Lemma \ref{c-skewing-lemma}. We take $S = \{2,2,4\}$, $\{1,2\}$, and $\{0\}$ in the three maps, respectively. On the second line, the iterations making up $\xi_0$ are shown in finer detail.}
\label{fig:c-skewing}
\end{figure}

\section{Quotient Presentation, Monomial Basis, and Hilbert Series}
\label{Hilbert}

\subsection{The superlex order on monomials}
We identify length $n$ words over $\{0, 1, 2, \dots, \bar{0}, \bar{1}, \bar{2}, \dots \}$
with monomials in $\Omega_n$ by 
\begin{equation}
(c_1, \dots, c_n)   \leftrightarrow
x_1^{c_i} \cdots x_n^{c_n} \times \theta_I \quad \quad
(I = \{ 1 \leq i \leq n \,:\, \text{$c_i$ is barred} \}).
\end{equation}
Here we adopt the  shorthand notation
$\theta_I := \theta_{i_1} \cdots \theta_{i_k}$ where $I = \{i_1 < \cdots < i_k \} \subseteq [n]$.
For example, inside the ring $\Omega_5$ we have
\begin{equation*}
(2, \bar{2}, \bar{0}, 0, \bar{1}) \leftrightarrow
(x_1^2 x_2^2 x_5) \times (\theta_{235}).
\end{equation*}
We sometimes compress this expression further by using exponent notation for the $x_i$ variables, e.g.\ 
$$ (x_1^2 x_2^2 x_5) \times (\theta_{235}) = x^{22001} \theta_{235}.$$
We refer to the word $(c_1, \dots, c_n)$ as the {\em exponent sequence}
of its corresponding superspace monomial in $\Omega_n$. The monomials corresponding
to substaircase sequences will be important.

\begin{defn}
\label{m-monomial-defn}
Let $n, k, s \geq 0$ be integers with $k \geq s$ 
and let $r \leq n-s$.  We let $\MMM_{n,k,s}$ (resp. $\MMM_{n,k,s}^{(r)}$)
be the family of monomials
in $\Omega_n$ corresponding to the substaircase sequences in
$\SSS_{n,k,s}$ (resp. $\SSS_{n,k,s}^{(r)}$).
\end{defn}

Orders on words $(c_1, \dots, c_n)$
give orders on monomials in $\Omega_n$.  The following total order will be crucial in our work.

\begin{defn}
\label{lex-order-defn}
{\em (Superlex Order)}
Let $a = (a_1, \dots, a_n)$ and $b = (b_1, \dots, b_n)$ be two length $n$ words over 
the alphabet
$\{0,1,2, \dots, \bar{0}, \bar{1}, \bar{2}, \dots \}$.
We write $a \prec b$ if there exists an index $1 \leq j \leq n$ such that
\begin{itemize}
\item $a_i = b_i$ for all $i < j$, and
\item $a_j < b_j$ under the order
\begin{equation*}
\cdots < \bar{3} < \bar{2} < \bar{1} < \bar{0} < 0 < 1 < 2 < 3 < \cdots
\end{equation*}
on individual letters.
\end{itemize}
We say that two monomials $m, m' \in \Omega_n$ satisfy $m \prec m'$ if their 
exponent sequences satisfy this relation.
Given any nonzero superspace element $f \in \Omega_n$, we let $\initial(f)$ be the $\prec$-maximal
superspace monomial which appears in $f$ with nonzero coefficient.
\end{defn}

In the absence of $\theta$-variables, superspace monomials in $\Omega_n$ reduce to 
classical monomials in $\QQ[x_1, \dots, x_n]$ and 
Definition~\ref{lex-order-defn} is the lexicographical order on monomials in this ring.
The classical lex order is a {\em monomial order} in the sense of Gr\"obner theory:
\begin{itemize}
\item there is no infinite descending chain $m_1 \succ m_2 \succ m_3 \succ \cdots$
of monomials in $\QQ[x_1, \dots, x_n]$ (i.e. $\prec$ is a {\em well order} on these 
monomials) and
\item if $m_1, m_2,$ and $m_3$ are monomials in $\QQ[x_1, \dots, x_n]$ with
$m_1 \prec m_2$, then $m_1 m_3 \prec m_2 m_3$.
\end{itemize}
Both of these properties fail when we introduce $\theta$-variables.
We have $\theta_1 \succ x_1 \theta_1 \succ x_1^2 \theta_1 \succ \cdots$ and 
$x_1 \prec x_1^2$ whereas $x_1 \theta_1 \succ x_1^2 \theta_1$.

Despite its failure to be a monomial
order, the superlex order $\prec$ of Definition~\ref{lex-order-defn} will be important for our work.
It is our belief that this relates to the inscrutable Gr\"obner theory of 
important superspace ideals such as the superspace coinvariant ideal in \cite{Zabrocki}.

\subsection{A dimension lower bound}
Recall that $\HHH_{n,k,s}$ is the $\Omega_n$-submodule (under the $\odot$-action)
of $\Omega_n$ generated by the Vandermonde $\delta_{n,k,s}$.
Therefore, the space $\HHH_{n,k,s}$ is spanned by elements of the form $m \odot \delta_{n,k,s}$,
where $m \in \Omega_n$ is a monomial.

In this subsection we derive a lower bound for the Hilbert series of $\HHH_{n,k,s}$
(and $\WWW_{n,k,s}$) by showing that every substaircase monomial in $\MMM_{n,k,s}$
arises as the initial term $\initial(m \odot \delta_{n,k,s})$ for some monomial $m \in \Omega_n$.
In fact, we show that the map $m \mapsto \initial(m \odot \delta_{n,k,s})$ is an involution on the set 
$\MMM_{n,k,s}$.

\begin{proposition}
\label{involution-proposition}
Let $n, k, s \geq 0$ be integers with $k \geq s$.
The map
\begin{equation}
\iota:  \MMM_{n,k,s} \longrightarrow \MMM_{n,k,s}
\end{equation}
given by $\iota(m) := \initial(m \odot \delta_{n,k,s})$ is a well-defined involution.
\end{proposition}

The well-definedness assertion in Proposition~\ref{involution-proposition} means that $m \odot \delta_{n,k,s}$
is nonzero whenever $m \in \MMM_{n,k,s}$ and that 
$\initial(m \odot \delta_{n,k,s}) \in \MMM_{n,k,s}$. We will see that for any $f \in \Omega_n$ with
$f \odot \delta_{n,k,s} \neq 0$ we have $\initial(f \odot \delta_{n,k,s}) \in \MMM_{n,k,s}$.

\begin{proof}Fix some $m \in \MMM_{n,k,s}$ and let $u$ be the exponent sequence of $m$. We use Algorithm \ref{alg:U} below to construct the exponent sequence $U$ of the monomial $M \in \delta_{n,k,s}$ such that $m \odot M \doteq \iota(m)$. 
To begin a running example, if $n = 9$, $k = 7$, $s=4$, and $u = (\bar{1},3,0,5,2,0,\bar{2},\bar{5},0)$, Algorithm \ref{alg:U} returns 
$$U = (\bar{6},3,2,\bar{6},\bar{6},1,\bar{6},\bar{6},0).$$

\begin{algorithm}
\caption{Constructing $U$ from $u$.}
\label{alg:U}
\begin{algorithmic}[1]
\STATE $i \leftarrow 1$, $\ell \leftarrow s$, $U \leftarrow \varnothing$
\WHILE{$i \leq n$} 
\IF{$u_i$ is barred or $u_i \geq \ell$} \label{u-U-barred}
\STATE $U_i \leftarrow \overline{k-1}$ \label{U-k-1}
\ELSE 
\STATE $\ell \leftarrow \ell - 1$
\STATE $U_i \leftarrow \ell$ \label{U-l}
\ENDIF
\ENDWHILE
\end{algorithmic}
\end{algorithm}

Since $u$ is sub-staircase, $M$ is a valid monomial in $ \delta_{n,k,s}$. We claim that $M$ is the largest monomial in $\delta_{n,k,s}$ with respect to superlexicographic order such that $m \odot M \neq 0$. Suppose $U'$ is the exponent sequence of another monomial $M'$ in $ \delta_{n,k,s}$ such that $U' \succ U$ . Then there exists an index $j$ such that $U'_i = U_i$ for all $i < j$ and $U'_j \succ U_j$. Since $U$ and $U'$ are shuffles of $s-1, s-2, \ldots, 1, 0$ with $(\overline{k-1})^{n-s}$, this implies that $U^{\prime}_j$ is not barred and $U_j = \overline{k-1}$. By lines \ref{u-U-barred} and \ref{U-k-1} in Algorithm \ref{alg:U}, $U_j = \overline{k-1}$ implies that either $u_j$ is barred or $u_j \geq L$, where $$L = \min\{\text{barred entries in }U_1 \ldots U_{j-1}\}.$$
If $u_j$ is barred and $U'_j$ is unbarred, $m \odot M' = 0$. Otherwise, $u_j \geq L$ and $L$ is strictly greater than all the potential unbarred values of $U'_j$, which again implies $m \odot M' = 0$. Therefore $M$ is the largest monomial in $\delta_{n,k,s}$ such that $m \odot M \neq 0$. 

For any other monomial $M''$ in $\delta_{n,k,s}$, $M^{\prime \prime} \prec M$ means that the first index $j$ at which their exponent sequences differ has a barred entry in the exponent sequence  $M''$ and an unbarred entry in $U$. At index $j$, the exponent sequence of $m \odot M''$ is barred while the exponent sequence of $m \odot U$ is unbarred, so $m \odot M'' \prec m \odot M$. Therefore $m \odot M = \initial(m \odot \delta_{n,k,s}) = \iota(m)$.

Next, we show that $\iota(m) \in \MMM_{n,k,s}$. Let $m^* = \iota(m)$ and $u^*$ be the exponent sequence of $m^*$. In Algorithm \ref{alg:v}, we use $u$ and $U$ to construct an $(n,k,s)$-staircase $v$ that certifies $m^* \in \MMM_{n,k,s}$, i.e.\ for every $i$, $u^* \leq v_i$ and $v_i$ is barred if and only if $u^*_i$ is barred. Continuing our running example, if $n=9$, $k=7$, $s=4$, $u = (\bar{1},3,0,5,2,0,\bar{2},\bar{5},0)$, and $U = (\bar{6},3,2,\bar{6},\bar{6},1,\bar{6},\bar{6},0)$, we can compute
$$u^* = (5,0,2,\bar{1},\bar{4},1,4,1,0).$$
From $u$ and $U$, Algorithm \ref{alg:v} produces
$$v = (6,3,2,\bar{4},\bar{4},1,6,6,0)$$
and $u^*$ is indeed `below' $v$. Note that, although the stage of Algorithm \ref{alg:v} involving the set $A$ seems complicated, we are simply letting $v_i$ be the unique barred entry that is allowed at that index in an $(n,k,s)$-staircase.

\begin{algorithm}
\caption{Creating an $(n,k,s)$-staircase $v$ from $u$ and $U$.}
\label{alg:v}
\begin{algorithmic}[1]
\STATE $i \leftarrow 1$, $\ell = s$, $v \leftarrow \varnothing$
\WHILE{$i \leq n$}
\IF{$u_i$ and $U_i$ are both barred}
\STATE $v_i \leftarrow k-1$
\ELSIF{$U_i$ is not barred}
\STATE $\ell \leftarrow \ell - 1$
\STATE $v_i \leftarrow \ell$  \label{v-l}
\ELSE
\STATE $A \leftarrow \{ j \in v_1 \ldots v_{i-1} : j \in \{0, 1, \ldots, s-1\}, \, j \text{ not barred} \}  $
\IF{$A = \varnothing$}
\STATE $v_i = \overline{k-s-1}$
\ELSE
\STATE $v_i = \overline{ k - \min A - 1}$
\ENDIF
\ENDIF
\ENDWHILE
\end{algorithmic}
\end{algorithm}

By construction, $v_i$ is barred if and only if $u^*_i$ is barred. In the first ``if'' clause of Algorithm \ref{alg:v} , we must have $u^*_i \leq v_i$ simply because $k-1$ is the largest possible entry of $u^*$. In the second such clause, if we have $u^*_i > v_i$, since $u^*_i = U_i - u_i$, we must have $U_i - u_i > v_i$ and $U_i > v_i$. This is impossible, since Algorithms \ref{alg:U} and \ref{alg:v} always use the largest unused element $\ell$ of $\{0,1,\ldots,s-1\}$ in line \ref{U-l} of each respective algorithm. In the third such clause,  we must have $U_i = \overline{k-1}$ and $u_i$ is not barred. If $A = \emptyset$, then line \ref{u-U-barred} in Algorithm \ref{alg:U} implies $u_i \geq s$. Hence $u_i^* \leq k-s-1$ and $u^*_i$ is barred, so $u^*_i$ and $v_i$ are both barred and $u_i^* \leq v_i$. If $A \neq \varnothing$, suppose for contradiction that $u^*_i > v_i$. Then
\begin{align*}
u^*_i = U_i - u_i &> v_i \\
k-1 - u_i &> k- \min A -1 \\
u_i &< \min A.
\end{align*}
so $u_i$ is an unbarred entry that is less than $\min A$. Then Algorithm \ref{alg:U} would have used line \ref{U-l} to set $U_i = \min A-1$, since at that stage, we would have $\ell > u_i$. Therefore $m^* \in \MMM_{n,k,s}$.

Finally, we show that $\iota$ is an involution. From the argument above, we know $\iota(m) = m^{*} \in \MMM_{n,k,s}$. Suppose $M^*$ is the monomial we obtain by running Algorithm \ref{alg:U} on $m^{*}$. Then $M^*$ is the monomial in $\delta_{n,k,s}$ such that $m^* \odot M^* \doteq \iota(m^{*})$. We claim that $M^* = M$. If we can show this, then we will have
$$\iota(\iota(m)) \doteq m^* \odot M^* = m^* \odot M \doteq (m \odot M) \odot M \doteq m$$
so $\iota(\iota(m)) = m$.
In our running example where $n= 9$, $k=7$, $s=4$, $u = (\bar{1},3,0,5,2,0,\bar{2},\bar{5},0)$,  $U = (\bar{6},3,2,\bar{6},\bar{6},1,\bar{6},\bar{6},0)$, $u^* =  (5,0,2,\bar{1},\bar{4},1,4,1,0)$, and $v =  (6,3,2,\bar{4},\bar{4},1,6,6,0)$, we get 
$$U^* =  (\bar{6},3,2,\bar{6},\bar{6},1,\bar{6},\bar{6},0) = U.$$

To prove that $U^* = U$ in every case, consider running Algorithm \ref{alg:U} on $m$ and $m^*$ in parallel. We let $\ell^{*}$ be the value in the algorithm creating $M^*$ corresponding to $\ell$ and let $U^*$ be the exponent sequence of $M^*$. 
This process has the following loop invariants.

{\bf Claim:} {\em After every loop of Algorithm \ref{alg:U}, we have $\ell^* = \ell$ and $U^*_i = U_i$.}

We prove the Claim by considering the $i^{th}$ loop of Algorithm \ref{alg:U} and 
 analyzing the possible values of $u_i$.
\begin{itemize}
\item Suppose $u_i$ is barred. Then $U_i = \overline{k-1}$ and $u^*_i = k-1-u_i$. Since $u$ is sub-staircase, $u_i \leq k - \ell -1$, so
$$u_i^{*} = k-1-u_i \geq k-1-(k-\ell-1) = \ell.$$
Hence $U^*_i = \overline{k-1} = U_i$.
\item Suppose $u_i$ is unbarred and $u_i \geq \ell$. Then $U_i = \overline{k-1}$, $u_i^{*}$ is barred, and $U_i^{*} = \overline{k-1} = U_i$. 
\item Finally, suppose $u_i$ is unbarred and $u_i < \ell$. Then we let $U_i = \ell-1$ and get
$$u^*_i = \ell - u_i - 1 = \ell^* - u_i - 1 < \ell^*.$$
Then the algorithm decrements both $\ell$ and $\ell^*$ and sets $U_i^{*} = \ell^{*} = \ell = U_i$.
\end{itemize}
This proves the Claim and the proposition.
\end{proof}

Proposition~\ref{involution-proposition} gives the following lower bound on the dimension
of $\WWW_{n,k,s}$. In Theorem~\ref{M-is-basis} we will prove that this lower bound is tight.

\begin{lemma}
\label{substaircase-appears-lemma}
The set $\MMM_{n,k,s}$ descends to a linearly independent subset of 
$\WWW_{n,k,s}$ and 
we have $\dim \WWW_{n,k,s} = \dim \HHH_{n,k,s} \geq |\MMM_{n,k,s}| = |\OSP_{n,k,s}|$.
\end{lemma}

\begin{proof}
Proposition~\ref{involution-proposition} shows that 
$\{m \odot \delta_{n,k,s} \,:\, m \in \MMM_{n,k,s} \}$ is a linearly independent subset of 
$\HHH_{n,k,s}$, so that 
$\dim \WWW_{n,k,s} = \dim \HHH_{n,k,s} \geq |\MMM_{n,k,s}|$.
The equality $|\MMM_{n,k,s}| = |\OSP_{n,k,s}|$ follows from
Theorem~\ref{code-is-bijection}.
\end{proof}

\subsection{Monomial basis and quotient presentation}
The goal of this section is twofold: giving an explicit generating set of the defining
ideal $\ann \, \delta_{n,k,s}$ of $\WWW_{n,k,s}$ and 
proving that $\MMM_{n,k,s}$ descends to a basis of $\WWW_{n,k,s}$.
The proofs of these two results will be deeply intertwined.

We begin by defining an ideal $I_{n,k,s} \subseteq \Omega_n$ which will turn 
out to equal $\ann \, \delta_{n,k,s}$.  If $S \subseteq [n]$ and $d \geq 0$, we let 
$e_d(S)$ be the elementary symmetric polynomial of degree $d$ in the variable 
set $\{ x_i \,:\, x \in S \}$.  For example, we have 
$e_2(134) = x_1 x_3  + x_1 x_4 + x_3 x_4$. 
We adopt the convention $e_d(S) = 0$ whenever $d > |S|$.

\begin{defn}
\label{ideal-definition}
Given integers $n, k, s \geq 0$ with $k \geq s$,
define $I_{n,k,s} \subseteq \Omega_n$ to be the ideal
\begin{multline}
I_{n,k,s} := \langle x_1^k, x_2^k, \dots, x_n^k \rangle +
\langle x_1^j \theta_1 + x_2^j \theta_2 + \cdots + x_n^j \theta_n \,:\, j \geq k-s \rangle + \\
\langle e_d(S) \cdot \theta_T \,:\, S \sqcup T = [n], \, d > |S| - s \rangle.
\end{multline}
In the third summand, the pair $(S,T)$ ranges over all disjoint union decompositions
$S \sqcup T$ of $[n]$.
\end{defn}

For example, suppose $(n,k,s) = (4,3,2)$. The ideal $I_{4,3,2} \subseteq \Omega_4$ 
is generated by the variable powers $x_1^3, x_2^3, x_3^3, x_4^3$,
the elements $x_1^j \theta_1 + x_2^j \theta_2 + x_3^j \theta_3 + x_4^j \theta_4$
where $j \geq 1$, and the polynomials
\begin{small}
\begin{multline*}
e_4(1234), e_3(1234),  \\ e_3(123) \cdot \theta_4, e_2(123) \cdot \theta_4,
e_3(124) \cdot \theta_3, e_2(124) \cdot \theta_3,
e_3(134) \cdot \theta_2, e_2(134) \cdot \theta_2, 
e_3(234) \cdot \theta_1, e_2(234) \cdot \theta_1,  \\
e_2(12) \cdot \theta_{34}, e_1(12) \cdot \theta_{34},
e_2(13) \cdot \theta_{24}, e_1(13) \cdot \theta_{24}, 
e_2(14) \cdot \theta_{23}, e_1(14) \cdot \theta_{23}, \\
e_2(23) \cdot \theta_{14}, e_1(23) \cdot \theta_{14},
e_2(24) \cdot \theta_{13}, e_1(24) \cdot \theta_{13},
e_2(34) \cdot \theta_{12}, e_1(34) \cdot \theta_{12}, \\
e_1(4) \cdot \theta_{123},
\theta_{123}, e_1(3) \cdot \theta_{124},
\theta_{124},  e_1(2) \cdot \theta_{134},
\theta_{134}, e_1(1) \cdot \theta_{234},
 \theta_{234}.
\end{multline*}
\end{small}
There is redundancy in this list of generators, e.g. the generator 
$e_1(4) \cdot \theta_{123} = x_4 \cdot \theta_{123}$
is a multiple of the generator $\theta_{123}$.

We adopt the temporary notation
\begin{equation}
\UUU_{n,k,s} := \Omega_n/I_{n,k,s}
\end{equation}
for the quotient of $\UUU_{n,k,s}$ by $I_{n,k,s}$. 
Since $I_{n,k,s}$ is $\symm_n$-stable and bihomogeneous, the quotient 
$\UUU_{n,k,s}$ is a bigraded $\symm_n$-module.
It will turn out that $\UUU_{n,k,s}$ and 
$\WWW_{n,k,s}$ coincide.
The following lemma shows that $\WWW_{n,k,s}$ is a quotient of $\UUU_{n,k,s}$.

\begin{lemma}
\label{I-in-ann}  Let $n, k, s \geq 0$ with $k \geq s$.
We have the containment $I_{n,k,s} \subseteq \ann \, \delta_{n,k,s}$.
\end{lemma}

\begin{proof}
For every generator $g$ of the ideal $I_{n,k,s}$ we check that $g \odot \delta_{n,k,s} = 0$.
We handle the three different kinds of generators $g$ of $I_{n,k,s}$ separately.
Throughout this proof we will use $\doteq$ to denote equality up to a nonzero scalar.

{\bf Case 1:} {\em $g = x_i^k$ for some $i$.}

Since no power of $x_i \geq k$ appears in
\begin{equation*}
\delta_{n,k,s} = \varepsilon_n \cdot
(x_1^{k-1} \cdots x_{n-s}^{k-1} x_{n-s+1}^{s-1} \cdots x_{n-1}^1 x_n^0 \times \theta_1 \cdots \theta_{n-s}),
\end{equation*}
it follows that $x_i^k \odot \delta_{n,k,s} = 0$.

{\bf Case 2:} {\em $g = x_1^j \theta_1 + \cdots + x_n^j \theta_n$ for some $j \geq k-s$.}

We compute
\begin{align}
g \odot \delta_{n,k,s} &= g \odot \varepsilon_n \cdot
(x_1^{k-1} \cdots x_{n-s}^{k-1} x_{n-s+1}^{s-1} \cdots x_{n-1}^1 x_n^0 \times \theta_1 \cdots \theta_{n-s}) \\
&= \varepsilon_n \cdot (g \odot (x_1^{k-1} \cdots x_{n-s}^{k-1} x_{n-s+1}^{s-1} \cdots x_n^0 \times \theta_1 \cdots \theta_{n-s}) ) \\
&\doteq \varepsilon_n \cdot \left( \sum_{i = 1}^{n-s}  (-1)^{i-1}
x_1^{k-1} \cdots x_i^{k-j-1} \cdots x_{n-s}^{k-1} x_{n-s+1}^{s-1} \cdots x_n^0 \times 
\theta_1 \cdots \widehat{\theta_i} \cdots \theta_{n-s} \right)  \\
&= \sum_{i = 1}^{n-s}  (-1)^{i-1} \varepsilon_n \cdot \left(
x_1^{k-1} \cdots x_i^{k-j-1} \cdots x_{n-s}^{k-1} x_{n-s+1}^{s-1} \cdots x_n^0 \times 
\theta_1 \cdots \widehat{\theta_i} \cdots \theta_{n-s} \right)  \\
&= 0.
\end{align}
The first equality is the definition of $\delta_{n,k,s}$. The second equality follows because $g \odot (-)$
is an $\symm_n$-invariant differential operator, and so commutes with the action of $\QQ[\symm_n]$
on $\Omega_n$.  The third equality is the effect of applying $g \odot (-)$ and the fourth is rearrangement.
The final equality holds since the hypothesis $j \geq s$ implies $k-j-1 \leq s-1$, so that $\varepsilon_n$
annihilates the given superspace monomial.

{\bf Case 3:} {\em $g = e_d(S) \cdot \theta_T$ for some $S \sqcup T = [n]$ with $d > |S| - s$.}

Let $\delta_{|S|,k,s}(S)$ denote the superspace Vandermonde $\delta_{|S|,k,s}$ defined in the  
set  $\{ x_i, \theta_i \,:\, i \in S \}$ of commuting and anticommuting variables indexed by $S$.
Applying the operator $\theta_T \odot (-)$ to $\delta_{n,k,s}$, we see that 
\begin{equation}
\theta_T \odot \delta_{n,k,s} \doteq
\prod_{j \in T} x_j^{k-1} \times \delta_{|S|,k,s}(S).
\end{equation}
Therefore, we have
\begin{equation}
g \odot \delta_{n,k,s} \doteq e_d(S) \odot  \left( \prod_{j \in T} x_j^{k-1} \times \delta_{|S|,k,s}(S) \right)  =
\prod_{j \in T} x_j^{k-1} \times \left(e_d(S) \odot \delta_{|S|,k,s}(S)\right) = 0
\end{equation}
where the second equality follows from the disjointness of the sets $S$ and $T$ and the final
equality is a consequence of \cite[Lem. 3.2]{RW}.
\end{proof}

\subsection{The harmonic space $I_{n,k,s}^{\perp}$}

Throughout this section, we adopt the notation
\begin{equation}
\Omega'_{n-1} := \QQ[x_2, \dots, x_n] \otimes \wedge \{ \theta_2, \dots, \theta_n \}
\end{equation}
for rank $n-1$ superspace over the commuting variables $x_2, \dots, x_n$
and the anticommuting variables $\theta_2, \dots, \theta_n$.

Lemma~\ref{I-in-ann} gives the containment $I_{n,k,s} \subseteq \ann \, \delta_{n,k,s}$.
We consider the harmonic space
\begin{equation}
I_{n,k,s}^{\perp} = \{ f \in \Omega_n \,:\, g \odot f = 0 \text{ for all $g \in I_{n,k,s}$} \}
\end{equation}
to the ideal $I_{n,k,s}$.  The composite map
\begin{equation}
I_{n,k,s}^{\perp} \hookrightarrow \Omega_n \twoheadrightarrow \UUU_{n,k,s}
\end{equation}
is an isomorphism of bigraded $\symm_n$-modules.
Lemma~\ref{I-in-ann} gives the containment $I_{n,k,s} \subseteq \ann \, \delta_{n,k,s}$ which implies
\begin{equation}
\label{harmonic-containment}
\HHH_{n,k,s} \subseteq I_{n,k,s}^{\perp}.
\end{equation}
We will see (Theorem~\ref{M-is-basis}) that these spaces coincide, but 
it will be convenient to work with the {\em a priori} larger space $I_{n,k,s}^{\perp}$ in this section.

In order to show the containment \eqref{harmonic-containment} is in fact an equality, we 
bound the dimension of $I_{n,k,s}^{\perp}$ from above. To do this, we inductively show 
(Lemma~\ref{leading-term-lemma}) that the leading 
terms of nonzero polynomials $f \in I_{n,k,s}^{\perp}$ belong to the family 
$\MMM_{n,k,s}$ of substaircase monomials.
For the sake of this induction, we prove lemmata about the expansions
of polynomials $f \in I_{n,k,s}^{\perp}$ in terms of the initial variables $x_1$ and $\theta_1$.
The first of these is a simple degree bound.

\begin{lemma}
\label{k-upper-bound-lemma}
Let $n, k, s \geq 0$ with $k \geq s$ and let $f \in I_{n,k,s}^{\perp}$.
Then $f$ may be expressed uniquely as
\begin{equation}
f = x_1^0 g_0 + x_1^1 g_1 + \cdots + x_1^{k-1} g_{k-1} + x_1^0 \theta_1 h_0 + x_1^1 \theta_1 h_1 + \cdots
+ x_1^{k-1} \theta_1 h_{k-1}
\end{equation}
for some elements $g_0, g_1, \dots, g_{k-1}, h_0, h_1, \dots, h_{k-1} \in \Omega_{n-1}'$.
\end{lemma}

\begin{proof}
The element $x_1^k$ is a generator of $I_{n,k,s}$.  Since $f \in I_{n,k,s}^{\perp}$, we have 
$x_1^k \odot f = 0$, so that no power of $x_1 \geq k$ may appear in $f$.
\end{proof}

By the definition of the order $\prec$ on superspace monomials, 
if  $f \in I_{n,k,s}^{\perp}$ is nonzero and if
\begin{equation*}
f = x_1^0 g_0 + x_1^1 g_1 + \cdots + x_1^{k-1} g_{k-1} + x_1^0 \theta_1 h_0 + x_1^1 \theta_1 h_1 + \cdots
+ x_1^{k-1} \theta_1 h_{k-1}
\end{equation*}
as in Lemma~\ref{k-upper-bound-lemma}, the $\prec$-leading term of $f$ is given by
\begin{equation*}
\initial(f) = \begin{cases}
x_1^i \cdot \initial(g_i) & \text{if $g_j \neq 0$ for some $j$, and if $i$ is maximal with $g_i \neq 0$,} \\
x_1^i \theta_1 \cdot  \initial(h_i )& \text{if $g_j = 0$ for all $j$, and if $i$ is minimal with $h_i \neq 0$.}
\end{cases}
\end{equation*}
With Lemma~\ref{disjoint-union-decomposition} in mind, we aim to relate the harmonicity of 
$f \in I_{n,k,s}^{\perp}$ to showing that $g_i$ and $h_i$ lie in appropriate harmonic spaces within
$\Omega_{n-1}'$.
Our starting point is the following vanishing result for harmonics divisible by $\theta_1$.

\begin{lemma}
\label{h-vanishing-lemma}
Let $n,k,s \geq 0$ be integers with $k \geq s$ 
and suppose $f \in I_{n,k,s}^{\perp}$ is a multiple of $\theta_1$:
\begin{equation*}
f = x_1^0 \theta_1 h_0 + x_1^1 \theta_1 h_1 + \cdots
+ x_1^{k-1} \theta_1 h_{k-1}
\end{equation*}
for some $h_0, h_1, \dots, h_{k-1} \in \Omega'_{n-1}$.  We have 
$h_{k-s} = h_{k-s+1} = \cdots = h_{k-1} = 0$.
\end{lemma}

The vanishing conditions in Lemma~\ref{h-vanishing-lemma} will turn out 
to mirror the first branch of the disjoint
union decompositions in Lemma~\ref{disjoint-union-decomposition}.

\begin{proof}
The generator $x_1^{k-s} \theta_1 + x_2^{k-s} \theta_2 + \cdots + x_n^{k-s} \theta_n \in I_{n,k,s}$ 
annihilates $f$ under the $\odot$-action.
This gives an equation of the form
\begin{equation}
0 = C_{k-s} x_1^0 h_{k-s} + C_{k-s+1} x_1^1 h_{k-s+1} + \cdots + C_{k-1} x_1^{s-1} h_{k-1} + 
\text{(a multiple of $\theta_1$)}
\end{equation}
where the $C$'s are nonzero
constants. Taking the coefficients of 
$x_1^0, x_1^1, \dots, x_1^{s-1}$, we see that 
$h_{k-s} = h_{k-s+1} = \cdots = h_{k-1} = 0$.
\end{proof}

Lemma~\ref{h-vanishing-lemma} may be enhanced to consider harmonics $f \in I_{n,k,s}^{\perp}$
whose decomposition contains terms $x_1^i g_i$ for only relatively low values of $i$.

\begin{lemma}
\label{general-vanishing-lemma}
Let $n, k, s \geq 0$ be integers with $k \geq s$
and let $f \in I^{\perp}_{n,k,s}$.
Assume that 
\begin{equation*}
f =  x_1^0 g_0 + x_1^1 g_1 + \cdots + x_1^i g_i + 
x_1^0 \theta_1 h_0 + x_1^1 \theta_1 h_1 + \cdots
+ x_1^{k-1} \theta_1 h_{k-1}
\end{equation*}
for some $0 \leq i \leq k-1$.  Then 
$h_{k-s+i+1} = h_{k-s+i+2} = \cdots = h_{k-1} = 0$.
\end{lemma}

In light of Lemma~\ref{k-upper-bound-lemma},
Lemma~\ref{general-vanishing-lemma} gives no information in the range $s \leq i \leq k$.
In accordance with the middle branches of the disjoint union decompositions of 
Lemma~\ref{disjoint-union-decomposition}, 
Lemma~\ref{general-vanishing-lemma} will be useful in the range $0 \leq i \leq s-1$.

\begin{proof}
The space $I^{\perp}_{n,k,s}$ is closed under the operator $x_1^{i+1} \odot (-)$.
Applying this operator to $f$, we get the harmonic polynomial
\begin{equation}
x_1^{i+1} \odot f = 
C_{i+1} x_1^0  \theta_1  h_{i+1} + C_{i+2} x_1^1 \theta_1 h_{i+2} + \cdots +
C_{k-1} x_1^{k-i-2} \theta_1 h_{k-1} \in I^{\perp}_{n,k,s},
\end{equation}
where the $C$'s are nonzero constants.
Applying Lemma~\ref{h-vanishing-lemma} to this polynomial, we see that 
$h_{k-s+i+1} = h_{k-s+i+2} = \cdots = h_{k-1} = 0$.
\end{proof}

The following lemma characterizes the leading terms of nonzero
elements in $I^{\perp}_{n,k,s}$.
The recursive description of the set $\MMM_{n,k,s}$ of substaircase monomials given in 
Lemma~\ref{disjoint-union-decomposition} is crucial to its proof.

\begin{lemma}
\label{leading-term-lemma}
Let $n, k, s \geq 0$ with $k \geq s$ and let $f \in I^{\perp}_{n,k,s}$ be nonzero.
The $\prec$-leading term of $f$ is a substaircase monomial, i.e.
\begin{equation}
\initial(f) \in \MMM_{n,k,s}.
\end{equation}
\end{lemma}

\begin{proof}
We proceed by induction on $n$. When $n = 1$,  the space 
$I^{\perp}_{1,k,s}$ and the family of monomials $\MMM_{1,k,s}$ have three flavors depending 
on the value of $s$.
If $s = 0$, then 
\begin{equation}
\MMM_{1,k,0} = \{ x_1^{k-1} \theta_1, \dots, x_1 \theta_1, \theta_1, x_1^{k-1}, \dots, x_1^1, 1 \}
\end{equation}
and we have $I_{1,k,0}^{\perp} = \mathrm{span} \MMM_{1,k,0}$, so the result follows.
If $s = 1$, then $I_{1,k,1}^{\perp}$ is the ground field $\QQ$ (independent of the value of $k$)
and $\MMM_{1,k,1} = \{1\}$, so the result holds. If $s > 1$, then
$e_0 = 1 \in I_{1,k,s}$ so that $I_{1,k,s} = \Omega_1$ and $I_{1,k,s}^{\perp} = 0$;
since $\MMM_{1,k,s} = \varnothing$, the result is true in this case as well.
Going forward, we let 
 $n > 1$ and that the lemma has been established for all smaller values of $n$,
and all values of $k$ and $s$.

By Lemma~\ref{k-upper-bound-lemma}, we may write $f$ 
uniquely as 
\begin{equation}
f = x_1^0 g_0 + x_1^1 g_1 + \cdots + x_1^{k-1} g_{k-1} + x_1^0 \theta_1 h_0 + x_1^1 \theta_1 h_1 + \cdots
+ x_1^{k-1} \theta_1 h_{k-1}
\end{equation}
for some $g_0, g_1, \dots, g_{k-1}, h_0, h_1, \dots, h_{k-1} \in \Omega'_{n-1}$.  
Our analysis breaks up into cases depending on the vanishing properties of these elements of 
$\Omega'_{n-1}$.

{\bf Case 1:}  {\em We have $g_0 = g_1 = \cdots = g_{k-1} = 0$.}

In this case, $f$ is a multiple of $\theta_1$. Lemma~\ref{h-vanishing-lemma} applies and we may write 
\begin{equation}
f = x_1^i \theta_1 h_i + x_1^{i+1} \theta_1 h_{i+1} + \cdots + x_{k-s-1}^{k-s-1} \theta_1 h_{k-s-1}
\end{equation}
for some $0 \leq i < k-s$ with $h_i \neq 0$.  We have
\begin{equation}
\initial(f) = x_1^i \theta_1 \initial(h_i).
\end{equation}
By the upper branches of the disjoint union decompositions in Lemma~\ref{disjoint-union-decomposition},
it is enough to show that $h_i \in (I'_{n-1,k,s})^{\perp}$.  
Here $I'_{n-1,k,s} \subseteq \Omega'_{n-1}$ is the image of
$I_{n-1,k,s}$ under the map $x_j \mapsto x_{j+1}, \theta_j \mapsto \theta_{j+1}$.
We show that each generator of $I_{n-1,k,s}'$ annihilates $h_i$.

For $2 \leq j \leq n$, we have $x_j^k \odot f = 0$ so that 
\begin{equation}
0 = x_j^k \odot f = x_1^i \theta_1
[(x_j^k) \odot h_i]+ x_1^{i+1} \theta_1 [(x_j^k) \odot h_{i+1}] + \cdots + 
x_{k-s-1}^i \theta_1 [(x_j^k) \odot h_{k-s-1}].
\end{equation}
Taking the coefficient of $x_1^i \theta_1$ shows that $x_j^k \odot h_i = 0$.

Given $d \geq k-s$, we have $(x_1^d \theta_1 + x_2^d \theta_2 + \cdots + x_n^d \theta_n) \odot f = 0$,
so that
\begin{equation}
0 =  \sum_{j = i}^{k-s-1} x_1^j \theta_1 \times (x_2^d \theta_2 + \cdots + x_n^d \theta_n) \odot h_j + 
\text{(terms not involving $\theta_1$).}
\end{equation}
Taking the coefficient of
$x_1^i \theta_1$, we see that $(x_2^d \theta_2 + \cdots + x_n^d \theta_n) \odot h_j = 0$.

Finally, let $S \sqcup T = \{2,\dots,n\}$ be a disjoint union decomposition and let $d > |S|-s$.
We aim to show that $(e_d(S)  \theta_T) \odot h_i = 0$. Indeed, if we let $\hat{T} := T \cup \{1\}$,
we have $(e_d(S)  \theta_{\hat{T}}) \odot f = 0$ since $f \in I_{n,k,s}^{\perp}$. This means  
\begin{equation}
0 = x_1^i [(e_d(S) \theta_T) \odot h_i]  + x_1^{i+1} [(e_d(S) \theta_T) \odot h_{i+1}] + \cdots 
+ x_1^{k-s-1} [ (e_d(S) \theta_T) \odot h_{k-s-1} ] 
\end{equation}
and taking the coefficient of $x_1^i$ shows $(e_d(S) \theta_T) \odot h_i = 0$.

The previous three paragraphs imply  $h_i \in (I'_{n-1,k,s})^{\perp}$, so 
Lemma~\ref{disjoint-union-decomposition} and induction on $n$ complete the proof of Case 1.

{\bf Case 2:}  {\em At least one $g_j$ is nonzero, but $g_j = 0$ for all $j \geq s$.}

In this case, Lemma~\ref{general-vanishing-lemma} applies and we may write
\begin{equation}
f =  x_1^0 g_0 + x_1^1 g_1 + \cdots + x_1^i g_i + 
x_1^0 \theta_1 h_0 + x_1^1 \theta_1 h_1 + \cdots
+ x_1^{k-s-i} \theta_1 h_{k-s+i}
\end{equation}
for some $0 \leq i < s$ where $g_i \neq 0$. We have 
\begin{equation}
\initial(f) = x_1^i \initial(g_i),
\end{equation}
and inspired by the middle branch of Lemma~\ref{disjoint-union-decomposition}
we verify that $g_i \in (I_{n-1,k,s-1}')^{\perp}$, where the prime again denotes increasing variable indices by
one. We check that $g_i$ is annihilated by every generator of $I_{n-1,k,s-1}$.
The verification that $x_j^k \odot g_i = 0$ is analogous to that in Case 1 and is omitted.

Let $d \geq k-s+1$.  We aim to show that
$(x_2^d \theta_2 + \cdots + x_n^d \theta_n) \odot g_i = 0$.  Indeed, since $f \in I_{n,k,s}^{\perp}$, we have
\begin{multline}
0 = (x_1^d \theta_1 + x_2^d \theta_2 + \cdots + x_n^d \theta_n) \odot f \\
=  \sum_{j = 0}^i x_1^j (x_2^d \theta_2 + \cdots + x_n^d \theta_n) \odot g_j +
\sum_{j' = 0}^{k-s+i} x_1^{j'} \theta_1 (x_2^d \theta_2 + \cdots + x_n^d \theta_n) \odot h_j +
\sum_{j'' = 0}^{k-s+i-d} C_{j''} x_1^{j''} h_{j'' + d}
\end{multline}
where the $C$'s are nonzero constants. Since $d \geq k-s+1$, the final sum does not involve the power
$x_1^i$.  Taking the coefficient of $x_1^i$ shows that 
$(x_2^d \theta_2 + \cdots + x_n^d \theta_n) \odot g_i = 0$. 

Finally, let $S \sqcup T = \{2, \dots, n\}$ be a disjoint union decomposition and let $d > |S| - s + 1$.
We need to show $(e_d(S) \theta_T) \odot g_i = 0$.
If we let $\hat{S} := S \cup \{1\}$, we have the relation
\begin{equation}
e_d(\hat{S}) = e_d(S) + x_1 e_{d-1}(S)
\end{equation}
and $e_d(\hat{S}) \theta_T$ is a generator of $I_{n,k,s}$.  
The equation $(e_d(\hat{S}) \theta_T) \odot f) = 0$ has 
the form
\begin{equation}
0 = \sum_{j = 0}^i  [(e_d(S) + x_1 e_{d-1}(S)) \theta_T] \odot (x_1^j g_j) +
(\text{a multiple of $\theta_1$}).
\end{equation}
 Taking the coefficient of $x_1^i$ shows that
$(e_d(S) \theta_T) \odot g_i = 0$, as required.

The previous two paragraphs show that $g_i \in (I'_{n-1,k,s-1})^{\perp}$. 
Lemma~\ref{disjoint-union-decomposition} and induction on $n$ complete the proof of Case 2.

{\bf Case 3:} {\em At least one $g_j$ is nonzero for some $j \geq s$.}

We apply Lemma~\ref{k-upper-bound-lemma} to write
\begin{equation}
f =  x_1^0 g_0 + x_1^1 g_1 + \cdots + x_1^i g_i + 
x_1^0 \theta_1 h_0 + x_1^1 \theta_1 h_1 + \cdots
+ x_1^{k-1} \theta_1 h_{k-1}
\end{equation}
with $g_i \neq 0$ for some $i \geq s$.
We verify  $g_i \in (I'_{n-1,k,s})^{\perp}$ by showing that every generator
of $I'_{n-1,k,s}$ annihilates $g_i$. Generators of the form $x_j^k$ are handled as before.

Let $d \geq k-s$. Since $f \in I_{n,k,s}^{\perp}$ we have
\begin{align}
0 &= (x_1^d \theta_1 + x_2^d \theta_2 + \cdots + x_n^d \theta_n) \odot f \\ 
&=
\sum_{j = 0}^i x_1^j [ (x_2^d \theta_2 + \cdots + x_n^d \theta_n) \odot g_j ] +
\sum_{j' = 0}^{k-d-1} C_{j'} x_1^{j'} h_{j'+d} + 
\text{(a multiple of $\theta_1$)}
\end{align}
where the $C$'s are constants.
Since $i \geq s > k-d-1$, taking the coefficient of $x_1^i$ shows that 
$(x_2^d \theta_2 + \cdots + x_n^d \theta_n) \odot g_i = 0$.

Finally, let $S \sqcup T = \{2, \dots, n\}$ be a disjoint union decomposition and let $d > |S| - s$.
We show that $(e_d(S) \theta_T) \odot g_i = 0$ as follows.
If we set $\hat{S} := S \cup \{1\}$, we calculate
\begin{align}
x_1^i e_d(S) &= x_1^{i-1} e_{d+1}(\hat{S}) - x_1^{i-1} e_{d+1}(S) \\
&= x_1^{i-1} e_{d+1}(\hat{S}) - x_1^{i-2} e_{d+2}(\hat{S}) + x_1^{i-2} e_{d+2}(S) \\ &= \cdots  \\
&= x_1^{i-1} e_{d+1}(\hat{S}) - x_1^{i-2} e_{d+2}(\hat{S}) + \cdots + (-1)^{i-1} e_{d+i}(\hat{S}) + 
(-1)^i e_{d+i}(S) \\
&= x_1^{i-1} e_{d+1}(\hat{S}) - x_1^{i-2} e_{d+2}(\hat{S}) + \cdots + (-1)^{i-1} e_{d+i}(\hat{S}).
\end{align}
where the last equality used the conditions $i \geq s$ and $d > |S| - s$ so that $e_{d+i}(S) = 0$.
Multiplying both sides by $\theta_T$, we get
\begin{equation}
x_1^i e_d(S) \theta_T = x_1^{i-1} e_{d+1}(S') \theta_T - x_1^{i-2} e_{d+2}(S') \theta_T + 
\cdots + (-1)^{i-1} e_{d+i}(S') \theta_T \in I_{n,k,s}
\end{equation}
since every term on the right-hand side is a generator of $I_{n,k,s}$.  Since $f \in I_{n,k,s}^{\perp}$
this means 
\begin{equation}
0 = (x_1^i e_d(S) \theta_T) \odot f = C (e_d(S) \theta_T) \odot g_i + 
\text{(a multiple of $\theta_1$)}
\end{equation}
where $C$ is a nonzero constant. This implies that 
$(e_d(S) \theta_T) \odot g_i  = 0$.

The previous arguments show that $g_i \in (I'_{n-1,k,s})^{\perp}$. The fact that 
$\initial(f) = x_1^i \initial(g_i)$, the lower branches of the disjoint union decompositions in 
Lemma~\ref{disjoint-union-decomposition}, and induction on $n$ complete the proof of Case 3,
and of the lemma.
\end{proof}

\subsection{Quotient presentation and monomial basis}

We are ready to prove that the ideal $I_{n,k,s}$ and the annihilator $\ann \, \delta_{n,k,s}$ coincide,
so that $\WWW_{n,k,s} = \UUU_{n,k,s}$, and that the substaircase monomials $\MMM_{n,k,s}$ descend
to a basis of $\WWW_{n,k,s}$. 
Thanks to our lemmata, this is a quick argument.

\begin{theorem}
\label{M-is-basis}
Let $n, k, s \geq 0$ be integers with $k \geq s$.
The ideal $I_{n,k,s}$ is the annihilator of the superspace 
Vandermonde $\delta_{n,k,s}$:
\begin{equation}
\ann \, \delta_{n,k,s} =  I_{n,k,s}.
\end{equation}
Consequently, the quotient rings $\WWW_{n,k,s}$ and $\UUU_{n,k,s}$
coincide:
\begin{equation}
\WWW_{n,k,s} = \UUU_{n,k,s}.
\end{equation}
Furthermore, the set $\MMM_{n,k,s}$ descends to a monomial basis of $\WWW_{n,k,s}$.
\end{theorem}

\begin{proof}
We bound the dimension of $I_{n,k,s}^{\perp}$ from above using Lemma~\ref{leading-term-lemma}
and an argument appearing in the work of Rhoades, Yu, and Zhao \cite[Sec. 4.4]{RYZ} on 
harmonic spaces.
Let $N := | \MMM_{n,k,s} |$ be the number of substaircase monomials.
We claim that $\dim I_{n,k,s}^{\perp} \leq N$.
Given $f_1, \dots f_N, f_{N+1} \in I_{n,k,s}^{\perp}$, we have
\begin{equation}
\label{dependence-relation}
f = c_1 f_1 + \cdots  + c_N f_N + c_{N+1} f_{N+1} \in I_{n,k,s}^{\perp}
\end{equation}
for any scalars $c_1, \dots, c_N, c_{N+1} \in \QQ$. We may select $c_i$ not all zero so that the coefficient
of $m$ vanishes on the right-hand side of Equation~\eqref{dependence-relation} for all $m \in \MMM_{n,k,s}$.
By Lemma~\ref{leading-term-lemma}, this forces $f = 0$, so that Equation~\eqref{dependence-relation}
shows $\{ f_1, \dots, f_N, f_{N+1} \}$ is linearly dependent.

We have the chain of equalities 
\begin{equation}
N \geq \dim I_{n,k,s}^{\perp} = \dim \UUU_{n,k,s} \geq \dim \WWW_{n,k,s} \geq N,
\end{equation}
where we applied Lemma~\ref{I-in-ann} to get that $\UUU_{n,k,s}$ projects onto 
$\WWW_{n,k,s}$ and Lemma~\ref{substaircase-appears-lemma} to get 
$\dim \WWW_{n,k,s} \geq N$.
This proves that 
$\ann \, \delta_{n,k,s} =  I_{n,k,s}$ and $\WWW_{n,k,s} = \UUU_{n,k,s}$.
Lemma~\ref{substaircase-appears-lemma} implies that $\MMM_{n,k,s}$
descends to a monomial basis of $\WWW_{n,k,s}$.
\end{proof}

To illustrate Theorem~\ref{M-is-basis}, let $(n,k,s) = (3,2,2)$.
The staircases in this case are 
\begin{equation*}
(1,1,0), \quad (1,0,1), \quad (1,\bar{0},0), \quad (1,0,\bar{1})
\end{equation*}
so that 
\begin{equation*}
\MMM_{3,2,2} = 
\{
1, x_1, x_2, x_3, x_1 x_2, x_1 x_3, \theta_2, \theta_3, x_1 \theta_2, x_1 \theta_3, x_3 \theta_3,
x_1 x_3 \theta_3
\}
\end{equation*}
and Theorem~\ref{M-is-basis} asserts that $\MMM_{3,2,2}$ descends to a basis of 
$\WWW_{3,2,2}$. In particular, the bigraded Hilbert series 
$\Hilb(\WWW_{3,2,2}; q, z)$ is given by
\begin{equation*}
\Hilb(\WWW_{3,2,2}; q, z) = 
1 + 3q + 2q^2 + 2z + 3qz + q^2 z.
\end{equation*}
We may display this Hilbert series as a matrix
\begin{equation*}
\Hilb(\WWW_{3,2,2}; q, z) = 
\begin{pmatrix}
1 & 3 & 2 \\ 2 & 3  & 1
\end{pmatrix}
\end{equation*}
by letting rows track $\theta$-degree and columns track $x$-degree.
The $180^{\circ}$ rotational symmetry of this matrix is guaranteed by the 
Rotational Duality of Theorem~\ref{previous-w-knowledge}.
The larger example
\begin{equation*}
\Hilb(\WWW_{6,3,2}; q, z)  = \begin{pmatrix}
1 & 6 & 21 & 50 & 90 & 125 & 134 & 105 & 55 & 15 \\
6 & 35 & 119 & 273 & 463 & 575 & 511 & 301 & 105 & 20 \\
15 & 84 & 274 & 580 & 853 & 853 & 580 & 274 & 84 & 15 \\
20 & 105 & 301 & 511 & 575 & 463 & 273 & 119 & 35 & 6 \\
15 & 55 & 105 & 134 & 125 & 90 & 50 & 21 & 6 & 1 
\end{pmatrix}.
\end{equation*}
is easy to compute with  the 
following recursion.

\begin{corollary}
\label{matrix-recursion}
Suppose $n, k, s \geq 0$ are integers with $k \geq s$. We have
\begin{equation*}
\Hilb(\WWW_{n,k,s};q,z) = \sum_{r = 0}^{n-s} z^r \cdot \sum_{\sigma \in \OSP_{n,k,s}^{(r)}} q^{\coinv(\sigma)}
  = \sum_{r = 0}^{n-s} z^r \cdot \sum_{\sigma \in \OSP_{n,k,s}^{(r)}} q^{\codinv(\sigma)}.
\end{equation*}
This bigraded Hilbert series satisfies the recursion
\begin{equation*}
\Hilb(\WWW_{n,k,s}; q, z) = (z + q^s) \cdot [k-s]_q  \cdot \Hilb(\WWW_{n-1,k,s};q,z) +
[s]_q \cdot \Hilb(\WWW_{n-1,k,s-1};q,z).
\end{equation*}
\end{corollary}

\begin{proof}
This follows from Lemma~\ref{disjoint-union-decomposition} and 
Theorem~\ref{M-is-basis}.
\end{proof}

Together with the initial conditions
\begin{equation}
\Hilb(\WWW_{n,k,s}; q, z) = 0 \quad \quad \text{if $s > n$}
\end{equation}
and 
\begin{equation}
\Hilb(\WWW_{1,k,s}; q, z) = \begin{cases}
(1 + z) \cdot [k]_q & s = 0 \\
1 & s = 1,
\end{cases}
\end{equation}
Corollary~\ref{matrix-recursion} determines $\Hilb(\WWW_{n,k,s};q,z)$ completely. It is
predicted \cite[Conj. 6.5]{RW} that the matrices
$\Hilb(\WWW_{n,k,s};q,z)$ have unimodal rows and columns.
In the case $(n,k,s) = (n,n,n)$, the ring $\WWW_{n,n,n}$ is the  cohomology 
$H^{\bullet}(\mathcal{F \ell}_n; \QQ)$ of the flag variety 
$\mathcal{F \ell}_n$ and is a consequence of the Hard Lefschetz property for this 
smooth and compact complex manifold.
Corollary~\ref{matrix-recursion} has been used to verify this conjecture for 
all triples $n \geq k \geq s$ with $n \leq 9$.

Since the composition of maps $\HHH_{n,k,s} \hookrightarrow \Omega_n \twoheadrightarrow \WWW_{n,k,s}$
is an isomorphism, any basis of the harmonic space
$\HHH_{n,k,s}$ descends to a basis 
of $\WWW_{n,k,s}$. A basis of $\WWW_{n,k,s}$ obtained in this way is a {\em harmonic basis}.
Harmonic bases of quotients of the polynomial ring $\QQ[x_1, \dots, x_n]$ have received 
significant attention \cite{BG, RYZ} and are useful because working with them does not involve 
computationally expensive operations with cosets.
The space $\WWW_{n,k,s}$ admits the following harmonic basis.

\begin{corollary}
\label{harmonic-basis-corollary}
Let $n, k, s \geq 0$ be integers with $k \geq s$.  We have the following equality of subspaces of 
$\Omega_n$.
\begin{equation}
\HHH_{n,k,s} = I_{n,k,s}^{\perp}.
\end{equation}
Furthermore, the set 
\begin{equation}
\{ m \odot \delta_{n,k,s} \,:\, m \in \MMM_{n,k,s} \}
\end{equation}
is a basis of $\HHH_{n,k,s}$, and therefore descends to a harmonic basis of $\WWW_{n,k,s}$.
\end{corollary}

\begin{proof}
The equality $\HHH_{n,k,s} = I_{n,k,s}^{\perp}$ is immediate from 
Theorem~\ref{M-is-basis}.
For any  $c_1, \dots, c_N \in \QQ$ and any monomials $m_1, \dots, m_N$ we have
\begin{equation}
c_1 (m_1 \odot \delta_{n,k,s}) + \cdots + c_N (m_N \odot \delta_{n,k,s}) = 
(c_1 m_1 + \cdots + c_N m_N) \odot \delta_{n,k,s}
\end{equation}
so that a linear dependence in the subset 
$\{ m_1 \odot \delta_{n,k,s}, \dots, m_N \odot \delta_{n,k,s} \}$ of $\Omega_n$
induces a linear dependence in $\{ m_1, \dots, m_N \}$ modulo 
$\ann \, \delta_{n,k,s}$.  
Theorem~\ref{M-is-basis} implies
$\{ m \odot \delta_{n,k,s} \,:\, m \in \MMM_{n,k,s} \}$ is linearly independent and a dimension count
finishes the proof.
\end{proof}

The harmonic basis of  Corollary~\ref{harmonic-basis-corollary} will be used to calculate the bigraded 
Frobenius image of $\WWW_{n,k,s}$.

\section{Frobenius Image}
\label{Frobenius}

\subsection{A tensor product decomposition of $\Omega_n$}
The goal of this section is to prove that the graded Frobenius image $\grFrob(\WWW_{n,k,s};q,z)$
has the combinatorial expansion $C_{n,k,s}(\xx;q,z)$. Since $\WWW_{n,k,s} \cong \HHH_{n,k,s}$
as bigraded $\symm_n$-modules, we will often use  the harmonics space 
$\HHH_{n,k,s}$ instead to avoid working with cosets. 
The combinatorial recursion of Lemma~\ref{c-skewing-lemma} necessitates
restricting these spaces to a given $\theta$-degree.

\begin{defn}
\label{w-r-defn}
Given $r \geq 0$,
let $\WWW_{n,k,s}^{(r)} \subseteq \WWW_{n,k,s}$ 
and $\HHH_{n,k,s}^{(r)} \subseteq \HHH_{n,k,s}$ be the subspaces of homogeneous $\theta$-degree $r$.
\end{defn}

 $\WWW_{n,k,s}^{(r)}$ and 
$\HHH_{n,k,s}^{(r)}$ are isomorphic singly graded $\symm_n$-modules under  $x$-degree.
Our goal is to show 
\begin{equation}
\label{main-goal-equation}
\grFrob(\WWW_{n,k,s}^{(r)};q) = \grFrob(\HHH_{n,k,s}^{(r)};q) =  C_{n,k,s}^{(r)}(\xx;q).
\end{equation}
Lemma~\ref{skew-by-e-lemma} allows us to prove Equation~\eqref{main-goal-equation} by showing
 both sides  satisfy the same recursion under an appropriate
family of skewing operators.
The combinatorial side $C_{n,k,s}^{(r)}$ was handled by Lemma~\ref{c-skewing-lemma};
we must now consider the representation theoretic side.
In order to avoid repeating hypotheses, we fix the following \\

\noindent
{\bf Notation.}
{\em For the remainder of this section, we fix integers $n \geq s \geq 0$, $k \geq s$, $n-s \geq r \geq 0$, and $1 \leq j \leq n$.} \\

The group algebra $\QQ[\symm_j]$ and its antisymmetrizing and symmetrizing 
elements $\varepsilon_j$ and $\eta_j$ act on
the first
$j$ indices of  $\Omega_n$.
Given the relationship
\begin{equation}
\label{eta-recursion-for-w}
\grFrob( \eta_j \HHH_{n,k,s}^{(r)}; q) = h_j^{\perp} \grFrob(\HHH_{n,k,s}^{(r)}; q),
\end{equation}
and Lemma~\ref{c-skewing-lemma},
we would like a recursive understanding of the $\symm_{n-j}$-modules 
$\eta_j \HHH_{n,k,s}^{(r)}$. 
This will be accomplished by finding a strategic basis $\CCC$ of $\eta_j \HHH_{n,k,s}^{(r)}$
(see Definition~\ref{C-defn} below).

We will need to distinguish between the first $j$ and last $n-j$ indices
appearing in $\Omega_n$. To this end, we make the tensor product identification
\begin{equation} \Omega_n = \Omega_j \otimes \Omega_{n-j} = 
\left(
\QQ[x_1, \dots, x_j] \otimes \wedge \{ \theta_1, \dots, \theta_j \}
\right) \otimes
\left(
\QQ[x_{j+1}, \dots, x_n] \otimes \wedge \{ \theta_{j+1}, \dots, \theta_n \}
\right).
\end{equation}
We will frequently make use of the following simple properties of this decomposition.

\begin{proposition}
\label{tensor-basic-facts}
Consider the tensor product decomposition $\Omega_n = \Omega_j \otimes \Omega_{n-j}$.
\begin{enumerate}
\item If $f, f' \in \Omega_j$ and $g, g' \in \Omega_{n-j}$ then
\begin{equation*}
(f \otimes g) \odot (f' \otimes g') = (f \odot f') \otimes (g \odot g').
\end{equation*}
\item  If $u \in \symm_j, v \in \symm_{n-j}, f \in \Omega_j,$ and $g \in \Omega_{n-j}$
then the action of $u \times v \in \symm_j \times \symm_{n-j} \subseteq \symm_n$ on
$f \otimes g$ is given by
\begin{equation*}
(u \times v) \cdot (f \otimes g) = (u \cdot f) \otimes (v \cdot g).
\end{equation*}
\item  If $f \in \varepsilon_j \Omega_j$ and $g \in \Omega_{n-j}$ then 
\begin{equation*}
(f \otimes g) \odot \delta_{n,k,s} \in \eta_j \HHH_{n,k,s}.
\end{equation*}
\end{enumerate}
\end{proposition}

\begin{proof}
Items (1) and (2) are straightforward and left to the reader. For Item (3), let $u \in \symm_j$.
We calculate
\begin{align}
u \cdot \left[ (f \otimes g) \odot \delta_{n,k,s} \right] &=
\left[ u \cdot f \otimes g \right] \odot \left[u \cdot \delta_{n,k,s} \right] \\
&= \left[ \sign(u) f \otimes g \right] \odot \left[\sign(u) \delta_{n,k,s} \right] \\
&= \sign(u)^2 \times u \cdot \left[ (f \otimes g) \odot \delta_{n,k,s} \right] \\
&= u \cdot \left[ (f \otimes g) \odot \delta_{n,k,s} \right].
\end{align}
The first equality uses (2), the second equality uses $f \in \varepsilon_j \Omega_j$ and
$\delta_{n,k,s} \in \varepsilon_j \Omega_n$, and the third equality is bilinearity.
\end{proof}

Proposition~\ref{tensor-basic-facts} (3) gives rise to a `duality' between the images of the 
$\WWW$-modules under $\varepsilon_j$ and the images of the $\HHH$-modules under
$\eta_j$. We state this duality as follows.

\begin{proposition}
\label{w-h-duality}
Let $\AAA \subseteq \varepsilon_j \Omega_n$ be a subset of homogeneous $\theta$-degree $r$.
Define a subset $\AAA^{\vee} \subseteq \HHH_{n,k,s}^{(n-s-r)}$ by 
\begin{equation*}
\AAA^{\vee} := \{ f \odot \delta_{n,k,s} \,:\, f \in \AAA \}.
\end{equation*}
\begin{enumerate}
\item  $\AAA$ descends to a linearly independent subset of $\varepsilon_j \WWW_{n,k,s}^{(r)}$ if and only if 
$\AAA^{\vee}$ is linearly independent in $\eta_j \HHH_{n,k,s}^{(n-s-r)}$.
\item  $\AAA$ descends to a spanning subset of $\varepsilon_j \WWW_{n,k,s}^{(r)}$
 if and only if $\AAA^{\vee}$ spans
$\eta_j \HHH_{n,k,s}^{(n-s-r)}$.
\item  $\AAA$ descends to a basis of $\varepsilon_j \WWW_{n,k,s}^{(r)}$ if and only if 
$\AAA^{\vee}$ is a basis of $\eta_j \HHH_{n,k,s}^{(n-s-r)}$.
\end{enumerate}
\end{proposition}

\begin{proof}
Proposition~\ref{tensor-basic-facts} (3) shows that $\AAA^{\vee}$ is indeed a subset of 
$\eta_j \HHH_{n,k,s}^{(n-s-r)}$.
The isomorphism 
$\WWW_{n,k,s} = \Omega_n/ \ann \, \delta_{n,k,s} \xrightarrow{\, \, \sim \, \, } \HHH_{n,k,s}$
given by $f \mapsto f \odot \delta_{n,k,s}$ implies Items (1) through (3).
\end{proof}

\subsection{A spanning subset of $\varepsilon_j \WWW_{n,k,s}^{(r)}$}
Proposition~\ref{w-h-duality} allows us to move back and forth between the alternating subspace
$\varepsilon_j \WWW_{n,k,s}$ and the invariant subspace $\eta_j \HHH_{n,k,s}$.
The following subset $\BBB \subseteq \varepsilon_j \Omega_n$ will turn out to descend to a basis
of $\varepsilon_j \WWW_{n,k,s}^{(r)}$.

\begin{defn}
\label{B-defn}
Define a subset $\BBB \subseteq \Omega_n$ by
\begin{equation} \BBB := 
\bigsqcup_{\substack{a,b \geq 0 \\ a \leq r, \, \, b \leq s}}  \bigsqcup_{\iii}
\left\{
\varepsilon_j \cdot (x_1^{i_1} \cdots x_j^{i_j} \cdot \theta_1 \cdots \theta_a) \otimes m \,:\,
m \in \MMM_{n-j,k,s-b}^{(r-a)}
\right\}.
\end{equation}
where the index $\iii = (i_1, \dots, i_j)$ of the inner disjoint union
 ranges over all length $j$ integer sequences whose first $a$, next $b$, and final $j-a-b$ entries
 satisfy the conditions
\begin{equation*}
0 \leq i_1 \leq \cdots \leq i_a \leq k-s-1+b, \quad
0 \leq i_{a+1} < \cdots < i_{a+b} \leq s-1, \quad 
s \leq i_{a+b+1} < \cdots < i_j \leq k-1.
\end{equation*}
The subset $\BBB \subseteq \Omega_n$ depends on $n,k,s,r,$ and $j$, but we suppress this dependence
to reduce notational clutter.
\end{defn}

Thanks to our tensor product notation, the monomials $m$ appearing in 
Definition~\ref{B-defn}
are automatically elements of
$\Omega_{n-j} = \QQ[x_{j+1}, \dots, x_n] \otimes \wedge \{ \theta_{j+1}, \dots, \theta_n \}$.
We leave it to the reader to verify that the union in Definition~\ref{B-defn} is disjoint.
The conditions on $a, b$, and the sequences $\iii = (i_1, \dots, i_j)$ appearing in 
Definition~\ref{B-defn} may look complicated, but they will combinatorially correspond to the sum and 
$q$-binomial coefficients in the skewing recursion of Lemma~\ref{c-skewing-lemma} satisfied by the 
$C$-functions. Algebraically, they are obtained by applying $\varepsilon_j$ to every monomial
in $\MMM_{n,k,s}^{(r)}$ and removing `obvious' linear dependencies.

\begin{lemma}
\label{alternating-spanning-lemma}
The subset $\BBB$ 
of $\Omega_n$ descends to a spanning set of $\varepsilon_j \WWW_{n,k,s}^{(r)}$.
\end{lemma}

\begin{proof}
By Theorem~\ref{M-is-basis}, we know that $\MMM_{n,k,s}^{(r)}$ descends to a basis for $\WWW_{n,k,s}^{(r)}$.
This implies that 
\begin{equation}
\label{first-approximation-to-spanning-set}
\varepsilon_j \MMM_{n,k,s}^{(r)} := \left\{ \varepsilon_j \cdot m_0 \,:\, m_0 \in \MMM_{n,k,s}^{(r)} \right\}
\end{equation}
descends to a spanning set of $\varepsilon_j \WWW_{n,k,s}^{(r)}$. We proceed to remove 
obvious linear dependencies from the set $\varepsilon_j \MMM_{n,k,s}^{(r)}$.

Let $m_0 \in \MMM_{n,k,s}^{(r)}$. There exists a permutation $w \in \symm_j$ such that 
$w \cdot m_0 = x_1^{i_1} \cdots x_j^{i_j} \theta_1 \cdots \theta_a \otimes m$, where
$0 \leq i_1 \leq \cdots \leq i_a \leq k-1$, 
$0 \leq i_{a+1} \leq \cdots \leq i_j \leq k-1$.
We have 
\begin{equation}
\varepsilon_j \cdot (w \cdot m_0) = \sign(w)  \varepsilon_j  \cdot m_0 = \pm \varepsilon_j \cdot m_0.
\end{equation}
Furthermore, the action of $\varepsilon_j$ annihilates $w \cdot m_0$ unless $i_{a+1} < \cdots < i_j$.
It follows that $\varepsilon_j \WWW_{n,k,s}^{(r)}$ is spanned by
\begin{equation}
\label{second-approximation-to-spanning-set}
\bigsqcup_{\substack{a,b \geq 0 \\ a \leq r, \, \, b \leq s}}
\left\{
\varepsilon_j \cdot(x_1^{i_j} \cdots x_j^{i_1} \cdot \theta_{j-a+1} \cdots \theta_j) \otimes m   \,:\,
\begin{array}{c}
(x_1^{i_1} \cdots x_j^{i_j} \cdot \theta_1 \cdots \theta_a) \otimes m \in \MMM_{n,k,s}^{(r)}, \\
0 \leq i_1 \leq \cdots \leq i_a \leq k-1, \\
0 \leq i_{a+1} < \cdots < i_{a+b} \leq s-1, \\
s \leq i_{a+b+1} < \cdots < i_j \leq k-1,
\end{array}
\right\}
\end{equation}
where the parameter $b$ tracks the index at which the sequence 
$0 \leq i_{a+1} < \cdots < i_j \leq k-1$ exceeds the value $s$.
The `variable reversal'
\begin{equation*}
x_1^{i_1} \cdots x_j^{i_j} \cdot \theta_{1} \cdots \theta_a
\leadsto
x_1^{i_j} \cdots x_j^{i_1} \cdot \theta_{j-a+1} \cdots \theta_j 
\end{equation*}
in \eqref{second-approximation-to-spanning-set} relative to the definition of $\BBB$
introduces a sign upon application of $\varepsilon_j$ and is harmless.

For which monomials $m \in \Omega_{n-j}$ and sequences $(i_1, \dots, i_j)$ do we have the containment
\begin{center}
$(x_1^{i_j} \cdots x_j^{i_1} \cdot \theta_{j-a+1} \cdots \theta_j) \otimes m \in \MMM_{n,k,s}^{(r)}$?
\end{center}
The lemma is reduced to the following claim.

{\bf Claim:} 
{\em Given sequences $0 \leq i_1 \leq \cdots \leq i_a \leq k-1,
0 \leq i_{a+1} < \cdots < i_{a+b} \leq s-1$, and
$s \leq i_{a+b+1} < \cdots < i_j \leq k-1$ and a monomial $m \in \Omega_{n-j}$, we have 
\begin{center}
$(x_1^{i_j} \cdots x_j^{i_1} \cdot \theta_{j-a+1} \cdots \theta_j) \otimes m \in \MMM_{n,k,s}^{(r)}$
if and only if $m \in \MMM_{n-j,k,s-b}^{(r-a)}$ and $i_j \leq k-s-1+b$.
\end{center}}

The reason for this `reversal' in 
\eqref{second-approximation-to-spanning-set} is to make the monomials in the claim 
divisible by staircase monomials. We leave the proof of the claim to the reader.
\end{proof}

\subsection{$\BBB$ is linearly independent}
The goal of this technical subsection is to show that the set $\BBB$ in 
Definition~\ref{B-defn} is linearly independent in $\Omega_n$.  By virtue of 
Lemma~\ref{alternating-spanning-lemma}, this implies that $\BBB$ descends to a basis 
of $\varepsilon_j \WWW_{n,k,s}^{(r)}$.

In the context of orbit harmonics, independence results of this kind 
can often be proven immediately from a spanning result
such as Lemma~\ref{alternating-spanning-lemma}. Knowing the ungraded $\symm_n$-structure 
of a graded $\symm_n$-module $V$ often gives enough information to determine the dimension
$\dim(\varepsilon_j V)$ of its alternating subspace $\varepsilon_j V$. 
Since orbit harmonics are not yet available in superspace, we show that $\BBB$ is linearly independent more
directly by verifying that the set
\begin{equation}
\BBB^{\vee} := \{ f \odot \delta_{n,k,s} \,:\, f \in \BBB \} \subseteq \eta_j \Omega_n
\end{equation}
 is linearly independent and applying Proposition~\ref{w-h-duality}.



We will show that $\BBB^{\vee}$ is linearly independent by considering a strategic basis of the space 
$\eta_j \Omega_n$ and showing that the expansions of the $f \odot \delta_{n,k,s}$
satisfy in this basis satisfy a triangularity condition. 
Both our basis and the triangularity condition will be defined in terms of the superlex order $\prec$.

\begin{defn}
A monomial $m \in \Omega_n$ 
with exponent sequence $u = (u_1, \dots, u_n)$
 is {\em $j$-increasing} if
$u_1 \leq \cdots \leq u_j$ under the order
\begin{equation*}
\cdots < \bar{3} < \bar{2} < \bar{1} < \bar{0} < 0 < 1 < 2 < 3 < \cdots 
\end{equation*}
on letters defining superlex order $\prec$.
\end{defn}

Each orbit of the $\symm_j$-action on the first $j$ letters of exponent sequences $(u_1, \dots, u_n)$
of monomials in $\Omega_n$ has a unique $j$-increasing representative. 
In terms of the operator $\eta_j$, we have the following.

\begin{observation}
\label{obs:increasing-basis}
The nonzero
elements in  $\{\eta_j \cdot m : m \text{ is $j$-increasing}\}$ form a basis for $\eta_j \Omega_n$. 
\end{observation}

The word `nonzero' is necessary in Observation~\ref{obs:increasing-basis}. Indeed, when
$n = 3$ and $j = 2$, the word $(\bar{0}, \bar{0}, 0)$ is $j$-increasing and
$\eta_2 \cdot \theta_1 \theta_2  = 0$.

For any $f \in \BBB$, the element $f \odot \delta_{n,k,s} \in \eta_j \Omega_n$ may be uniquely 
expanded in the basis of Observation~\ref{obs:increasing-basis}.
The next definition extracts a useful `leading term' of this expansion. This is a variant of the superlex
order which incorporates the parameter $j$.

\begin{defn}
\label{j-initial-term}
Given any nonzero $g \in \eta_j \Omega_n$, let $\initial_j(g)$ be the $j$-increasing monomial $m$ such that
\begin{itemize}
\item $\eta_j \cdot m$ appears in the expansion of $g$ in the basis of
Observation~\ref{obs:increasing-basis} with nonzero coefficient, and
\item for any $j$-increasing monomial $m'$ such that $\eta_j \cdot m'$ 
appears in this expansion with nonzero coefficient, we have $m \succeq m'$.
\end{itemize}
\end{defn}

As an example of these notions, for $n=5$, $k=4$, $s=2$, and $j=3$ we consider the superpolynomial $f \in \BBB$ given by
\begin{align}
\label{ex:j-sym}
f := \varepsilon_3 \cdot ( x^{00110} \theta_1 ) =
 (x_3 - x_2) x_4 \theta_1 + (x_1 - x_3) x_4 \theta_2 + (x_2 - x_1) x_4 \theta_3.
\end{align}
Then
\begin{align}
\label{ex:j-leading}
 f \odot \delta_{5,4,2} = &-18\eta_3 \cdot (x^{00323} \theta_{45}) + 54 \eta_3 \cdot (x^{21320} \theta_{14}) + 18 \eta_3 \cdot (x^{20303} \theta_{15})  \\ \nonumber&- 54 \eta_3 \cdot (x^{20321} \theta_{14}) + 18 \eta_3 \cdot ( x^{30320} \theta_{14}) - 18 \eta_3 \cdot (x^{32300} \theta_{12}).
 \end{align}
 We have written the terms in the final expression so that each monomial that appears is $j$-increasing and these monomials decrease in superlex order from left to right. Therefore $$\initial_j(f \odot \delta_{n,k,s}) = x^{00323} \theta_{45}.$$
The most important property of 
 $\initial_j$ is as follows.

 \begin{lemma}
\label{lemma:j-leading}
The map
$$\initial_j : \BBB^{\vee} \longrightarrow \{\text{$j$-increasing monomials in }  \Omega_n\}$$
is injective. 
\end{lemma} 


\begin{proof}
We give a method for finding $\initial_j(f \odot \delta_{n,k,s})$ for any $f \in \BBB$. Suppose that 
$$f = \varepsilon_j \cdot \left( x_1^{i_1} \cdots x_j^{i_j} \theta_{1} \cdots \theta_a\right) \otimes m \in \BBB$$
so $m \in \MMM_{n-j,k,s-b}^{(r-a)}$ and 
$$0 \leq i_1 \leq \cdots \leq i_a \leq k-s-1+b, \ 0 \leq i_{a+1} < \cdots < i_{a+b} \leq s-1, \ s \leq i_{a+b+1} < \cdots < i_j \leq k-1.$$
We define a reordering $(h_1, \dots, h_j)$ of the sequence $(i_1, \dots, i_j)$ by
$$
(h_1, \dots, h_j) := (i_{a+b+1}, i_{a+b+2}, \dots, i_{j-1}, i_j, i_{a+b}, i_{a+b-1}, \dots, i_2, i_1).
$$
The sequence $(h_1, \dots, h_j)$ satisfies
$$s \leq h_1 < \cdots < h_{j-a-b} \leq k-1, \ s-1 \geq h_{j-a-b+1} > \cdots > h_{j-a} \geq 0, \ k-s-1+b \geq h_{j-a+1} \geq \cdots \geq h_j \geq 0.$$
Furthermore, let $t = (t_1, \dots, t_n)$ be the exponent sequence of the entire monomial 
$$ \left( x_1^{h_1} \cdots x_j^{h_j} \theta_{j-a+1} \cdots \theta_j \right) \otimes m.$$
We construct a monomial $M$ with exponent sequence $u$ as follows:
\begin{itemize}
\item $u_1 = \cdots = u_{j-a-b} = \overline{k-1}$,
\item $(u_{j-a-b+1},  \ldots,  u_{j-a}) =  (s-1, s-2, \ldots, s-b)$,
\item $u_{j-a+1} = \cdots = u_j = \overline{k-1}$, and
\item For $p=j+1$ to $n$, $u_p$ is $\overline{k-1}$ if $t_p$ is barred or greater than the largest unused element of $\{0,1,\ldots,s-b-1\}$. Otherwise, we let $u_p$ be the largest unused element of $\{0,1,2\ldots,s-b-1\}$.
\end{itemize}
The last bullet point above is an instance of Algorithm~\ref{alg:U} in the proof of 
Proposition~\ref{involution-proposition}.
Since $m$ is substaircase, this algorithm successfully produces a monomial $M$ appearing in $\delta_{n,k,s}$. 

\textbf{Claim: } $\initial_j(f \odot \delta_{n,k,s}) \doteq \left( x_1^{h_1} \cdots x_j^{h_j} \theta_{1} \cdots \theta_a \otimes m \right) \odot M.$

We check this construction for the example $f =\varepsilon_3 ( x^{00110} \theta_1 ) \in \BBB$ given in \eqref{ex:j-sym} where $n=5$, $k=4$, $s=2$, and $j=3$. Then $a=1$, $b=2$, $t = (1,0,\bar{0},1,0)$, 
$u = (1,0,\bar{3},\bar{3},\bar{3})$, and 
$$  x^{10010} \theta_3  \odot x^{10333} \theta_{345} \doteq x^{00323} \theta_{45}$$
which agrees with the computation of $\initial_j(f \otimes \delta_{5,4,2})$ in \eqref{ex:j-leading}.

Now we prove the claim.  Let  $v = (v_1, \dots, v_n)$ be the exponent sequence of the monomial
$( x_1^{h_1} \cdots x_j^{h_j} \theta_{j-a+1} \cdots \theta_j \otimes m ) \odot M$ appearing in the claim. 
First, we check that $v$ is indeed $j$-increasing.
 From the definitions of $t$ and $u$, the first $j-a-b$ entries of $v$ are all barred and
$$k-s-1 \geq v_1 > \cdots > v_{j-a-b} \geq 0.$$
The next $a+b$ entries of $v$ are all unbarred and satisfy
$$0 \leq v_{j-a-b+1} \leq \cdots \leq v_{j-a} < s -b \leq v_{j-a+1} \leq \cdots \leq v_j \leq k-1$$
so $v$ is $j$-increasing.

Let $v' = (v'_1, \dots, v'_n)$ be the $\prec$-maximal 
 $j$-increasing exponent sequence whose monomial appears in $f \odot \delta_{n,k,s}$. 
There are monomials
 appearing in $f$ and $\delta_{n,k,s}$ that yield $v'$ under the left $\odot$ action of $f$ on $\delta_{n,k,s}$. 
Denote the exponent
sequences of these monomials  by $t' = (t'_1, \dots, t'_n)$ and $u' = (u'_1, \dots, u'_n)$, respectively. 
Then $t'$ is obtained from $t$ by some rearrangement of its first $j$ entries.
We show that $v = v'$ as follows.

The sequence $(t'_1, \dots, t'_j)$ has $j-a-b$ unbarred entries $> s$. Thus,
the sequence  $(v'_1, \dots, v'_j)$ must have at least $j-a-b$ barred entries.
 Since $v'$ is $j$-increasing, 
 the first $j-a-b$ entries of $v'$ must be barred, which forces
 \begin{center}
 $u'_1 = \cdots = u'_{j-a-b} = \overline{k-1}$,
which agrees with
 $u_1 = \cdots = u_{j-a-b} = \overline{k-1}$.
 \end{center}
 Since $v \preceq v'$, and $v'$ is $j$-increasing, the barred entries
 $v'_1, \dots, v'_{j-a-b}$ are all  $< k-s$.
We see that the first $j-a-b$ entries of $t'$ are unbarred and $> s$.  Since the first $j$ entries of $t'$
are a rearrangement of the first $j$ entries of $t$, this forces the first $j$ entries of $t'$ to be
$t_1, \dots, t_{j-a-b}$ (in some order).
Since $v'$ is $j$-increasing we have 
\begin{center}
$(t'_1, \dots, t'_{j-a-b}) = (t_1, \dots, t_{j-a-b})$ so that
$(v'_1, \dots, v'_{j-a-b}) = (v_1, \dots, v_{j-a-b})$.
\end{center}

The above paragraph forces 
 $(t'_{j-a-b+1}, \dots, t'_j)$ to be a rearrangement of $(t_{j-a-b+1}, \dots, t_j)$. Both of these sequences
 contain
\begin{itemize}
\item
$b$ unbarred entries $< s$, all of which are unique, and 
\item
$a$ barred entries which are all $< k-s-1$. 
\end{itemize}
The fact that $v'$ is the superlex maximal $j$-increasing exponent sequence 
appearing in $f \odot \delta_{n,k,s}$ has the following consequences.
\begin{enumerate}
\item  Every entry in the subsequence $(v'_{j-a-b+1}, \dots, v'_j)$ is unbarred (since $v \preceq v'$).
  Consequently, 
the corresponding $b$ entries of $u'$ are unbarred and
 the corresponding $a$ entries of $u'$ are barred.
 \item  The corresponding $b$ entries in $u'$ must be $s-1, s-2, \dots, s-b$ (in some order)
 and the corresponding $a$ entries of $u'$ must be $\overline{k-1}$.
 \item The $b$ entries must come before the $a$ entries.
\end{enumerate}
Items (2) and (3) above imply
\begin{center}
$(u'_{j-a+1}, \dots, u'_{j}) = (\overline{k-1}, \dots, \overline{k-1}) = (u_{j-a+1}, \dots, u_j)$
\end{center}
and since $v'$ is $j$-increasing we have
\begin{center}
$(t'_{j-a+1}, \dots, t'_j) = (t_{j-a+1}, \dots, t_j)$ so that
$(v'_{j-a+1}, \dots, v'_j) = (v_{j-a+1}, \dots, v_n)$.
\end{center}
We see that the three pairs of length $b$ sequences 
\begin{equation*}
\begin{cases}
\text{$(u'_{j-a-b+1}, \dots, u'_{j-a})$ and $(u_{j-a-b+1}, \dots, u_{j-a})$}, \\ 
\text{$(t'_{j-a-b+1}, \dots, t'_{j-a})$ and $(t_{j-a-b+1}, \dots, t_{j-a})$},  \\
\text{$(v'_{j-a-b+1}, \dots, v'_{j-a})$ and $(v_{j-a-b+1}, \dots, v_{j-a})$}
\end{cases}
\end{equation*}
are rearrangements of each other and
 $(u_{j-a-b+1}, \dots, u_{j-a}) = (s-1, s-2, \dots, s-b)$. We may apply a simultaneous permutation of 
 $(u'_{j-a-b+1}, \dots, u'_{j-a})$ and $(t'_{j-a-b+1}, \dots, t'_{j-a})$  to get 
 \begin{center}
  $(u'_{j-a-b+1}, \dots, u'_{j-a}) = (s-1, s-2, \dots, s-b) = (u_{j-a-b+1}, \dots, u_{j-a})$ 
  \end{center}
  without affecting $(v'_{j-a-b+1}, \dots v'_{j-a})$.  Since $v'$ is $j$-increasing, and the entries
  in $(t'_{j-a-b+1}, \dots, t'_{j-a})$  are distinct we have 
  $t'_{j-a-b+1} > \cdots > t'_{j-a}$ which implies
  \begin{center}
  $(t'_{j-a-b+1}, \dots, t'_{j-a}) = (t_{j-a-b+1}, \dots, t_{j-a})$ so that
  $(v'_{j-a-b+1}, \dots, v'_{j-a}) = (v_{j-a-b+1}, \dots, v_{j-a})$.
  \end{center}
  
The last two paragraphs show that the first $j$ entries of $v'$ and $v$ coincide. 
The fact that the last $n-j$ entries of $v'$ and $v$ coincide follows from an argument 
similar to the proof of Proposition~\ref{involution-proposition}; it is omitted here.
This completes the proof of the Claim.

Our Claim implies that
  $\initial_j(f \odot \delta_{n,k,s})$ is the monomial associated to $v = (v_1, \dots, v_n)$.
We show how to recover the exponent sequence  $u = (u_1, \dots, u_n)$ of the monomial $M$
appearing in the Claim.
In turn, this allows us to recover
$x_1^{i_1} \cdots x_j^{i_j} \theta_{1} \cdots \theta_a$ and $m$ such that 
$$f = \varepsilon_j \cdot \left( x_1^{i_1} \cdots x_j^{i_j} \theta_{1} \cdots \theta_a\right) \otimes m,$$
completing the proof of the lemma.

Since $m \in \MMM_{n-j,k,s-b}^{(r-a)}$ and the construction of the last $n-j$ entries of $u$ and $v$ 
follows the algorithms in Proposition~\ref{involution-proposition}.
That proposition proves we can recover the last $n-j$ entries of $u$ 
and the last $n-j$ entries of $t$ from the last $n-j$ entries of $v$. 
Since $r$ is a global parameter and $m$ has degree $r-a$ in the $\theta_i$ variables, 
knowing the last $n-j$ entries of $t$ recovers $a$. 
This gives $u_{j-a+1} = \cdots = u_j = \overline{k-1}$ and, 
since we know these entries of $v$, we recover $t_{j-a+1}$ through $t_j$. 
By construction, $v$ begins with $j-a-b$ barred entries. 
Since we know $j$ and $a$, we learn $b$. 
Then we know that the first $j-a$ entries of $u$ are $j-a-b$ copies of 
$\overline{k-1}$ followed by the sequence $(s-1)(s-2) \dots (s-b)$. From this information and the corresponding entries of $v$, we can recover $t_1$ through $t_{j-a}$. Now we know all of $t$. Since we know all of $v$,
this determines all of $u$.
\end{proof}

Finally, we put everything together to get our desired basis of $\varepsilon_j \WWW_{n,k,s}^{(r)}$.

\begin{lemma}
\label{linearly-independent-lemma}
The subset of $\BBB \subseteq \Omega_n$ descends to
a basis of $\varepsilon_j \WWW^{(r)}_{n,k,s}$.
\end{lemma}

\begin{proof}
Lemma~\ref{lemma:j-leading} implies that $\BBB^{\vee}$ is linearly independent in $\Omega_n$, so 
Proposition~\ref{w-h-duality} shows that $\BBB$ is linearly independent
in $\WWW^{(r)}_{n,k,s}$. Now apply Lemma~\ref{alternating-spanning-lemma}.
\end{proof}

\subsection{The $\CCC$  basis of $\eta_j \HHH_{n,k,s}^{(r)}$}
Lemma~\ref{linearly-independent-lemma} gives a basis $\BBB$ of $\varepsilon_j \WWW^{(r)}_{n,k,s}$.
The corresponding harmonic basis is as follows.

\begin{defn}
\label{C-defn}
Let $\CCC \subseteq \Omega_n$  be the disjoint union
\begin{equation} \CCC := 
\bigsqcup_{\substack{a,b \geq 0 \\ a \leq n-s-r, \, \, b \leq s}} \bigsqcup_{\iii}
\left\{
\left[ \varepsilon_j \cdot (x_1^{i_1} \cdots x_j^{i_j} \cdot \theta_1 \cdots \theta_a) \otimes m \right]
\odot \delta_{n,k,s} \,:\,
m \in \MMM_{n-j,k,s-b}^{(n-s-r-a)}
\right\}.
\end{equation}
where the index $\iii = (i_1, \dots, i_j)$ of the inner disjoint union
 ranges over all length $j$ integer sequences whose first $a$, next $b$, and final $j-a-b$ entries
 satisfy the conditions
\begin{equation*}
0 \leq i_1 \leq \cdots \leq i_a \leq k-s-1+b, \quad
0 \leq i_{a+1} < \cdots < i_{a+b} \leq s-1, \quad 
s \leq i_{a+b+1} < \cdots < i_j \leq k-1.
\end{equation*}
The subset $\CCC \subseteq \Omega_n$ depends on $n,k,s,r,$ and $j$, but we suppress this dependence
to avoid notational clutter.
\end{defn}

Definition~\ref{C-defn} is formulated so that $\CCC \subseteq \HHH_{n,k,s}^{(r)}$.
The following result justifies the disjointness of the unions appearing in Definition~\ref{C-defn}.

\begin{lemma}
\label{harmonic-invariant-lemma}
The family $\CCC \subseteq \HHH_{n,k,s}^{(r)}$ is a basis of the invariant space
$\eta_j \HHH_{n,k,s}^{(r)}$.
\end{lemma}

\begin{proof}
Apply Proposition~\ref{w-h-duality} and Lemma~\ref{linearly-independent-lemma}.
\end{proof}

Thanks to its avoidance of cosets, the $\CCC$ basis will be more convenient for us going forward.

\subsection{The $\triangleleft$ order on bidegrees}
Our goal is to use the basis $\CCC$ of Lemma~\ref{harmonic-invariant-lemma}
to show that $\eta_j \HHH_{n,k,s}^{(r)}$ is isomorphic as a graded $\symm_{n-j}$-module
to a direct sum of graded shifts of smaller $\HHH$-modules in a way that matches 
the recursion of Lemma~\ref{c-skewing-lemma}.
To do this, we introduce a strategic direct sum decomposition of
$\Omega_n = \Omega_j \otimes \Omega_{n-j}$, place a total
order $\triangleleft$ on the pieces of this decomposition, and examine the $\triangleleft$-lowest 
components
of the elements in $\CCC$.

Our direct sum decomposition of $\Omega_n = \Omega_j \otimes \Omega_{n-j}$ is defined as follows.
The second factor $\Omega_{n-j}$ of the tensor product $\Omega_n = \Omega_j \otimes \Omega_{n-j}$
is a bigraded algebra
\begin{equation}
\Omega_{n-j} = \bigoplus_{p, q \geq 0} (\Omega_{n-j})_{p,q}
\end{equation}
where 
\begin{equation}
(\Omega_{n-j})_{p,q} = \QQ[x_{j+1}, \dots, x_n]_p \otimes \wedge^q \{ \theta_{j+1}, \dots, \theta_n \}.
\end{equation}
We therefore have a direct sum
decomposition
\begin{equation}
\label{omega-j-decomposition}
\Omega_n = \bigoplus_{p, q \geq 0} \Omega_n(p,q)
\end{equation}
of the larger algebra $\Omega_n$, where we set
\begin{equation}
\Omega_n(p,q) := \Omega_j \otimes (\Omega_{n-j})_{p,q}.
\end{equation}
In other words, the bigrading $\Omega_n(-,-)$ on $\Omega_n$ is obtained by 
focusing on commuting and anticommuting degree in the last $n-j$ indices only.
In particular, we have the following important observation.

\begin{observation}
\label{parabolic-stability}
Although the direct sum decomposition $\Omega_n = \bigoplus_{p,q \geq 0} \Omega_n(p,q)$
is not stable under the action of $\symm_n$, it is stable under the parabolic subgroup 
$\symm_j \times \symm_{n-j}$ and the further subgroup $\symm_{n-j}$.
\end{observation}

Th decomposition \eqref{omega-j-decomposition} 
will be used to study the action of $\symm_{n-j}$ on $\Omega_n$, and in particular on the basis
$\CCC$ of $\eta_j \HHH_{n,k,s}^{(r)}$. We will need a nonstandard notion of `lowest degree component' 
for elements of $\CCC$. To this end,
we introduce a total order $\triangleleft$ on the summands 
$\Omega_n(p,q)$ appearing in 
\eqref{omega-j-decomposition} as follows.

\begin{defn}
\label{triangle-order}
Given two pairs of nonnegative integers $(p, q)$ and $(p', q')$, write $(p, q) \triangleleft (p',q')$ if 
\begin{itemize}
\item we have $q' < q$, or
\item we have $q' = q$ and $p' > p$.
\end{itemize}
We also use the symbol $\triangleleft$ to denote the induced order on the summands $\Omega_n(p,q)$ of 
the direct sum~\eqref{omega-j-decomposition}.
\end{defn}

If there were no $\theta$-variables, Definition~\ref{triangle-order} would be the usual
degree order induced from the second factor of the tensor decomposition
$\QQ[x_1, \dots, x_n] = \QQ[x_1, \dots, x_j] \otimes \QQ[x_{j+1}, \dots, x_n]$.
The order $\triangleleft$ first compares $\theta$-degrees  in the `opposite' order
and breaks ties by comparing $x$-degrees in the classical order.
This superization of polynomial degree order should be compared with
the superization of the lexicographical term order $\prec$ in Definition~\ref{lex-order-defn}
in which anticommuting variables involve a similar `reversal'.

\subsection{Factoring the $\triangleleft$-lowest components of the $\CCC$ basis}
We consider elements in the $\CCC$ basis of $\eta_j \HHH_{n,k,s}^{(r)}$ with respect to the 
direct sum decomposition $\Omega_n = \bigoplus_{p,q \geq 0} \Omega_n(p,q)$
of \eqref{omega-j-decomposition}. While these elements are almost always inhomogeneous members
of this direct sum, their $\triangleleft$-lowest components have a useful recursive structure.

\begin{lemma}
\label{leading-degree-observation}
Let $\left[ \varepsilon_j \cdot (x_1^{i_1} \cdots x_j^{i_j} \cdot \theta_1 \cdots \theta_a) \otimes m \right]
\odot \delta_{n,k,s} \in \CCC$ where $m \in \MMM_{n-j,k,s-b}^{(n-s-r-a)}$ and the sequence
$\iii = (i_1, \dots, i_j)$ satisfies
\begin{equation*}
0 \leq i_{a+1} < \cdots < i_{a+b} < s \leq i_{a+b+1} < \cdots < i_j \leq k-1 \quad \text{and} \quad
0 \leq i_1 \leq \cdots \leq i_a \leq k-s-1+b.
\end{equation*}
This element of $\CCC$ has an expansion under $\triangleleft$ of the form
\begin{equation}
\left[ \varepsilon_j \cdot (x_1^{i_1} \cdots x_j^{i_j} \cdot \theta_1 \cdots \theta_a) \otimes m \right]
\odot \delta_{n,k,s} =
\pm f_{a,b,\iii} \otimes [m \odot \delta_{n-j,k,s-b}] +  \text{\rm{greater terms under $\triangleleft$}} 
\end{equation}
where 
\begin{equation}
f_{a,b,\iii} := \left[ \varepsilon_j \cdot (x_1^{i_1} \cdots x_j^{i_j} \cdot \theta_1 \cdots \theta_a) \right] \odot
\left[ \varepsilon_j \cdot m_{a,b,\iii} \right]
\end{equation}
and 
\begin{equation}
m_{a,b,\iii} := x_1^{k-1} \cdots x_a^{k-1} x_{a+1}^{s-b} x_{a+2}^{s-b+1} \cdots x_{a+b}^{s-1} 
x_{a+b+1}^{k-1} \cdots x_j^{k-1} \cdot \theta_1 \theta_2 \cdots  \theta_a \cdot
\theta_{a+b+1} \theta_{a+b+2} \cdots \theta_j.
\end{equation}
The polynomial $f_{a,b,\iii} \in \Omega_j$ is $\symm_j$-invariant and nonzero.
\end{lemma}

\begin{remark}
The polynomial $f_{a,b,\iii}$ is bihomogeneous of $x$-degree 
\begin{equation}
(k-1) \cdot (j-b) + b \cdot (s-b) + {b \choose 2} - i_1 - \cdots - i_j 
\end{equation}
and $\theta$-degree $j - b$.
The $\symm_j$-invariance of $f_{a,b,\iii}$ is justified by Proposition~\ref{tensor-basic-facts}.
Observe that $f_{a,b,\iii}$ depends on $a, b,$ and $\iii$, but is independent of 
$m \in \MMM_{n-j,k,s-b}^{(n-s-r-a)}$.
\end{remark}

To better understand Lemma~\ref{leading-degree-observation}, we analyze 
its statement in the case of the classical Vandermonde $\delta_n$.

\begin{example}
In the special case $(n,k,s) = (n,n,n)$ and $j = 1$, Lemma~\ref{leading-degree-observation}
follows from the following expansion for the Vandermonde determinant $\delta_n$ in the variable $x_1$:
\begin{equation}
\delta_n = \varepsilon_n \cdot (x_1^{n-1} x_2^{n-2} \cdots x_{n-1}^1 x_n^0) =  
\sum_{d = 0}^{n-1} (-1)^{n-d+1} x_1^d \otimes  \varepsilon_{n-1} \cdot
(x_2^{n-1}  \cdots x_{n-d}^{d+1} x_{n-d+1}^{d-1} \cdots x_n^0).
\end{equation}
Since we have no anticommuting variables, the order $\triangleleft$ is simply the degree order on the 
last $n-j = n-1$ variables $x_2, \dots, x_n$.
Therefore, the $d = n-1$ term
\begin{equation}
x_1^{n-1} \otimes \varepsilon_{n-1} \cdot (x_2^{n-2} \cdots x_n^0) = x_1^{n-1} \otimes \delta_{n-1}
\end{equation}
is the lowest $\triangleleft$-degree component of this expansion. The condition
$x_1^{c_1} x_2^{c_2} \cdots x_n^{c_n} \in \MMM_{n,n,n}$ means that 
$(c_1, c_2, \dots, c_n) \leq (n-1, n-2, \dots, 0)$ componentwise.  If we set $m := x_2^{c_2} \cdots x_n^{c_n}$,
applying $x_1^{c_1} x_2^{c_2} \cdots x_n^{c_n} = (\varepsilon_1 \cdot x_1^{c_1}) \otimes m$ to 
$\delta_n$ under the $\odot$-action yields
\begin{equation}
[(\varepsilon_1 \cdot x_1^{c_1}) \otimes m] \odot \delta_n = 
C x_1^{n - c_1} \otimes [m \odot \delta_{n-1}] + \text{greater terms under $\triangleleft$}
\end{equation}
where $C$ is a nonzero constant.
\end{example}

\begin{proof}  (of Lemma~\ref{leading-degree-observation})
By definition, the superspace Vandermonde $\delta_{n,k,s}$ is the antisymmetrization
\begin{equation}
\label{factor-lemma-first}
\delta_{n,k,s} = \varepsilon_n \cdot (x_1^{p_1} \cdots x_n^{p_n} \theta_1^{q_1} \cdots \theta_n^{q_n})
\end{equation}
where the exponent sequences are given by $(p_1, \dots, p_n) = ( (k-1)^{n-s}, s-1, s-2, \dots, 1, 0)$ and 
$(q_1, \dots, q_n) = (1^{n-s},0^s)$.
We expand Equation~\eqref{factor-lemma-first} in a fashion compatible with the 
tensor product decomposition $\Omega_n = \Omega_j \otimes \Omega_{n-j}$. For
any system of right coset representatives $w$ for $\symm_j \times \symm_{n-j}$ inside
$\symm_n$, Equation~\eqref{factor-lemma-first} reads
\begin{small}
\begin{multline}
\label{factor-lemma-second}
\delta_{n,k,s} =  \\ \sum_{w} \pm \left[ \varepsilon_j \cdot 
(x_{1}^{p_{w^{-1}(1)}} \cdots x_{j}^{p_{w^{-1}(j)}}
\theta_{1}^{q_{w^{-1}(1)}} \cdots \theta_{j}^{q_{w^{-1}(j)}}) \right]
\otimes 
\left[ \varepsilon_{n-j} \cdot 
(x_{j+1}^{p_{w^{-1}(j+1)}} \cdots x_n^{p_{w^{-1}(n)}} \theta_{j+1}^{q_{w^{-1}(j+1)}} \cdots
 \theta_n^{q_{w^{-1}(n)}}) \right]
\end{multline}
\end{small}
where $\pm$ is the sign of the coset representative $w$.
Equation~\eqref{factor-lemma-second} is true for any system of coset representatives, but 
we restrict our choice somewhat in the next paragraph.

Our aim is to apply the operator 
\begin{equation}
\label{factor-lemma-third}
\left\{   \left[\varepsilon_j \cdot (x_1^{i_1} \cdots x_j^{i_j} \theta_1 \cdots \theta_a) \right] \otimes m \right\} 
 \odot ( - ) =
\left\{ \left[\varepsilon_j \cdot (x_1^{i_1} \cdots x_j^{i_j} \theta_1 \cdots \theta_a) \right]  \odot (-) \right\} \otimes
\left\{ m \odot (-) \right\}
\end{equation}
to Equation~\eqref{factor-lemma-second} and extract the lowest term under $\triangleleft$ for which 
the first tensor factor does not vanish.
A summand of Equation~\eqref{factor-lemma-second}
 indexed by a permutation $w$ for which fewer than $a$ entries in
$(q_{w^{-1}(1)}, \dots, q_{w^{-1}(j)})$ equal 1 will be 
annihilated by first tensor factor of the operator \eqref{factor-lemma-third}.
We may therefore restrict our attention to those summands in which at least $a$ of the entries in
$(q_{w^{-1}(1)}, \dots, q_{w^{-1}(j)})$ equal 1; the corresponding entries in
$(p_{w^{-1}(1)}, \dots , p_{w^{-1}(j)})$  will equal $k-1$.
We now assume the coset representatives $w$ 
in Equation~\eqref{factor-lemma-second} are such that 
\begin{center}
$p_{w^{-1}(1)} = \cdots = p_{w^{-1}(a)} = k-1$ and 
$0 \leq p_{w^{-1}(a+1)} \leq \cdots \leq p_{w^{-1}(j)} \leq k-1$.
\end{center}
When $k = s$ we further assume further that
\begin{center}
$(q_{w^{-1}(1)}, \dots, q_{w^{-1}(j)})$ has the form $(1^a,0^{j-a-c},1^c)$ for some $a+c \leq j$
\end{center}
(this follows from the assumption on the $p$'s  when $k > s$).
There are typically many choices of coset representatives $w$ which achieve this. 

We apply the operator \eqref{factor-lemma-third} to each term
\begin{small}
\begin{equation}
\label{factor-lemma-fourth}
\left[ \varepsilon_j \cdot 
(x_{1}^{p_{w^{-1}(1)}} \cdots x_{j}^{p_{w^{-1}(j)}}
\theta_{1}^{q_{w^{-1}(1)}} \cdots \theta_{j}^{q_{w^{-1}(j)}}) \right]
\otimes 
\left[ \varepsilon_{n-j} \cdot 
(x_{j+1}^{p_{w^{-1}(j+1)}} \cdots x_n^{p_{w^{-1}(n)}} \theta_{j+1}^{q_{w^{-1}(j+1)}} \cdots
 \theta_n^{q_{w^{-1}(n)}}) \right]
\end{equation}
\end{small}
of Equation~\eqref{factor-lemma-second}. In the first tensor factor, we get
\begin{equation}
\label{factor-lemma-fifth}
\left[\varepsilon_j \cdot (x_1^{i_1} \cdots x_j^{i_j} \theta_1 \cdots \theta_a) \right]  \odot 
\left[ \varepsilon_j \cdot 
(x_{1}^{p_{w^{-1}(1)}} \cdots x_{j}^{p_{w^{-1}(j)}}
\theta_{1}^{q_{w^{-1}(1)}} \cdots \theta_{j}^{q_{w^{-1}(j)}}) \right] 
\end{equation}
The element of $\Omega_j$ in Equation~\eqref{factor-lemma-fifth} is certainly $\symm_j$-invariant,
but can vanish for some coset representatives $w$.

In order to minimize the application \eqref{factor-lemma-fifth}
of \eqref{factor-lemma-third} to \eqref{factor-lemma-fourth} under the 
order $\triangleleft$, we select our coset representative $w$ in 
\eqref{factor-lemma-fifth} such that
\begin{enumerate}
\item  the expression \eqref{factor-lemma-fifth} 
is a nonzero element of $\Omega_j$, and subject to this
\item  the expression \eqref{factor-lemma-fifth} has the smallest 
$\theta$-degree possible, and subject to this
\item  the expression \eqref{factor-lemma-fifth} has the largest 
$x$-degree possible.
\end{enumerate}
For Item (1), we see that \eqref{factor-lemma-fifth} is nonzero if and only if we have
the componentwise inequalities
\begin{equation}
\label{factor-lemma-sixth}
\text{$(i_1, \dots, i_j) \leq (p_{w^{-1}(1)}, \dots, p_{w^{-1}(j)})$ and
$(1^a, 0^{j-a}) \leq (q_{w^{-1}(1)}, \dots, q_{w^{-1}(j)}).$}
\end{equation}
Subject to this, in order to satisfy Item (2) we must have 
\begin{equation}
\label{factor-lemma-seventh}
(q_{w^{-1}(1)}, \dots, q_{w^{-1}(j)}) = (1^a, 0^b, 1^{j-a-b}).
\end{equation}
Subject to both \eqref{factor-lemma-sixth} and \eqref{factor-lemma-seventh},
to satisfy Item (3) we must have
\begin{equation}
\label{factor-lemma-eighth}
(p_{w^{-1}(1)}, \dots, p_{w^{-1}(j)}) = ( (k-1)^a, s-b, s-b+1, \dots, s-1, (k-1)^{j-a-b}).
\end{equation}
Said differently, Items (1)-(3) are satisfied precisely when
\begin{equation}
x_{1}^{p_{w^{-1}(1)}} \cdots x_{j}^{p_{w^{-1}(j)}} \theta_{1}^{q_{w^{-1}(1)}} \cdots \theta_{j}^{q_{w^{-1}(j)}} 
= m_{a,b,\iii}
\end{equation}
so that Equation~\eqref{factor-lemma-fifth} equals $f_{a,b,\iii}$.

Let $w$ be the unique coset representative which satisfies
$\eqref{factor-lemma-sixth}, \eqref{factor-lemma-seventh},$ and $\eqref{factor-lemma-eighth}$. For this 
coset representative, the exponent sequence $(p_{w^{-1}(j+1)}, \dots, p_{w^{-1}(n)})$ appearing 
in the second tensor factor of \eqref{factor-lemma-fourth} is a rearrangement
of $( (k-1)^{n-s-j+b}, s-b-1, \dots, 1, 0)$. This shows that 
\begin{equation}
\varepsilon_{n-j} \cdot 
(x_{j+1}^{p_{w^{-1}(j+1)}} \cdots x_n^{p_{w^{-1}(n)}} \theta_{j+1}^{q_{w^{-1}(j+1)}} \cdots
 \theta_n^{q_{w^{-1}(n)}})  = \pm \delta_{n-j,k,s-b}
\end{equation}
for this choice of $w$. In summary, the application of \eqref{factor-lemma-third}
to $\delta_{n,k,s}$ has the form
\begin{multline}
\left\{   \left[\varepsilon_j \cdot (x_1^{i_1} \cdots x_j^{i_j} \theta_1 \cdots \theta_a) \right] \otimes m \right\} 
 \odot  \delta_{n,k,s}  \\ =
\pm f_{a,b,\iii} \otimes (m \odot \delta_{n-j,k,s-b}) + \text{greater terms under $\triangleleft$}
\end{multline}
which is what we wanted to show.
\end{proof}

We give an example
to illustrate the statement and proof of Lemma~\ref{leading-degree-observation} 
in the presence of $\theta$-variables.

\begin{example}
Suppose $(n,k,s) = (7,7,4)$ so the superspace Vandermonde is
\begin{equation*}
\delta_{7,7,4} = \varepsilon_7  \cdot (x^{6663210} 
\theta_{123} ) \in \Omega_7.
\end{equation*}
We record the exponent sequence of the $x$-variables by
$(p_1, \dots, p_7) = (6,6,6,3,2,1,0)$ and that of the $\theta$-variables 
by $(q_1, \dots, q_7) = (1,1,1,0,0,0,0)$.

We take $j = 3$, so that we factor superspace elements according to 
$\Omega_7 = \Omega_j \otimes \Omega_{n-j} = \Omega_3 \otimes \Omega_4$.
For any system of right coset representatives of $\symm_3 \times \symm_4$ in $\symm_7$, the Vandermonde 
$\delta_{7,7,4}$ expands as 
\begin{small}
\begin{multline*}
\delta_{7,7,4} =  \\ \sum_{w} \pm \left[ \varepsilon_3 \cdot 
(x_{1}^{p_{w^{-1}(1)}} x_2^{p_{w^{-1}(2)}} x_{3}^{p_{w^{-1}(3)}}
\theta_{1}^{q_{w^{-1}(1)}} \theta_2^{q_{w^{-1}(2)}} \theta_{3}^{q_{w^{-1}(3)}}) \right]
\otimes 
\left[ \varepsilon_4 \cdot 
(x_4^{p_{w^{-1}(4)}} \cdots x_7^{p_{w^{-1}(7)}} \theta_4^{q_{w^{-1}(4)}} \cdots
 \theta_7^{q_{w^{-1}(7)}}) \right]
\end{multline*}
\end{small}
where $\pm$ is the sign of the coset representative $w$.

Let $m \in \Omega_4 \subseteq \Omega_3 \otimes \Omega_4$ be an arbitrary monomial
in $x_4, \dots, x_7, \theta_4, \dots, \theta_7$.
We consider applying the operator
\begin{equation*}
\left\{   \left[\varepsilon_3 \cdot (x^{325} \theta_1) \right] \otimes m \right\}  \odot ( - ) =
\left\{ \left[\varepsilon_3 \cdot (x^{325}  \theta_1) \right]  \odot (-) \right\} \otimes
\left\{ m \odot (-) \right\}
\end{equation*}
to $\delta_{7,7,4}$ by applying it to each term 
\begin{small}
\begin{equation*} \pm
\left[ \varepsilon_3 \cdot 
(x_{1}^{p_{w^{-1}(1)}} x_2^{p_{w^{-1}(2)}} x_{3}^{p_{w^{-1}(3)}}
\theta_{1}^{q_{w^{-1}(1)}} \theta_2^{q_{w^{-1}(2)}} \theta_{3}^{q_{w^{-1}(3)}}) \right]
\otimes 
\left[ \varepsilon_4 \cdot 
(x_4^{p_{w^{-1}(4)}} \cdots x_7^{p_{w^{-1}(7)}} \theta_4^{q_{w^{-1}(4)}} \cdots
 \theta_7^{q_{w^{-1}(7)}}) \right]
\end{equation*}
\end{small}
of its expansion in $\Omega_3 \otimes \Omega_4$. 
We choose our coset representatives $w$ so that  
\begin{center}
$q_{w^{-1}(2)} \leq q_{w^{-1}(3)}$ and 
$p_{w^{-1}(2)} \leq p_{w^{-1}(3)}$
\end{center} 
in each term.

Focusing on the first tensor factor, the evaluation
\begin{equation*}
\left[ \varepsilon_3 \cdot (x^{325} \theta_1) \right] \odot 
\left[ \varepsilon_3 \cdot 
(x_{1}^{p_{w^{-1}(1)}} x_2^{p_{w^{-1}(2)}} x_{3}^{p_{w^{-1}(3)}}
\theta_{1}^{q_{w^{-1}(1)}} \theta_2^{q_{w^{-1}(2)}} \theta_{3}^{q_{w^{-1}(3)}}) \right] 
\end{equation*}
is $\symm_3$-invariant, and is nonzero if and only if we have the componentwise inequalities
\begin{center}
$(3,2,5) \leq (p_{w^{-1}(1)}, p_{w^{-1}(2)}, p_{w^{-1}(3)})$ and 
$(1,0,0) \leq (q_{w^{-1}(1)}, q_{w^{-1}(2)}, q_{w^{-1}(3)})$.
\end{center}
To minimize under the order $\triangleleft$ subject
to these conditions, we take the unique coset representative $w$ for which
\begin{center}
$(p_{w^{-1}(1)}, p_{w^{-1}(2)}, p_{w^{-1}(3)}) = (6,3,6)$ and 
$(q_{w^{-1}(1)}, q_{w^{-1}(2)}, q_{w^{-1}(3)}) = (1,0,1)$.
\end{center}
For this choice of $w$, in the second tensor factor we have 
\begin{equation*}
\varepsilon_4 \cdot 
(x_4^{p_{w^{-1}(4)}} \cdots x_7^{p_{w^{-1}(7)}} \theta_4^{q_{w^{-1}(4)}} \cdots
 \theta_7^{q_{w^{-1}(7)}}) = \pm \varepsilon_4 \cdot (x_4^6 x_5^2 x_6^1 x_7^0 \theta_4) = \pm \delta_{4,7,3}
\end{equation*}
so that 
\begin{equation*}
m \odot \left[ \varepsilon_4 \cdot 
(x_4^{p_{w^{-1}(4)}} \cdots x_7^{p_{w^{-1}(7)}} \theta_4^{q_{w^{-1}(4)}} \cdots
 \theta_7^{q_{w^{-1}(7)}})   \right] = 
 \pm m \odot \delta_{4,7,3}.
\end{equation*}
This yields an expansion
\begin{small}
\begin{multline*}
\left\{   \left[\varepsilon_3 \cdot (x^{325} \theta_1) \right] \otimes m \right\}  \odot \delta_{7,7,4} = \\
\pm \left[ \varepsilon_3 \cdot (x^{325} \theta_1) \right] \odot 
\left[ \varepsilon_3 \cdot 
( x^{626}
\theta_{13} ) \right] 
 \otimes
[m \odot \delta_{4,7,3}] + \text{\rm{greater terms under $\triangleleft$}}
\end{multline*}
\end{small}
where the $\symm_3$-invariant
 in the first tensor 
factor is nonzero.
\end{example}

\subsection{A recursion for $\eta_j \HHH_{n,k,s}^{(r)}$}
We are ready to state our recursion for the graded $\symm_{n-j}$-module
$\eta_j \HHH_{n,k,s}^{(r)}$. Our crucial tools are Observation~\ref{parabolic-stability} and
Lemma~\ref{leading-degree-observation}.
Given any graded $\symm_{n-j}$-module $V = \bigoplus_{i} V_{i}$ and 
any integer $i_0$, let $V\{-i_0\}$ denote the same $\symm_{n-j}$ module
with degree shifted up by $i_0$. That is, the $i^{th}$ graded piece of $V\{-i_0\}$ is given by
\begin{equation}
( V\{-i_0\} )_{i} := V_{i-i_0}.
\end{equation}

\begin{lemma}
\label{h-module-recursion}
Let $n, k, s \geq 0$ with $k \geq s$ and let $0 \leq r \leq n-s$.  Let $1 \leq j \leq n$.  
There holds an isomorphism of graded $\symm_{n-j}$-modules
\begin{equation}
\label{h-recursion}
\eta_j \HHH_{n,k,s}^{(r)} \cong 
\bigoplus_{\substack{a,b \geq 0 \\ a \leq n-s-r \\ b \leq s}} \left[
\bigoplus_{\iii} \HHH_{n-j,k,s-b}^{(r - j + a + b)}
\left\{ i_1 + \cdots + i_j - {b \choose 2} - b \cdot (s-b) - (k-1) \cdot (j-b) \right\}
\right]
\end{equation}
where the inner direct sum is over all $j$-tuples $\iii = (i_1, \dots, i_j)$ of nonnegative integers such that 
\begin{equation*}
0 \leq i_{a+1} < \cdots < i_{a+b} < s \leq i_{a+b+1} < \cdots < i_j \leq k-1
\end{equation*}  
and 
\begin{equation*}
0 \leq i_1 \leq \cdots \leq i_a \leq k-s-1+b
\end{equation*}
Equivalently, we have the symmetric function identity
\begin{multline}
h_j^{\perp} \grFrob( \HHH_{n,k,s}^{(r)}; q) = \\
\sum_{\substack{a,b \geq 0 \\ a \leq n-s-r \\ b \leq s}} 
q^{ {j-a-b \choose 2} + (s-b)a} \times  
{k-s-1+a+b \brack a}_q \cdot {s \brack b}_q \cdot {k-s \brack j-a-b}_q \cdot
\grFrob(\HHH^{(r-j+a+b)}_{n-j,k,s-b}; q) 
\end{multline}
\end{lemma}

\begin{proof}
The basis $\CCC$ of $\eta_j \HHH_{n,k,s}^{(r)}$ in 
Lemma~\ref{harmonic-invariant-lemma} admits a disjoint union decomposition
\begin{equation}
\label{basic-c-basis}
\bigsqcup_{\substack{a,b \geq 0 \\ a \leq n-s-r \\ b \leq s}}  
\bigsqcup_{\iii}
\left\{
\left[
\varepsilon_j \cdot (x_1^{i_1} \cdots x_j^{i_j} \cdot \theta_1 \cdots \theta_a) \otimes m
\right]  \odot \delta_{n,k,s} \,:\,
m \in \MMM_{n-j,k,s-b}^{(n-s-r-a)} 
\right\}
\end{equation}
where the tuples $\iii = (i_1, \dots, i_j)$ satisfy the conditions in the statement of the theorem.
Lemma~\ref{leading-degree-observation} shows that, for each triple $(a,b,\iii)$ indexing the 
disjoint union \eqref{basic-c-basis}, there is a single nonzero invariant 
$f_{a,b,\iii} \in (\Omega_j)^{\symm_j}$ such that 
\begin{equation}
\label{key-expansion-equation}
\pm \left[
\varepsilon_j \cdot (x_1^{i_1} \cdots x_j^{i_j} \cdot \theta_1 \cdots \theta_a) \otimes m
\right]  \odot \delta_{n,k,s} =
f_{a,b,\iii} \otimes (m \odot \delta_{n-j,k,s-b}) + \text{ greater terms under $\triangleleft$}
\end{equation}
for all $m \in \MMM_{n-j,k,s-b}^{(n-s-r-a)}$.
The sign in Equation~\eqref{key-expansion-equation} may depend on $m$, but the invariant
$f_{a,b,\iii}$ does not.
The leading term $f_{a,b,\iii} \otimes (m \odot \delta_{n-j,k,s-b})$ in
\eqref{key-expansion-equation} lies in a single bihomogeneous piece 
\begin{equation}
\label{pq-definition}
(p, q) = (p(a,b,\iii), q(a,b,\iii))
\end{equation}
of the decomposition $\Omega_n = \bigoplus_{p,q \geq 0} \Omega_n(p,q)$ defining 
$\triangleleft$.

The observations of the previous paragraph give rise to a filtration of $\eta_j \HHH_{n,k,s}^{(r)}$.
For any $p, q \geq 0$, we set
\begin{multline}
\label{less-than-or-equal-to}
\left(
\eta_j \HHH_{n,k,s}^{(r)}
\right)_{\trianglelefteq (p,q)} := \\ \mathrm{span}  
\bigsqcup_{(p(a,b,\iii), q(a,b,\iii)) \trianglelefteq (p,q)}
\left\{
\left[
\varepsilon_j \cdot (x_1^{i_1} \cdots x_j^{i_j} \cdot \theta_1 \cdots \theta_a) \otimes m
\right]  \odot \delta_{n,k,s} \,:\,
m \in \MMM_{n-j,k,s-b}^{(n-s-r-a)} 
\right\}
\end{multline}
where $(p(a,b,\iii), q(a,b,\iii))$ is defined as in \eqref{pq-definition}. Analogously, we define
\begin{multline}
\label{less-than}
\left(
\eta_j \HHH_{n,k,s}^{(r)}
\right)_{\triangleleft (p,q)} := \\ \mathrm{span}  
\bigsqcup_{(p(a,b,\iii), q(a,b,\iii)) \triangleleft (p,q)}
\left\{
\left[
\varepsilon_j \cdot (x_1^{i_1} \cdots x_j^{i_j} \cdot \theta_1 \cdots \theta_a) \otimes m
\right]  \odot \delta_{n,k,s} \,:\,
m \in \MMM_{n-j,k,s-b}^{(n-s-r-a)} 
\right\}
\end{multline}
and the corresponding quotient space
\begin{equation}
\label{equal-to}
\left(
\eta_j \HHH_{n,k,s}^{(r)}
\right)_{= (p,q)}  := 
\left(
\eta_j \HHH_{n,k,s}^{(r)}
\right)_{\trianglelefteq (p,q)} / \left(
\eta_j \HHH_{n,k,s}^{(r)}
\right)_{\triangleleft (p,q)}.
\end{equation}

By Observation~\ref{parabolic-stability},  the decomposition
$\Omega_n = \bigoplus_{p,q \geq 0} \Omega_n(p,q)$  is 
$\symm_{n-j}$-stable. Therefore, the three vector spaces
\eqref{less-than-or-equal-to}, \eqref{less-than}, and \eqref{equal-to} are graded
 $\symm_{n-j}$-modules under  $x$-degree. Furthermore, we have an isomorphism
 \begin{equation}
 \label{filtration-isomorphism}
 \eta_j \HHH_{n,k,s}^{(r)} \cong 
 \bigoplus_{p,q \geq 0} 
 \left(
\eta_j \HHH_{n,k,s}^{(r)}
\right)_{= (p,q)}.
 \end{equation}
 What does the summand $\left(\eta_j \HHH_{n,k,s}^{(r)} \right)_{= (p,q)}$ in 
 \eqref{filtration-isomorphism} look like? The expansion \eqref{key-expansion-equation}
 implies an isomorphism of graded $\symm_{n-j}$-modules
 \begin{equation}
 \label{leading-isomophism}
 \left(
\eta_j \HHH_{n,k,s}^{(r)}
\right)_{= (p,q)} \cong 
\bigoplus_{(p(a,b,\iii), q(a,b,\iii) = (p,q)} \mathbb{V}_{a,b,\iii}
 \end{equation}
 where
\begin{equation}
\label{special-basis}  \mathbb{V}_{a,b,\iii} := 
\mathrm{span} 
\left\{
f_{a,b,\iii} \otimes (m \odot \delta_{n-j,k,s-b})   \,:\, 
m \in \MMM^{(n-s-r-a)}_{n-j,k,s-b} 
\right\}.
\end{equation}
In light of Corollary~\ref{harmonic-basis-corollary},
the space $\mathbb{V}_{a,b,\iii}$
affords a copy of the $\symm_{n-j}$-module
$\HHH_{n-j,k,s-b}^{(r-j+a+b)}$ with degree shifted up by
\begin{equation}
\text{$x$-degree of $f_{a,b,\iii}$}
 = {b \choose 2} + b \cdot (s-b) + (k-1) \cdot (j-b) - i_1 - \cdots - i_j.
\end{equation}
This completes the proof.
\end{proof}

\subsection{The $\symm_n$-structure of $\WWW_{n,k,s}$}
All of the pieces are in place for us to prove our combinatorial formula for 
$\grFrob(\WWW_{n,k,s};q,z)$.

\begin{theorem}
\label{graded-module-structure}
Let $n, k \geq s \geq 0$ be integers. The bigraded Frobenius image
$\grFrob(\WWW_{n,k,s};q,z)$ has the following combinatorial expressions in terms of the 
statistics $\coinv$ and $\codinv$.
\begin{equation}
\grFrob(\WWW_{n,k,s};q,z) = C_{n,k,s}(\xx;q,z) = D_{n,k,s}(\xx;q,z).
\end{equation}
\end{theorem}

\begin{proof}
It suffices to show
\begin{equation}
\grFrob(\HHH_{n,k,s}^{(r)};q) = C_{n,k,s}^{(r)}(\xx;q)
\end{equation}
for all $r$. Lemmas \ref{skew-by-e-lemma} and \ref{module-skew-by-e} reduce this task to proving 
\begin{equation}
\label{eta-to-h-goal}
\grFrob(\eta_j \HHH_{n,k,s}^{(r)};q) = h_j^{\perp} C_{n,k,s}^{(r)}(\xx;q)
\end{equation}
for any $j \geq 1$. By Lemmas \ref{c-skewing-lemma} and \ref{h-module-recursion}, both
sides of Equation~\eqref{eta-to-h-goal} satisfy the same recursion and we are done by induction on $n$.
\end{proof}

As an example of Theorem~\ref{graded-module-structure}, let $(n,k,s) = (3,2,2)$. We list the 12 ordered set
superpartitions $\sigma \in \OSP_{3,2,2}$ together with their coinversion numbers, reading words, 
and inverse descent sets.
\begin{center}
\begin{tabular}{c | c | c | c}
$\sigma$ & $\coinv(\sigma)$ & $\mathrm{read}(\sigma)$ & $\iDes(\mathrm{read}(\sigma))$ \\ \hline 
$(1, \, 2 \mid 3)$ & $2$ & $312$ & $\{2\}$ \\
$(1, \, 3 \mid 2)$ & $1$ & $213$ & $\{1\}$ \\
$(2, \, 3 \mid 1)$ & $0$ & $123$ & $\varnothing$ \\
$(1 \mid 2, \, 3)$ & $2$ & $213$ & $\{1\}$ \\
$(2 \mid 1, \, 3)$ & $1$ & $123$ & $\varnothing$ \\
$(3 \mid 1, \, 2)$ & $1$ & $132$ & $\{2\}$ \\
\end{tabular}  \quad 
\begin{tabular}{c | c | c | c}
$\sigma$ & $\coinv(\sigma)$ & $\mathrm{read}(\sigma)$ & $\iDes(\mathrm{read}(\sigma))$ \\ \hline 
$(1, \, \bar{2} \mid 3)$ & $1$ & $231$ & $\{1\}$ \\
$(1, \, \bar{3} \mid 2)$ & $1$ & $321$ & $\{1,2\}$ \\
$(2, \, \bar{3} \mid 1)$ & $0$ & $312$ & $\{2\}$ \\
$(1 \mid 2, \, \bar{3})$ & $2$ & $321$ & $\{1,2\}$ \\
$(2 \mid 1, \, \bar{3})$ & $1$ & $312$ & $\{2\}$ \\
$(3 \mid 1, \, \bar{2})$ & $0$ & $213$ & $\{1\}$ \\
\end{tabular} 
\end{center}
This leads to the expression
\begin{equation*}
F_{\varnothing,3} + q \cdot (F_{\varnothing,3} + F_{1,3} + F_{2,3}) + 
q^2 \cdot (F_{1,3} + F_{2,3}) + z \cdot (F_{1,3} + F_{2,3}) + qz \cdot (F_{1,3} + F_{2,3} + F_{12,3}) +
q^2 z \cdot (F_{12,3})
\end{equation*}
for $\grFrob(\WWW_{3,2,2};q,z)$ which has Schur expansion
\begin{equation*}
s_{3} + q \cdot (s_3 + s_{2,1}) + q^2 \cdot s_{2,1} + z \cdot s_{2,1} + qz \cdot (s_{2,1} + s_{1,1,1}) + 
q^2 z \cdot s_{1,1,1}.
\end{equation*}
As with the bigraded Hilbert series, the bigraded Frobenius image is more attractive in matrix format
\begin{equation*}
\grFrob(\WWW_{3,2,2};q,z) = \begin{pmatrix}
s_3 & s_3 + s_{2,1} & s_{2,1} \\ s_{2,1} & s_{2,1} + s_{1,1,1} & s_{1,1,1}
\end{pmatrix}
\end{equation*}
where the Rotational Duality of Theorem~\ref{previous-w-knowledge} becomes apparent.

\begin{remark}
In the case $k = s$ the family $\OSP_{n,k,k}$ admits an involution $\iota$
sending an ordered set superpartition $\sigma = (B_1 \mid \cdots \mid B_k) \in \OSP_{n,k,k}$ to 
$\iota(\sigma) = (\overline{B_k} \mid \cdots \mid \overline{B_1})$,
where $\overline{B_i}$ is obtained from $B_i$ by switching the barred/unbarred status of every non-minimal
element of $B_i$. 
It may be checked that $\iota$ complements both the statistic $\coinv$ and
the subset $\iDes$ in the sense that 
\begin{equation*}
\coinv(\sigma) + \coinv(\iota(\sigma)) = {k \choose 2} + (n-k) \cdot (k-1) \quad \text{and} \quad
\iDes(\mathrm{read}(\sigma)) \sqcup \iDes(\mathrm{read}(\iota(\sigma))) = [n-1]
\end{equation*}
for any $\sigma \in \OSP_{n,k,k}$.
The Rotational Duality statement of 
Theorem~\ref{previous-w-knowledge} is therefore consistent with
Theorem~\ref{graded-module-structure}.
\end{remark}

The authors do not know a combinatorial formula for the Schur expansion of 
$\grFrob(\WWW_{n,k,s};q,z)$. On the other hand, if we consider $\WWW_{n,k,s}$ as a singly
graded $\symm_n$-module under $\theta$-degree, we have a simple formula for this Schur 
expansion.
Recall that a partition $\lambda$ is a {\em hook} if it has the form $\lambda = (a,1^{m-a})$ for some $a$.
The anticommutative graded pieces of $\WWW_{n,k,s}$ are built out of hook shapes.

\begin{corollary}
\label{hook-corollary}
Consider $\WWW_{n,k,s}$ as a singly graded $\symm_n$-module under  $\theta$-degree.
Then
\begin{equation}
\grFrob(\WWW_{n,k,s}; z) = \sum_{(\lambda^{(1)}, \dots, \lambda^{(k)})}
z^{n - \lambda^{(1)}_1 - \cdots - \lambda^{(k)}_1} \cdot 
s_{\lambda^{(1)}} \cdots s_{\lambda^{(s)}} \cdot 
s_{\lambda^{(s+1)}/(1)} \cdots s_{\lambda^{(k)}/(1)}
\end{equation}
where $(\lambda^{(1)}, \dots, \lambda^{(k)})$ ranges over $k$-tuples of nonempty hooks
with a total of
\begin{equation*}
 |\lambda^{(1)}| + \cdots + |\lambda^{(k)}| = n + k - s
 \end{equation*}
  boxes.
\end{corollary}

Observe that the skew partitions $\lambda^{(s+1)}/(1), \dots, \lambda^{(k)}/(1)$ indexing the 
last $k-s$ factors in Corollary~\ref{hook-corollary} can be empty and we have the factorization
$s_{(a,1^{m-a})/(1)} = h_{a-1} \cdot e_{m-a}$.
When $(n,k,s) = (4,3,2)$, Corollary~\ref{hook-corollary} says that the piece of 
$\WWW_{4,3,2}$ of anticommuting degree 1 has Frobenius image
\begin{equation*}
s_{2,1} \cdot s_1 + s_{1,1} \cdot s_2 + s_2 \cdot s_{1,1} + s_1 \cdot s_{2,1} + 
s_{1,1} \cdot s_1 \cdot h_1 + s_1 \cdot s_{1,1} \cdot h_1 + s_{2} \cdot s_1 \cdot e_1 +
s_1 \cdot s_2 \cdot e_1 + s_1 \cdot s_1 \cdot (h_1 e_1).
\end{equation*}

\begin{proof}
We have a natural bijection between column tableaux and tuples of hook shaped semistandard tableaux,
where a $\bullet$ at height zero corresponds to a skewed-out box, viz.
\begin{equation*}
\begin{Young}
  \bar{4} & \bar{6} & ,& \bar{4} &, \cr
  \bar{3} & \bar{5} &,  &  \bar{1} &, \cr
  2 & 3 & \bullet & \bullet & \bullet \cr
  ,& 3 &, & 7 & 1 \cr
   ,& 4 & ,& ,& 1 \cr
\end{Young}
\quad
\begin{Young}
, \cr
,\cr 
,\Leftrightarrow  \cr
, \cr
, \cr
\end{Young}
\quad
\begin{Young}
, \cr
2 \cr
3 \cr
4 
\end{Young} \, ,  \, \, \,
\begin{Young}
, \cr
3 & 3 & 4 \cr
5 \cr
6
\end{Young} \, ,  \, \, \,
\begin{Young}
, \cr
\bullet \cr
, \cr
, \cr
\end{Young} \, , \, \, \,
\begin{Young}
, \cr
\bullet & 7 \cr
1 \cr
4
\end{Young} \, ,  \, \, \,
\begin{Young}
, \cr
\bullet & 1 & 1 \cr
, \cr
, \cr
\end{Young}
\end{equation*}
Now apply Theorem~\ref{graded-module-structure}.
\end{proof}

\section{Conclusion}
\label{Conclusion}

In this paper, we gave a combinatorial formula (Theorem~\ref{graded-module-structure})
for the bigraded Frobenius image
$\grFrob(\WWW_{n,k,s};q,z)$ of a family of quotients $\WWW_{n,k,s}$ 
of the superspace ring $\Omega_n$.
The following problem remains open.

\begin{problem}
\label{schur-expansion}
Find the Schur expansion of $\grFrob(\WWW_{n,k,s};q,z)$.
\end{problem}

A solution to Problem~\ref{schur-expansion} would refine 
Corollary~\ref{hook-corollary}.
We remark that symmetric functions admit the operation of {\em superization} which has the 
plethystic definition 
$f[\xx] \mapsto f[\xx - \mathbf{y}]$, where $\xx$ and $\mathbf{y}$ are two infinite alphabets;
see \cite{NT} for more details.
Although superization can be used to yield (for example) the bigraded Frobenius image of 
$\Omega_n$ from that of $\QQ[x_1, \dots, x_n]$, we do not know of a way to use superization
to solve Problem~\ref{schur-expansion}.

One possible way to solve Problem~\ref{schur-expansion} would be to define 
a statistic on ordered set superpartitions which extends the {\em major index} 
statistic on permutations. In the context of ordered set partitions, 
two such extensions are available:
$\maj$ and $\mathrm{minimaj}$ \cite{BCHOPSY, HRW, RemmelWilson, Rhoades}.
If such an extension were found, an insertion argument of Wilson \cite{WMultiset}
(in the case of $\maj$)
or the crystal techniques of Benkart,
Colmenarejo, Harris, Orellana, Panova, Schilling, and Yip \cite{BCHOPSY} 
(in the case of $\mathrm{minimaj}$)
could perhaps be used 
to solve Problem~\ref{schur-expansion}.

The substaircase basis of $\WWW_{n,k,s}$ of Theorem~\ref{M-is-basis}
is tied to the inversion-like statistics $\coinv$ and $\codinv$ on ordered set superpartitions.
Steinberg proved \cite{Steinberg}  that the set of {\em descent monomials} 
$\{ d_w \,:\, w \in \symm_n \}$ form a basis for the classical coinvariant ring
$\QQ[x_1, \dots, x_n]/\langle e_1, \dots, e_n \rangle$ where
\begin{equation}
d_w := \prod_{\substack{1 \leq i \leq n-1 \\ w(i) > w(i+1)}} 
x_{w(1)} \cdots x_{w(i)}.
\end{equation}
This basis was studied further by Garsia and Stanton \cite{GS} in the context of 
Stanley-Reisner theory and extended to the context of $R_{n,k}$
by Haglund, Rhoades, and Shimozono \cite{HRS}.

\begin{problem}
\label{descent-monomial-problem}
Define an extension of the major index statistic to ordered set superpartitions in 
$\OSP_{n,k,s}$ and a companion 
descent monomial basis of $\WWW_{n,k,s}$.
\end{problem}

In the context of $R_{n,k}$, Meyer used \cite{Meyer} the descent monomial basis 
to refine the formulas for $\grFrob(R_{n,k};q)$ of Haglund, Rhoades, and Shimozono \cite{HRS}.
This refinement matched a symmetric function arising from the crystal-theoretic machinery
of Benkart et.\ al.\ \cite{BCHOPSY}.
A solution to Problem~\ref{descent-monomial-problem} might lead to similar refinements in the 
$\WWW$-module context.

There are several results about the anticommuting degree zero piece of $\WWW_{n,k}$ one could try to push to the entire module $\WWW_{n,k}$. We mention a couple briefly. One could look for a Hecke action on all of superspace, extending work of Huang, Rhoades, and Scrimshaw \cite{HuangRhoadesScrimshaw}. One could also hope to generalize 
the geometric discoveries of Pawlowski and Rhoades to the entire module $\WWW_{n,k}$ \cite{PR}. In particular, it would be interesting to see if this larger module allows for the definition of a Schubert basis with nonnegative structure constants, eliminating the negative constants that can appear in \cite{PR}. 

We close with a discussion of the superspace coinvariant problem \cite{Zabrocki} which was a key motivation
of this work. As mentioned in the introduction, it is equivalent to the following.

\begin{conjecture}
\label{superspace-coinvariant-problem}
Let $\langle (\Omega_n)^{\symm_n}_+ \rangle \subseteq \Omega_n$ be the ideal generated by $\symm_n$-invariants with vanishing constant term.
For any $k$ we have
\begin{equation}
\{ z^{n-k} \}  \, \grFrob( \Omega_n/ \langle (\Omega_n)^{\symm_n}_+ \rangle; q, z) =
 \{ z^{n-k} \}  \, \grFrob(\WWW_{n,k}; q, z)
\end{equation}
where $q$ tracks  $x$-degree and $z$ tracks $\theta$-degree.
\end{conjecture}

Let $k \leq n$ be positive integers and let $\Omega_n^{(n-k)}$ be the piece of $\Omega_n$
of homogeneous $\theta$-degree $n-k$, viewed as 
a free $\QQ[x_1, \dots, x_n]$-module of rank ${n \choose k}$.
We define two submodules $A_{n,k}, B_{n,k} \subseteq \Omega^{(n-k)}$ in terms
of generating sets as follows.  Let $\mathcal{E} := \sum_{i = 1}^n \frac{\partial}{\partial x_i} \theta_i$
be the Euler operator.
\begin{equation}
A_{n,k} := \langle e_j \cdot \theta_T \,:\, j \geq 1, \, \, T \subseteq [n], \, \, |T| = n - k \rangle +
\langle \mathcal{E}(p_j) \cdot \theta_{T'} \,:\, j \geq 1, \, \, T' \subseteq [n], \, \, |T'| = n-k-1 \rangle
\end{equation}
\begin{multline}
B_{n,k} :=  \langle e_j(S) \cdot \theta_T \,:\, j \geq 1, \, \, S \sqcup T = [n], \, \, |T| = n - k \rangle +  \\
\langle \mathcal{E}(p_j) \cdot \theta_{T'} \,:\, j \geq 1, \, \, T' \subseteq [n], \, \, |T'| = n-k-1 \rangle + 
\langle x_i^k \cdot \theta_T \,:\, T \subseteq [n], \, \, |T| = n-k \rangle
\end{multline}

Both of the submodules $A_{n,k}, B_{n,k} \subseteq \Omega_n^{(n-k)}$ are 
homogeneous under the $x$-grading and stable under the $\symm_n$-action.
While their generating sets are similar, they differ as follows:
\begin{itemize}
\item The generators $e_j \cdot \theta_T$ of $A_{n,k}$ involve elementary symmetric polynomials
in the full variable set $\{x_1, \dots, x_n\}$ whereas the corresponding generators $e_j(S) \cdot \theta_T$
of $B_{n,k}$ involve a `separation of variables' with $S \sqcup T = [n]$.
\item The submodule $B_{n,k}$ has generators of the form $x_i^k \cdot \theta_T$ which are not 
present in $A_{n,k}$.
\end{itemize}

\begin{proposition}
\label{ideal-comparison}
Let $n \geq k$ be positive integers and let 
$\langle ( \Omega_n)^{\symm_n}_+ \rangle \subseteq \Omega_n$ be the ideal generated by 
$\symm_n$-invariants with vanishing constant term.  Also let 
$I_{n,k} := \ann \, \delta_{n,k} \subseteq \Omega_n$ be the defining ideal of $\WWW_{n,k}$.
We have
\begin{equation}
\langle ( \Omega_n)^{\symm_n}_+ \rangle \cap \Omega_n^{(n-k)} = A_{n,k} \quad \quad
I_{n,k} \cap \Omega_n^{(n-k)} = B_{n,k}.
\end{equation}
Therefore, Conjecture~\ref{superspace-coinvariant-problem} is equivalent to 
the isomorphism of graded $\symm_n$-modules
\begin{equation}
\Omega^{(n-k)}_n/A_{n,k} \cong \Omega^{(n-k)}_n/B_{n,k}.
\end{equation}
\end{proposition}

\begin{proof}
The equality $\langle ( \Omega_n)^{\symm_n}_+ \rangle \cap \Omega_n^{(n-k)} = A_{n,k}$ 
follows from a beautiful result of Solomon \cite{Solomon}: 
the ideal $\langle ( \Omega_n)^{\symm_n}_+ \rangle \subseteq \Omega_n$ is generated
by $e_1, e_2, \dots, e_n$ together with 
$\mathcal{E}(p_1), \mathcal{E}(p_2), \dots, \mathcal{E}(p_n)$.
In fact, either list of polynomials $e_1, \dots, e_n$ and $p_1, \dots, p_n$ could be replaced by 
any algebraically independent set of generators of the ring 
$\QQ[x_1, \dots, x_n]^{\symm_n}$ of symmetric polynomials and this assertion would remain true.

The second equality $I_{n,k} \cap \Omega_n^{(n-k)} = B_{n,k}$ may be deduced as follows.
From the definition of $I_{n,k}$ we see that $I_{n,k} \cap \Omega_n^{(n-k)}$ is generated by $B_{n,k}$
together with generators of the form
\begin{equation}
e_j(S) \cdot \theta_T \,:\, S \cup T = [n], \, \, |T| = n-k, \, \, j > |S| - k
\end{equation}
where the union $S \cup T = [n]$ need not be disjoint. 
If we consider such a generator $e_j(S) \cdot \theta_T$, we may write the set $S$ as a disjoint 
union $S = S_1 \sqcup S_2$ where $S_1 = S - T$ and $S_2 = S \cap T$.
The identity
\begin{equation}
e_j(S) \cdot \theta_T  = e_j(S_1 \sqcup S_2) \cdot \theta_T = 
\sum_{a + b = j} e_a(S_1) \cdot e_b(S_2) \cdot \theta_T
\end{equation}
and the assumption $j > |S| - k = |S_1| + |S_2| - k$ imply that all nonvanishing terms 
$e_a(S_1) \cdot e_b(S_2) \cdot \theta_T$ in this sum satisfy $a > 0$ and are therefore
members of $B_{n,k}$.
\end{proof}

Proposition~\ref{ideal-comparison} gives a side-by-side comparison of the 
$\symm_n$-modules on either side of the conjecture
\eqref{true-fields-conjecture} as explicit quotients of the same free
$\QQ[x_1, \dots, x_n]$-module.
The authors hope that this will help shed light on the ring of superspace coinvariants.

\section{Acknowledgements}

B. Rhoades was partially supported by NSF grant 
DMS-1953781.
The authors are grateful to Tianyi Yu for helpful conversations.

\end{document}